\documentclass[3p,11pt]{elsarticle}
\usepackage{url}
\usepackage{amssymb}
\usepackage{amsmath}
\usepackage{amsthm}
\usepackage{stmaryrd}
\usepackage{siunitx}
\usepackage{commath}
\usepackage{algorithm}
\usepackage{algorithmic}
\usepackage{subfig}
\usepackage{enumerate}
\usepackage{amsmath}
\usepackage{multirow}
\usepackage{makecell}
\usepackage{lipsum}
\usepackage{amsfonts}
\usepackage{graphicx}
\usepackage{epstopdf}
\usepackage{algorithm}
\usepackage{algorithmic}
\usepackage{amssymb}
\usepackage{bm}
\usepackage{color}
\usepackage{mathrsfs}
\usepackage{caption}
\usepackage{subfig}

\usepackage{amsmath}
\usepackage{multirow}
\usepackage{makecell}
\usepackage{lipsum}
\usepackage{amsfonts}
\usepackage{graphicx}
\usepackage{epstopdf}
\usepackage{algorithm}
\usepackage{algorithmic}
\usepackage{amssymb}
\usepackage{bm}
\usepackage{color}
\usepackage{mathrsfs}
\usepackage{caption}
\usepackage{subfig}

\usepackage[hidelinks]{hyperref}
\hypersetup{
colorlinks   = true, %Colours links instead of ugly boxes
urlcolor     = red, %Colour for external hyperlinks
linkcolor    = red, %Colour of internal links
citecolor   = red %Colour of citations
}

\biboptions{sort&compress}
\usepackage[misc]{ifsym}
\usepackage{color}
\usepackage{cleveref}

\theoremstyle{definition}

\newtheorem{remark}{Remark}[section]

\makeatletter
\def\ps@pprintTitle{%
\let\@oddhead\@empty
\let\@evenhead\@empty
\def\@oddfoot{}%
\let\@evenfoot\@oddfoot}

\newcommand{\tnorm}{\@ifstar\@tnorms\@tnorm}
\newcommand{\@tnorms}[1]{%
\left|\mkern-1.5mu\left|\mkern-1.5mu\left|
#1
\right|\mkern-1.5mu\right|\mkern-1.5mu\right|
}
\newcommand{\@tnorm}[2][]{%
\mathopen{#1|\mkern-1.5mu#1|\mkern-1.5mu#1|}
#2
\mathclose{#1|\mkern-1.5mu#1|\mkern-1.5mu#1|}
}

\renewenvironment{proof}[1][\proofname]{\par
\vspace{-\topsep}% remove the space after the theorem
\pushQED{\qed}%
\normalfont
\topsep0pt\partopsep0pt % no space before
\trivlist
\item[\hskip\labelsep
\itshape
#1\@addpunct{.}]\ignorespaces
}{%
\popQED\endtrivlist\@endpefalse
\addvspace{6pt plus 6pt} % some space after
}
\makeatother

\newtheoremstyle{mystyle}
{3pt} % Space above
{1pt} % Space below
{} % Body font
{} % Indent amount
{\bfseries} % Theorem head font
{.} % Punctuation after theorem head
{.5em} % Space after theorem head
{} % Theorem head spec (can be left empty, meaning `normal')

\theoremstyle{mystyle} 
\newtheorem{theorem}{Theorem}[section]
\newtheorem{lemma}{Lemma}[section]

\allowdisplaybreaks[4]
\DeclareMathOperator{\diag}{diag}

\newtheorem{assumption}[theorem]{Assumption}
\theoremstyle{definition}

\theoremstyle{remark}

\numberwithin{equation}{section}
\usepackage{float}
\usepackage{comment}

\renewcommand{\theequation}{\arabic{section}.\arabic{equation}}
\makeatletter

\@addtoreset{equation}{section}

\makeatother

\begin{document}
\abovedisplayskip=4.2pt plus 4.2pt
\abovedisplayshortskip=0pt plus 4.2pt
\belowdisplayskip=4.2pt plus 4.2pt
\belowdisplayshortskip=2.1pt plus 4.2pt
%------------------------------------------------------------------------------
\begin{frontmatter}
\title{Non-centered parametric variational Bayes' approach for hierarchical inverse problems of partial differential equations
}
%\titlerunning{Short form of title}        % if too long for running head

\author[1]{Jiaming Sui}
%		\cortext[first]{Corresponding author: Jiaming Sui (sjming1997327@stu.xjtu.edu.cn)}

\author[1]{Junxiong Jia  \corref{first}}
\cortext[first]{Corresponding author: Junxiong Jia (jjx323@xjtu.edu.cn)}

\address[1]{School of Mathematics and Statistics, Xi'an Jiaotong University,
	Xi'an, Shaanxi 710049, China}
%		\address[2]{School of Mathematics and Statistics, Xi'an Jiaotong University,
%			Xi'an, Shaanxi 710049, China}

% The correct dates will be entered by the editor
\begin{abstract}
This paper proposes a non-centered parameterization based infinite-dimensional mean-field variational inference (NCP-iMFVI) approach for solving the hierarchical Bayesian inverse problems.	This method can generate available estimates from the approximated posterior distribution efficiently. To avoid the mutually singular obstacle that occurred in the infinite-dimensional hierarchical approach, we propose a rigorous theory of the non-centered variational Bayesian approach.	Since the non-centered parameterization weakens the connection between the parameter and the hyper-parameter, we can introduce the hyper-parameter to all terms of the eigendecomposition of the prior covariance operator. We also show the relationships between the NCP-iMFVI and infinite-dimensional hierarchical approaches with centered parameterization. The proposed algorithm is applied to three inverse problems governed by the simple smooth equation, the Helmholtz equation, and the steady-state Darcy flow equation. Numerical results confirm our theoretical findings, illustrate the efficiency of solving the iMFVI problem formulated by large-scale linear and nonlinear statistical inverse problems, and verify the mesh-independent property.
\end{abstract}

%------------------------------------------------------------------------------
\end{frontmatter}
\section{Introduction}

Due to the multidisciplinary cross-application in seismic exploration \cite{weglein2003inverse}, medical imaging \cite{zhou2020bayesian} and so on, the inverse problems of partial differential equations (PDEs) have undergone an enormous development over the past few decades \cite{arridge2019solving}. 
For solving the inverse problems of PDEs, uncertainties are ubiquitous, e.g., measurement uncertainty and epistemic uncertainty. 
The Bayesian inverse approach provides a flexible framework that solves inverse problems by transforming them into statistical inference problems, thereby making it feasible to analyze the uncertainty of the solutions to the inverse problems. 

Inverse problems of PDEs have been defined on some infinite-dimensional spaces \cite{Kirsch2011Book} generally, which are not compatible with the well-studied finite-dimensional Bayesian inference approach \cite{bishop2006pattern,kaipio2006statistical}.
To overcome this obstacle, we usually employ these two main approaches:
\begin{itemize}
 \item Discretize-then-Bayesianize: 
 The PDEs are initially discretized to approximate the original problem in some finite-dimensional space, and the reduced
 approximated problem is then solved by using the Bayes' method \cite{kaipio2006statistical}. 
 \item Bayesianize-then-discretize:
 The Bayes' formula and algorithms are initially constructed on infinite-dimensional space, and after the infinite-dimensional
 algorithm is built, some finite-dimensional approximation is carried out \cite{dashti2013bayesian}.
\end{itemize}

The above two approaches both have advantages and disadvantages.
The first approach enables us to employ all of the Bayesian inference methods developed in the statistical literature \cite{kaipio2006statistical} to solve the inverse problems.
However, given that the original problems are defined on the infinite-dimensional spaces, some problems have arisen, e.g., descending convergence rate and mesh dependence, which generates an inevitable barrier for employing this approach.\ Concerned with the \emph{Bayesianize-then-discretize} approach, it has the following advantages: 

\begin{itemize}
 	\item A better understanding of the function space structures will be of importance for designing numerical schemes of PDEs, especially when the gradient information is employed \cite{hinze2008optimization}. 
	\item Formulating infinite-dimensional theory rigorously can avoid inappropriate intuitive ideas from the finite-dimensional inverse approach, e.g., total variation prior \cite{Lassas2004IP}. 
	\item Pushing the discretization to the last step usually generate \emph{mesh independence} algorithms, which means that  
	the sampling efficiency will not highly depend on the dimension of the discretization \cite{petra2014computational}. 
\end{itemize}
Based on these advantages, \emph{Bayesianize-then-discretize} approach has attracted numerous researchers' attention in recent years \cite{bui2013computational,Cotter2009IP,Stuart2010ActaNumerica}. 

One critical issue for applying the Bayesian approach is efficiently extracting posterior information. 
Based on the \emph{Bayesianize-then-discretize} perspective, an infinite-dimensional Markov chain Monte Carlo (MCMC) type algorithm named preconditioned Crank-Nicolson (pCN) has been proposed and analyzed in detail \cite{cotter2013,Pillai2014SPDE}, which has consistent sampling efficiency under different discretizations. 
Besides the pCN algorithm, other types of infinite-dimensional sampling algorithms have been proposed, e.g., 
infinite-dimensional sequential Monte Carlo algorithm \cite{Beskos2015SC} and 
infinite-dimensional importance sampling algorithm \cite{Agapiou2017SS}.
In order to enhance the sampling efficiency, infinite-dimensional MCMC type sampling algorithms 
with gradient and geometric informative proposals have been constructed, e.g., infinite-dimensional Metropolis-adjusted Langevin algorithm \cite{Thanh2016IPI}, and geometric pCN algorithm \cite{Beskos2017JCP}. 
Although these algorithms have mesh independence property, they can hardly be employed to solve large-scale inverse problems of PDEs such as full waveform inversion \cite{Fichtner2011Book}. 

Variational inference (VI), an efficient approximated sampling method, has been widely investigated in machine learning for quantifying the uncertainties of learning models \cite{blei2017variational,zhang2018advances}.
Various approximate probability measures have been frequently used for training deep neural networks in finite-dimensional spaces.
Thus VI methods are usually constructed for finite-dimensional problems.
Some studies on VI methods for inverse problems of PDEs are based on the \emph{discretize-then-Bayesianize} perspective on finite-dimensional spaces.
For example, a mean-field assumption based VI approach was employed to solve finite-dimensional inverse problems with hyper-parameters in prior and noise distributions \cite{Guha2015JCP, Jin2012JCP, jin2010hierarchical}.
Projected Stein variational gradient descent methods were constructed to solve inverse problems with low-intrinsic dimensions \cite{Chen2021SISC, Chen2019NIPS}. 

However, under the infinite-dimensional settings, the VI methods are much less studied for solving inverse problems of PDEs from the \emph{Bayesianize-then-discretize} perspective.
Specifically, when the approximated measures are restricted to be Gaussian, a novel Robbins-Monro algorithm was developed in \cite{pinski2015algorithms,Pinski2015SIAMMA} from the calculus-of-variations viewpoint.
In order to employ the non-Gaussian approximated measures, 
a general \textbf{i}nfinite-dimensional \textbf{m}ean-\textbf{f}ield \textbf{VI} (iMFVI) theoretical framework was established in \cite{jia2021variational}.
Based on this general theory, a generative deep neural network model named VINet was constructed in \cite{Jia2022VINet} by analyzing the approximated measures. 
The infinite-dimensional Stein variational gradient descent approach was proposed to solve the nonlinear inverse problems in \cite{jia2021stein}. 

Besides the problem of efficient sampling, another critical issue for applying the Bayesian approach is that we sometimes can hardly provide an appropriate prior measure. 
Especially we can hardly specify the scales of the variance, which corresponds to the regularization parameter in the classical regularization methods \cite{Agapiou2013SPTA,Dashti2013IP}.
The hierarchical Bayesian inference approach has been widely adopted to specify the hyper-parameters, e.g., the scales of the covariance operator.
In the finite-dimensional space, a good introduction is provided in Chapter 3 of \cite{kaipio2006statistical}.
A hierarchical Bayesian model, providing an efficient iterative algorithm for calculating the maximum a posterior estimates, has been constructed in \cite{calvetti2009conditionally, Calvetti2008hypermodels}, with the probability density of the hyper-parameter being Gamma distribution.
Furthermore, in order to extract the posterior information based on the hierarchical model, the Gibbs sampling method has been studied in \cite{papaspiliopoulos2008stability, papaspiliopoulos2003non, papaspiliopoulos2007general}, which provides an efficient sampling method, and introduces the non-centered parameterization.

In this work, we focus on a similar hierarchical Bayesian model defined on infinite-dimensional space that has been investigated in \cite{agapiou2014analysis,chen2018dimension,dunlop2017hierarchical} under the framework of infinite-dimensional MCMC sampling methods. 
As indicated in \cite{agapiou2014analysis,chen2018dimension}, how to introduce hyper-parameters in the prior measure in the infinite-dimensional case is very different from the finite-dimensional case since the infinite-dimensional Gaussian measures tend to be singular with each other (see also Chapter 2 of \cite{Prato2014book} for details). 
In order to overcome the singular issue of infinite-dimensional Gaussian measures, the prior covariance operator $\mathcal{C}_0$ has been represented by its eigendecomposition in \cite{jia2021variational}, and the hyper-parameter is only introduced to a finite number of eigenbasis.
As stated in Subsection 3.1 of \cite{jia2021variational}, this strategy will lead to an inappropriate Bayesian model that can not adequately incorporate information encoded in data. 
Recently, a detailed analysis has been given in \cite{Dunlop2020SMAIJCM}, which indicates that centered parametrization is suitable for evaluating maximum a posterior estimate, but non-centered parameterization is more appropriate for using the Gibbs sampling method. 
To the best of our knowledge, there are few studies of the infinite-dimensional hierarchical Bayesian models under the infinite-dimensional VI framework that can overcome the singular issue without the truncation of eigenexpansions. 
Hence, we intend to formulate the iMFVI method with a non-centered parametrization (NCP) approach, which yields a new VI method that allows us to introduce the hyper-parameters for the whole eigenexpansion of the prior covariance operator. 
The new VI method has been constructed rigorously (see Subsections \ref{subsec2.2}, and \ref{subsec2.3}) and the relations with the centered parametrization based method are discussed in detail (see Subsections \ref{subsec2.4}). 
Finally, we provide numerical discretization strategies based on the low-rank structure of the posterior measure (see Subsections \ref{subsec2.5} and \ref{subsec2.6}). 
In summary, this work mainly contains four contributions:

\begin{itemize}
	\item For the hierarchical Bayesian model, a \textbf{n}on-\textbf{c}entered \textbf{p}arameterization based iMFVI (NCP-iMFVI) approach is proposed.
	Compared with the iMFVI approach proposed in Section 3.1 of \cite{jia2021variational}, we can introduce the hyper-parameter for the whole covariance operator of the prior measure rather than for the truncated finite number of components based on the eigenbasis.  
	\item Different from the iMFVI method proposed in \cite{jia2021variational}, we transform the unknown parameter into a new one by taking the non-centered parameterization and formulating a new VI problem instead of the original VI problem.
	We carry out detailed discussions about the relationships between these two problems, which yield compelling reasons for employing the NCP formulation to solve the hierarchical inference problem in the infinite-dimensional space.
	\item Through a detailed structural analysis of the posterior measure of the hyper-parameter, we transform the complicated trace calculation into solving PDEs. 
	Based on scalable PDE solvers and the ideas of low-rank approximation \cite{bui2013computational,Ghattas2021ActaNumerica}, the proposed method will also be scalable that can be employed for large-scale problems.
	\item We verify the accuracy of the NCP-iMFVI approach on a one-dimensional elliptic inverse problem. 
	In addition, we demonstrate the scalability of the NCP-iMFVI approach for the number of parameters by the inverse source problem of the Helmholtz equation and the inverse permeability of the steady-state Darcy flow equation.
\end{itemize}

The outline of this paper is as follows. 
In Section $\ref{sec2}$, we construct the non-centered VI method based on the framework under the linear case and verify the essential conditions in VI theory.
In Subsection $\ref{subsec2.1}$, we introduce the iMFVI theory proposed in \cite{jia2021variational}, and illustrate the critical problems of the infinite-dimensional hierarchical approach. 
In Subsection $\ref{subsec2.2}$, we propose the non-centered Bayesian formulation. 
In Subsection $\ref{subsec2.3}$, we construct the general theory of NCP-iMFVI approach. 
In Subsection $\ref{subsec2.4}$, we discuss the relationships between the NCP formulation and the approach introduced in Subsection $\ref{subsec2.1}$, and provide convincing results to explain the reason for employing the NCP formulation.
In Subsections $\ref{subsec2.5}$ and $\ref{subsec2.6}$, the computational details are provided. 
In Section $\ref{sec3}$, we employ the developed non-centered VI method to three inverse problems, governed by the simple smooth equation, Helmholtz equation and the steady state Darcy flow equation, with noise and hyper-parameter both Gaussian.
Furthermore, in each numerical simulation, we will illustrate the mesh independence as expected for ``Bayesianize-then-discretize'' approach. 
In Section $\ref{sec4}$, we summarize our achievements and claim some deficiencies, and further investigate directions.

\section{Non-centered infinite-dimensional VI method}\label{sec2}

This section provides a non-centered infinite-dimensional VI method to solve the hierarchical Bayesian inverse problem.
Different from the iMFVI theory proposed in \cite{jia2021variational}, such a method avoids the obstacle that occurred in the iMFVI theory and has an explicit relationship with the centered VI problem.
Furthermore, under the settings in our article, we can clarify the relationship between these two VI problems.
Based on the analysis, an iterative algorithm is then proposed.

\subsection{Critical problems of infinite-dimensional hierarchical approach}\label{subsec2.1}
In this subsection, we first introduce the infinite-dimensional hierarchical Bayesian problem. 
Let $\mathcal{H}_u$ be a separable Hilbert space and $N_d$ be a positive integer. 
Denote that $\mathcal{N}(u, \mathcal{C})$ is a Gaussian measure with mean $u$ and covariance operator $\mathcal{C}$.
The linear inverse problem can be described as follow:
\begin{align}\label{eq:ca}
 \bm{d} = Hu + \bm{\epsilon},
\end{align}
where $\bm{d}\in \mathbb{R}^{N_d}$ is the measurement data, $u\in \mathcal{H}_u$ is the interested parameter, $H$ is a bounded linear operator from $\mathcal{H}_u$ to $\mathbb{R}^{N_d}$, and $\bm{\epsilon}$ is a Gaussian random vector with zero mean and variance $\bm{\Gamma}_{\text{noise}} := \tau^{-1}\textbf{I}$ ($\tau$ is a fixed positive number), which means
\begin{align}\label{eq:noise}
 \bm{\epsilon} \sim \mathcal{N}(0, \bm{\Gamma}_{\text{noise}}).
\end{align}

We adopt a hierarchical Bayesian approach with the unknown parameter $u \sim \mu^{u, \lambda}_0 = \mathcal{N}(0, \lambda^{2}\mathcal{C}_0)$, where $\mathcal{C}_0:\mathcal{H}_u\rightarrow \mathcal{H}_u$ is a positive defined, symmetric and trace-class operator, and $\lambda^2$ is the amplitude of prior variance. 
Let $\lbrace \alpha_k, e_k \rbrace_{k=1}^{\infty}$ be the eigensystem of the operator $\mathcal{C}_0$ such that 
$\mathcal{C}_0 e_k = \alpha_k e_k$ for $k=1,2,\cdots$.
Without loss of generality, we assume that the eigenvectors $\lbrace e_k \rbrace ^{\infty}_{k=1}$ are orthonormal and the eigenvalues $\lbrace \alpha_k \rbrace ^{\infty}_{k=1}$ are in descending order.
Assume that $\lambda$ is a Gaussian random variable with mean $\bar{\lambda}$ and variance $\sigma > 0$, i.e., $\lambda \sim \mu^{\lambda}_0$.
According to ($\ref{eq:ca}$), let us denote $\Phi(u) = \frac{1}{2}\lVert Hu - \bm{d} \rVert^2_{\bm{\Gamma}_{\text{noise}}}$ to be the potential function, where $\lVert \cdot \rVert^2_{\bm{\Gamma}_{\text{noise}}} = \lVert \bm{\Gamma}^{-1/2}_{\text{noise}}\cdot \rVert^2_{l^2}$, with $\lVert \cdot \rVert_{l^2}$ denoting the usual $l^2$-norm.
For notational convenience, we denote $\lVert \cdot \rVert_{l^2}$ by $\lVert \cdot \rVert$.
Then based on the Bayes' formula, the posterior measure $\widetilde{\mu}$ satisfies
\begin{align}\label{eq:bayeca}
 \widetilde{\mu}(du, d\lambda) \varpropto \exp(-\Phi(u))\widetilde{\mu}_0(du, d\lambda),
\end{align}
where $\widetilde{\mu}_0(du, d\lambda) = \mu_0^{u,\lambda}(du)\mu_0^{\lambda}(d\lambda)$.

For employing the iMFVI theory proposed in \cite{jia2021variational}, we need that the conditional prior measures with different
hyper-parameters, i.e., $\mu^{u, \lambda}_0$, are equivalent with each other. 
However, according to Remark 2.10 in \cite{da2006introduction}, we know that the Gaussian measures $\mu^{u, \lambda}_{0}$ and 
$\mu^{u, \lambda^{\prime}}_{0}$ are singular with each other if the hyper-parameters $\lambda \neq \lambda'$.
In order to avoid this singularity problem, the prior covariance operators with hyper-parameters introduced in \cite{jia2021variational} are as follows:
\begin{align}\label{CKlam}
\mathcal{C}^{K}_0(\lambda) := \sum^{K}_{j=1} \lambda^2 \alpha_j e_j \otimes e_j + \sum^{\infty}_{j=K+1} \alpha_j e_j \otimes e_j,
\end{align}
where $K$ is a pre-specified positive integer. 
This formulation has also been employed in \cite{feng2018adaptive}, which will be appropriate if the information in the data
only effective on the first $K$ terms of the prior covariance eigensystem.  
However, if the data is particularly informative and far from the prior, the prior covariance operator specified in (\ref{CKlam}) will lead to a Bayesian inference model that lacks the ability to incorporate the data information.
In the following, we will construct a new approach to overcome the difficult singularity problem and, at the same time, add hyper-parameter to all terms of the eigensystem of the prior covariance operator. 

\subsection{Non-centered formulation}\label{subsec2.2}
From Subsection $\ref{subsec2.1}$, we know that the priors $\mu^{u, \lambda}_0$ are mutually singular 
for different values of $\lambda$, 
which conflicts with the basic assumptions of the iMFVI theory \cite{jia2021variational}.
As indicated in \cite{agapiou2014analysis}, 
similar difficulties were also encountered for employing the two-component Metropolis-within-Gibbs (MwG) algorithms.
In the literature of MwG algorithms \cite{chen2018dimension,papaspiliopoulos2007general}, sampling $u$ and $\lambda$ iteratively is usually called the \emph{centered parameterization} (CP), which will suffer from increasing slow convergence as the discretization level increases. 
Hence, the \emph{non-centered parameterization} (NCP) methods are proposed in the investigations of MwG type algorithms. 
In the following, we introduce the NCP method into the formulation of the iMFVI theory to overcome the mutually singular problem of probability measures. 

Specifically, we reparameterize the prior $\mu^{\prime}_0(du, d\lambda) = \mu_0^{u,\lambda}(du)\mu_0^{\lambda}(d\lambda)$ by writing $u = \lambda v$ with $ v \sim \mu^{v}_0 = \mathcal{N}(0, \mathcal{C}_0)$.
This parameterization will not change the original assumptions on $u$ since the parameter $u=\lambda v$ still distributed according to $\mathcal{N}(0, \lambda^2 \mathcal{C}_0)$.
By working in variables $(v, \lambda)$ rather than $(u, \lambda)$, the measures $\mu^{v}_0$ and $\mu^{\lambda}_0$ are priori independent, which avoids the lack of robustness arising from mutual singularity and the non-informative issue arising from the truncated approach (\ref{CKlam}).
As for this non-centered parameterization, we employ the prior probability measure as follows:
\begin{align}\label{eq:prior}
 \mu_0 = \mu^v_0 \times \mu^{\lambda}_0,
\end{align}
then the forward problem turns to
\begin{align}\label{eq:nc}
 \bm{d} = H(\lambda v) + \bm{\epsilon},
\end{align}
where $\bm{\epsilon} \sim \mathcal{N}(0, \bm{\Gamma}_{\text{noise}})$ is the Gaussian noise.

Concerned with the NCP formulation, i.e., formulas (\ref{eq:prior}) and (\ref{eq:nc}), 
we need to answer the following two fundamental questions: 
\begin{itemize}
	\item Whether the Bayes' formula holds rigorously under the current NCP setting;
	\item What is the relationship between CP and NCP formulations for constructing iMFVI methods. 
\end{itemize}
The first question is addressed in the following Theorem \ref{BayesTheoremNCP}. To answer the second question, 
we need to give a brief illustration of the general VI theory. 
Hence, the statements are postponed to Subsection \ref{subsec2.4}.

\begin{theorem}\label{BayesTheoremNCP}
Let $\mathcal{H}_u$ be a separable Hilbert space and $N_d$ be a positive integer. Let $\mu_0 = \mu^v_0 \times \mu^{\lambda}_0 $ be a Gaussian measure defined on $\mathcal{H}_u \times \mathbb{R}$, and set the Gaussian noise $\bm{\epsilon} \sim \mathcal{N}(0, \bm{\Gamma}_{\text{noise}})$.
Let $\Phi : \mathcal{H}_u \times \mathbb{R} \rightarrow \mathbb{R}$ be defined as
\begin{align}\label{eq:likelihood}
 \Phi(v, \lambda) = \frac{1}{2}\lVert \bm{d} - \lambda Hv\rVert ^{2}_{\bm{\Gamma}_{\text{noise}}},
\end{align}
where $\bm{\Gamma}_{\text{noise}} = \tau^{-1}\textbf{I}$ and $\tau$ is a positive constant.
Then $\mu \ll \mu_0$ is a well-defined probability measure on $\mathcal{H}_u \times \mathbb{R}$, with Radon-Nikodym derivative
\begin{align}\label{eq:post}
	\qquad\,\,\frac{d\mu}{d\mu_0}(v, \lambda) = \frac{1}{Z_{\mu}}\exp (-\Phi(v, \lambda) ),
\end{align}
where $Z_{\mu}$ is a positive and finite constant given by
\begin{align*}
	Z_{\mu} = \int_{\mathcal{H}_u \times \mathbb{R}} \exp (-\Phi(v, \lambda))\mu^v_0(dv)\mu^{\lambda}_0(d\lambda).\\
\end{align*}
\end{theorem}

\begin{proof}
	Since $H$ is a bounded linear operator, it is easy to obtain that $\Phi(v, \lambda)$ is continuous with respect to the variables $v, \lambda$.
	Noticing that the operator $H$ is a bounded linear operator, we have $\lVert Hv \rVert \leqslant M\lVert v \rVert_{\mathcal{H}_u}$, where $M$ is a fixed constant.
	Through a simple calculation, we have
	\begin{align*}
		\bigg\lvert \Phi(v, \lambda;\bm{d}_1) - \Phi(v, \lambda;\bm{d}_2) \bigg\rvert &= \frac{1}{2}\bigg\lvert \lVert \bm{d}_1 - \lambda Hv\rVert 	^2_{\bm{\Gamma}_{\text{noise}}} -  \lVert \bm{d}_2 - \lambda Hv\rVert ^2_{\bm{\Gamma}_{\text{noise}}} \bigg\rvert \\
		&= \frac{\tau}{2} \bigg \lvert \lVert \bm{d}_1 \rVert^2 - \lVert \bm{d}_2 \rVert^2 -2 \langle \bm{d}_1 - \bm{d}_2, \lambda Hv \rangle_{\mathbb{R}^{N_d}} \bigg \rvert \\
		&\leqslant \tau (\lVert \bm{d}_1 \rVert + \lVert \bm{d}_2 \rVert + \lVert \lambda Hv \rVert)\lVert \bm{d}_1 - \bm{d}_2 \rVert \\
		&\leqslant \tau (\lVert \bm{d}_1 \rVert + \lVert \bm{d}_2 \rVert + \lvert \lambda\rvert M\lVert v \rVert_{\mathcal{H}_u})\lVert \bm{d}_1 - \bm{d}_2 \rVert.
	\end{align*}
	Hence, we easily know that the Assumption 1 and conditions of Theorems 15-16 in \cite{dashti2013bayesian} are satisfied, which are provided in the Appendix.
	Then the Bayes' formula ($\ref{eq:post}$) is well-defined.
	Moreover, $Z_{\mu}$ is positive and finite.
\end{proof}

\subsection{General theory of iMFVI}\label{subsec2.3}
Based on the preparations provided in Subsections \ref{subsec2.1} and \ref{subsec2.2}, 
we will apply a general mean-field assumption based iMFVI theory developed in \cite{jia2021variational} to our NCP setting.
Specifically, we will prove a theorem, which provides the foundation for constructing practical iterative algorithms. 

Let $\mathcal{H}$ be a separable Hilbert space, $\mathcal{M}(\mathcal{H})$ be the set of Borel probability measures on $\mathcal{H}$, and $\mathcal{A} \subset \mathcal{M}(\mathcal{H})$ be a set of ``simpler'' measures that can be efficiently calculated. 
Following Subsection $\ref{subsec2.2}$, we define prior measure $\mu_0$ on $\mathcal{H}$ that can be decomposed as $\mu_0 = \mu^v_0 \times \mu^{\lambda}_0$.
Let the measure $\mu$ be the posterior measure with respect to $\mu_0$ defined on $\mathcal{H}$, 
then we have the Bayes' formula ($\ref{eq:post}$). 

For any $\nu \in \mathcal{M}(\mathcal{H})$ that is absolutely continuous with respect to $\mu$, 
the Kullback-Leibler (KL) divergence is defined as 
\begin{align*}
 D_{\text{KL}}(\nu || \mu) &= \int_{\mathcal{H}} \log \bigg(\frac{d\nu}{d\mu}(u) \bigg)\frac{d\nu}{d\mu}(u)\mu(du)
 =\mathbb{E}^{\mu} \bigg[\log \bigg(\frac{d\nu}{d\mu}(u) \bigg) \frac{d\nu}{d\mu}(u) \bigg].
\end{align*}
Here, the notation $\mathbb{E^{\mu}}$ means taking expectation with respect to the probability measure $\mu$.
If $\nu$ is not absolutely continuous with respect to $\mu$, the KL divergence is defined as $+\infty$. 
The iMFVI methods aim to find the closest probability measure $\nu$ to the posterior measure $\mu$ with respect to the KL divergence from the set $\mathcal{A}$, i.e., solving the following minimization problem:  
\begin{align}\label{expre:problem}
 \mathop{\arg\min}\limits_{{\nu \in \mathcal{A}}} D_{{\text{KL}}} (\nu \Vert \mu).
\end{align}
The mean-field assumption means that all components of the parameters are assumed to be independent. 
For the current setting, we assume that the random variables $v$ and $\lambda$ are independent with each other.  
Hence, the space $\mathcal{H}$ and the set $\mathcal{A}$ can be decomposed as follows
\begin{align*}
 \mathcal{H}=\mathcal{H}_u \times \mathcal{H}_{\lambda}, \qquad 
 \mathcal{A}=\mathcal{A}_v \times \mathcal{A}_{\lambda},
\end{align*}
where $\mathcal{H}_u$ is a separable Hilbert space, $\mathcal{H}_{\lambda}$ obviously is $\mathbb{R}$, $\mathcal{A}_v \subset \mathcal{M}(\mathcal{H}_u)$, and $\mathcal{A}_{\lambda} \subset \mathcal{M}(\mathcal{H}_{\lambda}) = \mathcal{M}(\mathbb{R})$. 
In addition, we assume that the approximated probability measure $\nu$ is equivalent to $\mu_0$ defined by
\begin{align}\label{DefApproxMeasure}
 \frac{d\nu}{d\mu_0}(v, \lambda) = \frac{1}{Z_{\nu}}\exp (-\Phi_v(v)-\Phi_{\lambda}(\lambda)),
\end{align}
where $\Phi_{v}(\cdot)$ and $\Phi_{\lambda}(\cdot)$ are two potential functions, 
and $Z_\nu$ is the normalization constant. 
Obviously, the approximated measure $\nu$ can be decomposed into two components that are absolutely continuous with respect to the corresponding components of the prior measure, i.e., 
\begin{align}
	\frac{d\nu^v}{d\mu^v_0}(v) &= \frac{1}{Z^v_{\nu}}\exp (-\Phi_v(v))\label{DefApproxMeasure1}, \\
	\frac{d\nu^{\lambda}}{d\mu^{\lambda}_0}({\lambda}) &= \frac{1}{Z^{\lambda}_{\nu}}\exp (-\Phi_{\lambda}(\lambda))\label{DefApproxMeasure2},
\end{align}
with ${Z^v_{\nu}} = \mathbb{E}^{\mu^v_0} \left[\exp (-\Phi_v(v)) \right]$ and 
${Z^{\lambda}_{\nu}} = \mathbb{E}^{\mu^{\lambda}_0} \left[\exp (-\Phi_{\lambda}({\lambda}))\right]$. 
With these assumptions, the problem ($\ref{expre:problem}$) can be written specifically as 
\begin{align}\label{optimProb}
	\mathop{\arg\min}\limits_{\substack{\nu_v \in \mathcal{A}_{v} \\ 
			\nu_{\lambda} \in \mathcal{A}_{\lambda}}}D_{{\text{KL}}}
	\bigg (\nu^v \times \nu^{\lambda} \bigg \Vert \mu \bigg ).
\end{align}

For the finite-dimensional theory, the sets $\mathcal{A}_{v}$ and $\mathcal{A}_{\lambda}$ are specified as 
the set of any probability density functions. For the infinite-dimensional theory developed in \cite{jia2021variational}, 
we need more illustrations of the sets $\mathcal{A}_v$ and $\mathcal{A}_{\lambda}$, which ensure the measures 
obtained by iterations are still in some admissible sets. 

\begin{assumption}\label{assump1}

	Let $\nu^v$ and $\nu^{\lambda}$ be the approximate measures defined in $(\ref{DefApproxMeasure1})$ and $(\ref{DefApproxMeasure2})$ respectively.
	Let us define $T^v_N = \lbrace v|1/N \leqslant \lVert v\rVert_{\mathcal{H}_u} \leqslant N\rbrace$ that satisfies $\sup_N \mu ^v_0(T^v_N)=1$,
	and define $T^{\lambda}_N = \lbrace \lambda|-N \leqslant \lambda \leqslant N \rbrace$ that satisfies $\sup_N\mu ^{\lambda}_0(T^{\lambda}_N)=1$.
	Furthermore we assume that
	\begin{align}\label{expre:condi}
		\begin{split}
			T_1 &:= \sup \limits_{v \in T^v_N} \int_{\mathbb{R}}\Phi(v, \lambda)\cdot 1_{A}(v, \lambda)\nu^{\lambda}(d\lambda) < \infty\\
			T_2 &:= \sup \limits_{\lambda \in T^{\lambda}_N} \int_{\mathcal{H}_u}\Phi(v, \lambda)\cdot 1_{A}(v, \lambda)\nu^{v}(dv) < \infty\\
			T_3 &:= \int_{\mathcal{H}_u} \exp \bigg (-\int_{\mathbb{R}}\Phi(v, \lambda)\cdot \nu^{\lambda}(d\lambda) \bigg )\max(1,\lVert v \rVert^{2}_{\mathcal{H}_u})\mu^{v}_0(dv) < \infty\\ 
			T_4 &:= \int_{\mathbb{R}} \exp \bigg (-\int_{\mathcal{H}_u}\Phi(v, \lambda)\cdot \nu^{v}(dv) \bigg ) \max(1, \lambda^2)\mu^{\lambda}_0(d\lambda) < \infty,
		\end{split}
	\end{align}
	where the set $A := \lbrace v, \lambda \vert \Phi(v, \lambda) \geqslant 0 \rbrace$.
\end{assumption}

Let
\begin{align}\label{expre:R}
	\begin{split}
  R^1_v &= \bigg \lbrace \Phi_v \bigg | \sup \limits_{v \in T^v_N} \Phi_v(v) < \infty, \quad \forall N > 0 \bigg \rbrace, \\
  R^2_v &= \bigg \lbrace \Phi_v \bigg | \int_{\mathcal{H}_u}\exp(-\Phi_v(v))\max(1, \lVert v\rVert^2_{\mathcal{H}_u})\mu^v_0(dv) < \infty \bigg \rbrace, \\
  R^1_{\lambda} &= \bigg \lbrace \Phi_{\lambda} \bigg | \sup \limits_{{\lambda} \in T^{\lambda}_N} \Phi_{\lambda}({\lambda}) < \infty, \quad \forall N > 0 \bigg \rbrace, \\
  R^2_{\lambda} &= \bigg \lbrace \Phi_{\lambda} \bigg | \int_{\mathbb{R}}\exp(-\Phi_{\lambda}({\lambda}))\max(1, \lambda^2)\mu^{\lambda}_0(d\lambda) < \infty\bigg \rbrace.
	\end{split}
\end{align}

Now we define:
\begin{align}\label{expreAv}
	\mathcal{A}_v = \left\{
	\begin{tabular}{l|l}
		\multirowcell{2}[0pt][l]{$\nu^v \in \mathcal{M}(\mathcal{H}_u)$} &
		\multirowcell{2}[0pt][l]{$\nu^v$ is equivalent to $\mu^v_0$ with ($\ref{DefApproxMeasure1}$) holding true,\\
			and $\Phi_v \in R^1_v \bigcap R^2_v$} \\
		&
	\end{tabular}
	\right\},
\end{align}
\begin{align}\label{expreAlam}
	\mathcal{A}_{\lambda} = \left\{
	\begin{tabular}{l|l}
		\multirowcell{2}[0pt][l]{$\nu^{\lambda} \in \mathcal{M}(\mathcal{H}_{\lambda})$} &
		\multirowcell{2}[0pt][l]{$\nu^{\lambda}$ is equivalent to $\mu^{\lambda}_0$ with ($\ref{DefApproxMeasure2}$) holding true,\\
			and $\Phi_{\lambda} \in R^1_{\lambda} \bigcap R^2_{\lambda}$} \\
		&
	\end{tabular}
	\right\}.
\end{align}

\begin{remark}
	\itshape
	Noticing that the definitions of $R^1_v$ and $R^1_{\lambda}$ are different from the corresponding definitions given in \cite{jia2021variational}. 
	Here, we employ a modified version of the general theory provided in the Appendix of \cite{Jia2022VINet}, which relaxes the uniform bounds making the theory more applicable to concrete problems. 
	For the reader's convenience, we briefly introduce the general theory in the Appendix. 
\end{remark}

Now, we give the following key theorem that provides formulas for calculating the potential functions 
$\Phi_v$ and $\Phi_{\lambda}$ introduced in (\ref{DefApproxMeasure}).

\begin{theorem}\label{the:posterior}
Let the prior measure $\mu_0$, noise measure, and the posterior measure $\mu$ are defined as in $(\ref{eq:prior}), (\ref{eq:noise}), (\ref{eq:post})$, and $\Phi(v, \lambda)$ is defined as in $(\ref{eq:likelihood})$.
If the approximate probability measure in problem $(\ref{optimProb})$ satisfies Assumption $\ref{assump1}$, then problem $(\ref{optimProb})$ possesses a solution $\nu = \nu^v \times \nu^{\lambda} \in \mathcal{M}(\mathcal{H})$ with the following form:
\begin{align}
 \frac{d\nu}{d\mu_0}(v, \lambda) \varpropto \exp \bigg(-(\Phi_v(v) + \Phi_{\lambda}(\lambda)) \bigg),
\end{align}
where
\begin{align}\label{PotenFun1}
\begin{split}
 \Phi_v(v) = \int_{\mathbb{R}}\Phi (v, \lambda)\nu^{\lambda}(d\lambda) + \text{Const}, \quad
 \Phi_{\lambda}(\lambda) = \int_{\mathcal{H}_u}\Phi (v, \lambda)\nu^v(dv) + \text{Const},
\end{split}
\end{align}
and here ``Const'' denotes constants that are not relevant to the interested parameters. 
Furthermore, we have
\begin{align}
\begin{split}
 \nu^v(v) \varpropto \exp (-\Phi_v(v))\mu^v_0(dv), \quad
 \nu^{\lambda}(\lambda) \varpropto \exp (-\Phi_{\lambda}(\lambda))\mu^{\lambda}_0(d\lambda). 
\end{split}
\end{align}
\end{theorem}

~\\

\begin{proof}

To prove the theorem, we need to verify the conditions $(\ref{expre:condi})$ given in Assumption $\ref{assump1}$.
Noticing the fact that operator $H$ is a bounded linear operator, we have $\lVert Hv \rVert \leqslant M\lVert v \rVert_{\mathcal{H}_u}$.

For term $T_1$, we have
\begin{align*}
	T_1 &= \sup \limits_{v \in T^v_N}\int_{\mathbb{R}}\Phi(v, \lambda)\cdot 1_{A}(v, \lambda)\nu^{\lambda}(d\lambda) \\
	&= \sup \limits_{v \in T^v_N}\int_{\mathbb{R}}\frac{1}{2}\lVert \bm{d} - \lambda Hv \rVert ^2_{\bm{\Gamma}_{\text{noise}}}\nu^{\lambda}(d\lambda) \\
	&= \sup \limits_{v \in T^v_N}\int_{\mathbb{R}}\frac{\tau}{2}\bigg (\lVert \bm{d} \rVert ^2 + \lVert \lambda Hv \rVert ^2 - 2\langle d, \lambda Hv \rangle \bigg )\nu^{\lambda}(d\lambda) \\
	&\leqslant \sup \limits_{v \in T^v_N}\int_{\mathbb{R}}\tau\bigg (\lVert \bm{d} \rVert ^2 + \lambda^2\lVert Hv \rVert ^2 \bigg )\nu^{\lambda}(d\lambda) \\
	&\leqslant \sup \limits_{v \in T^v_N} 2\tau M^2\lVert v \rVert^2_{\mathcal{H}_u} \int_{\mathbb{R}}\lambda^2\nu^{\lambda}(d\lambda) + \text{Const} \\
	&\leqslant (2\tau M^2N^2) \cdot \int_{\mathbb{R}}\exp(-\Phi_{\lambda}({\lambda}))\max(1, \lambda^2)\mu^{\lambda}_0(d\lambda) + \text{Const}.
\end{align*}
Recalling the set $R^2_{\lambda}$ defined in $(\ref{expre:R})$, since $\nu^{\lambda} \in \mathcal{A}_{\lambda}$, we have 
\begin{align*}
	\int_{\mathbb{R}}\exp(-\Phi_{\lambda}({\lambda}))\max(1, \lambda^2)\mu^{\lambda}_0(d\lambda) < \infty.
\end{align*}
Then we obtain $T_1 < \infty$.

For term $T_2$, we have
\begin{align*}
	T_2 &= \sup \limits_{\lambda \in T^{\lambda}_N}\int_{\mathcal{H}_u}\Phi(v, \lambda)\cdot 1_{A}(v, \lambda)\nu^{v}(dv) \\
	&= \sup \limits_{\lambda \in T^{\lambda}_N} \int_{\mathcal{H}_u}\frac{1}{2}\lVert \bm{d} - \lambda Hv \rVert ^2_{\bm{\Gamma}_{\text{noise}}}\nu^{v}(dv) \\
	&\leqslant \sup \limits_{\lambda \in T^{\lambda}_N}\int_{\mathcal{H}_u}\tau\bigg (\lVert \bm{d} \rVert ^2 + \lambda^2\lVert Hv \rVert ^2 \bigg )\nu^v(dv) \\
	&\leqslant \sup \limits_{\lambda \in T^{\lambda}_N} 2\tau M^2\lambda^2 \int_{\mathcal{H}_u}\lVert v \rVert ^2_{\mathcal{H}_u}\nu^v(dv) + \text{Const} \\
	&\leqslant (2\tau M^2N^2) \cdot \int_{\mathcal{H}_u}\exp(-\Phi_v(v))\max(1, \lVert v\rVert^2_{\mathcal{H}_u})\mu^v_0(dv) + \text{Const}.
\end{align*}
Recalling the set $R^2_v$ defined in $(\ref{expre:R})$, since $\nu^v \in \mathcal{A}_v$, we have 
\begin{align*}
	\int_{\mathcal{H}_u}\exp(-\Phi_v(v))\max(1, \lVert v\rVert^2_{\mathcal{H}_u})\mu^v_0(dv) < \infty.
\end{align*}
Then we obtain $T_2 < \infty$.

For term $T_3$, we notice $-\int_{\mathbb{R}}\Phi(v, \lambda)\cdot \nu^{\lambda}(d\lambda) \leqslant 0$.
Then the term $T_3$ can be estimated as follows:
\begin{align*}
	T_3 &\leqslant \int_{\mathcal{H}_u} \max(1,\lVert v \rVert^{2}_{\mathcal{H}_u})\mu^{v}_0(du).
\end{align*}

For term $T_4$, we use the same strategy to obtain that
\begin{align*}
	T_4 &\leqslant \int_{\mathbb{R}} \max(1, \lambda^2)\mu^{\lambda}_0(d\lambda) < \infty.
\end{align*}
The proof is completed by combining the estimates of $T_1, T_2, T_3$, and $T_4$.
\end{proof}

The above Theorem \ref{the:posterior} provides explicit expressions of the potential functions $\Phi_{v}$ and $\Phi_{\lambda}$ relating to the parameters $v$ and $\lambda$, respectively. 
Although the equalities (\ref{PotenFun1}) are not closed form solutions, they can yield a practical iterative scheme, which will be explicitly formulated in Subsection \ref{subsec2.5} below. 

\subsection{Relationships of CP and NCP formulations}\label{subsec2.4}
In this section, we discuss the relationships between CP and NCP formulations.
The choice of CP and NCP formulations depends on the methods employed, and both formulations have advantages and disadvantages.
For example, the CP formulation is more appropriate when computing the maximum a posterior estimates \cite{Dunlop2020SMAIJCM}. 
However, the NCP formulation will be better when the Metropolis-within-Gibbs algorithm is employed \cite{agapiou2014analysis}.
Before constructing an iterative scheme, let us first clarify the relationships between the CP and NCP formulations under the circumstances of developing iMFVI methods. 

Due to the mutually singular problem discussed in Subsection $\ref{subsec2.1}$, we could not formulate the CP problem based on the current iMFVI theory.
To avoid the obstacle, we need to introduce the covariance operator $\mathcal{C}^K_0$.
Thus, it is hard to reveal the connection between CP and NCP formulations in infinite-dimensional spaces.
In order to state the relationships more clearly, we need to start with the finite-dimensional case.
And then, we can clarify the relationships in the infinite-dimensional space with the help of the conclusions in the finite-dimensional space.
In this subsection, we shall prove that 
\begin{itemize}
	\item The VI problems formulated by CP and NCP attain the same minimal value;
	\item One possible minimum point (probability density function) of two problems can be transformed into each other by taking a reparameterization.
\end{itemize}

Recall that $\lbrace \alpha_k, e_k\rbrace_{k=1}^{\infty}$ is the eigensystem of the operator $\mathcal{C}_0$ such that $\mathcal{C}_0 e_k = \alpha_k e_k$, as we stated in Subsection $\ref{subsec2.1}$.
We denote by $P^N$ the orthogonal projection of $\mathcal{H}_u$ onto $\mathcal{H}^N_u$, that is $\mathcal{H}^N_u = P^N\mathcal{H}_u := \text{span}\lbrace e_1, e_2, \cdots, e_N \rbrace$.
Define $\bm{\mathcal{C}}^N_0 = P^N \mathcal{C}_0 P^N$, such that $\bm{u}^N \sim \mu^{(\bm{u}^N, \lambda)}_0 = \mathcal{N}(0, \lambda^2\bm{\mathcal{C}}^N_0)$, and let $\bm{u}^N = P^Nu = \sum^{N}_{k=1} u_ke_k \in \mathcal{H}^N_u$.

Let us consider the finite-dimensional linear inverse problem:
\begin{align*}
	\bm{d} = H\bm{u}^N + \bm{\epsilon},
\end{align*}
where $\bm{d} \in \mathbb{R}^{N_d}$, $\bm{\epsilon} \in \mathbb{R}^{N_d}$ is the random Gaussian noise with mean zero and variance $\bm{\Gamma}_{\text{noise}} = \tau^{-1}\textbf{I}$.
We take the prior measure $\widetilde{\mu}^N_0 = \mu^{(\bm{u}^N, \lambda)}_0 \times \mu^{\lambda}_0$, and the potential function
\begin{align*}
	\Phi(\bm{u}^N) = \frac{1}{2}\lVert H\bm{u}^N - \bm{d} \rVert^2_{\bm{\Gamma}_{\text{noise}}}.
\end{align*}
According to the Bayes' formula, the posterior measure $\widetilde{\mu}^N$ is given by the Rando-Nikodym derivative
\begin{align*}
	\frac{d\widetilde{\mu}^N}{d\widetilde{\mu}^N_0}(\bm{u}^N, \lambda) = \frac{1}{\widetilde{Z}^N}\exp(-\Phi(\bm{u}^N)),
\end{align*}
where $\widetilde{Z}^N = \int_{\mathcal{H}^N\times \mathbb{R}}\exp(-\Phi(\bm{u}^N))\widetilde{\mu}^N_0(d\bm{u}^N, d\lambda)$ is the normalization constant.

Taking the NCP formulation $\bm{u}^N = \lambda \bm{v}^N$, we reformulate the linear inverse problem as
\begin{align*}
	\bm{d} = H(\lambda\bm{v}^N) + \bm{\epsilon}.
\end{align*}
Since $\bm{u}^N \sim \mu^{(\bm{u}^N, \lambda)}_0$, we obviously find that $\bm{v}^N \sim \mu^{\bm{v}^N}_0 = \mathcal{N}(0, \bm{\mathcal{C}}^N_0)$, and $\bm{v}^N = P^Nv \in \mathcal{H}^N$.
We have the prior measure $\mu^N_0 = \mu^{\bm{v}^N}_0 \times \mu^{\lambda}_0$, and the potential function
\begin{align*}
	\Phi(\bm{v}^N, \lambda) = \frac{1}{2}\lVert H(\lambda \bm{v}^N) - \bm{d} \rVert^2_{\bm{\Gamma}_{\text{noise}}}.
\end{align*}
Then the posterior measure $\mu^N$ is given by the Rando-Nikodym derivative
\begin{align}\label{eq:ncfin}
	\frac{d\mu^N}{d\mu^N_0}(\bm{v}^N, \lambda) = \frac{1}{Z^N}\exp(-\Phi(\bm{v}^N, \lambda)),
\end{align}
where $Z^N = \int_{\mathcal{H}^N\times \mathbb{R}}\exp(-\Phi(\bm{v}^N, \lambda))\mu^N_0(d\bm{v}^N, d\lambda)$ is the normalization constant.
For notational convenience, we will use the same symbol for the probability measure and probability density function of the corresponding measure in this section.
For example, we use the symbol $\mu^N$ to denote the probability measure $\mu^N$ and the probability density function of $\mu^N$.

Before introducing the finite-dimensional VI problems, we need to clarify that:
\begin{itemize}
	\item In the CP case, we aim to find the approximated probability density $\widetilde{\nu}^{\dagger}(\bm{u}^N, \lambda)$ to minimize the KL divergence between $\widetilde{\nu}^N(\bm{u}^N, \lambda)$ and $\widetilde{\mu}^N(\bm{u}^N, \lambda)$.
	Under the mean-field assumption, the approximated probability density could be written as $\widetilde{\nu}^N(\bm{u}^N, \lambda) = \widetilde{\nu}^N_1(\bm{u}^N)\widetilde{\nu}^N_2(\lambda)$, where $\widetilde{\nu}^N_1(\bm{u}^N)$ and $\widetilde{\nu}^N_2(\lambda)$ are the probability density functions.
	\item In the NCP case, the aim is to find the approximated probability density $\nu^{\dagger}(\bm{v}^N, \lambda)$ to minimize the KL divergence between $\nu^N(\bm{v}^N, \lambda)$ and $\mu^N(\bm{v}^N, \lambda)$.
	Taking the assumptions stated in Subsection $\ref{subsec2.3}$, $\nu^N(\bm{v}^N, \lambda)$ could be written as $\nu^N(\bm{v}^N, \lambda) = \nu^N_1(\bm{v}^N)\nu^N_2(\lambda)$, where $\nu^N_1(\bm{v}^N)$ and $\nu^N_2(\lambda)$ are the probability density functions.
\end{itemize}
Based on these goals, we introduce the sets
\begin{align*}
	\widetilde{\mathcal{A}} &:= \bigg \lbrace \widetilde{\nu}^N(\bm{u}^N, \lambda) \ \bigg \vert \ \widetilde{\nu}^N(\bm{u}^N, d\lambda) = \widetilde{\nu}^N_1(\bm{u}^N)\widetilde{\nu}^N_2(\lambda) \bigg \rbrace, \\
	\mathcal{A} &:= \bigg \lbrace \nu^N(\bm{v}^N, \lambda) \ \bigg \vert \ \nu^N(\bm{v}^N, \lambda) = \nu^N_1(\bm{v}^N)\nu^N_2(\lambda) \bigg \rbrace.
\end{align*}

Then, the VI problem formulated by CP can be written as:
\begin{align}\label{eq:vicafin}
	\begin{split}
		&\mathop{\arg\min}\limits_{\widetilde{\nu}^N \in \widetilde{\mathcal{A}}} D_{{\text{KL}}} (\widetilde{\nu}^N \Vert \widetilde{\mu}^N)\\
		= &\mathop{\arg\min}\limits_{\widetilde{\nu}^N_1, \widetilde{\nu}^N_2}\int_{\mathcal{H}^N\times \mathbb{R}}	\log \frac{\widetilde{\nu}^N_1(\bm{u}^N)\widetilde{\nu}^N_2(\lambda)}{\widetilde{\mu}^N(\bm{u}^N, \lambda)}\widetilde{\nu}^N_1(\bm{u}^N)\widetilde{\nu}^N_2(\lambda)d\bm{u}^Nd\lambda.
	\end{split}
\end{align}
Similarly, the VI problem formulated by NCP has the following form:
\begin{align}\label{eq:vincfin}
	\begin{split}
		&\mathop{\arg\min}\limits_{\nu^N \in \mathcal{A}} D_{{\text{KL}}} (\nu^N \Vert \mu^N) \\
		= &\mathop{\arg\min}\limits_{\nu^N_1, \nu^N_2}\int_{\mathcal{H}^N\times \mathbb{R}} \log \frac{\nu^N_1(\bm{v}^N)\nu^N_2(\lambda)}{\mu^N(\bm{v}^N, \lambda)}\nu^N_1(\bm{v}^N)\nu^N_2(\lambda)d\bm{v}^Nd\lambda.
	\end{split}
\end{align}

Before discussing the relationship between problems $(\ref{eq:vicafin})$ and $(\ref{eq:vincfin})$, we need to clarify the relationship between the probabiltiy densities $\widetilde{\mu}^N(\bm{u}^N, \lambda)$, and $\mu^N(\bm{v}^N, \lambda)$, which are given by
\begin{align*}
	\widetilde{\mu}(\bm{u}^N, \lambda) &\varpropto \exp \bigg (-\frac{1}{2}\lVert H\bm{u}^N - \bm{d} \rVert^2_{\bm{\Gamma}_{\text{noise}}} \bigg )\widetilde{\mu}^N_0(\bm{u}^N, \lambda), \\
	\mu(\bm{v}^N, \lambda) &\varpropto \exp \bigg (-\frac{1}{2}\lVert H(\lambda \bm{v}^N) - \bm{d} \rVert^2_{\bm{\Gamma}_{\text{noise}}} \bigg )\mu^N_0(\bm{v}^N, \lambda).
\end{align*}
We know that 
\begin{align*}
	\widetilde{\mu}^N_0(\bm{u}^N, \lambda) = \frac{1}{(2\pi)^{\frac{N}{2}}\lambda^N\det(\bm{\mathcal{C}}^N_0)^{\frac{1}{2}}} \exp \bigg (-\frac{1}{2}\lVert \lambda^{-1}\bm{\mathcal{C}}^{-1/2}_0\bm{u}^N\rVert^2 \bigg )\mu^{\lambda}_0(\lambda),
\end{align*}
and 
\begin{align*}
	\mu^N_0(\bm{v}^N, \lambda) = \frac{1}{(2\pi)^{\frac{N}{2}}\det(\bm{\mathcal{C}}^N_0)^{\frac{1}{2}}}\exp \bigg (-\frac{1}{2}\lVert \bm{\mathcal{C}}^{-1/2}_0\bm{v}^N\rVert^2 \bigg )\mu^{\lambda}_0(\lambda).
\end{align*}
Taking $\bm{u}^N = \lambda \bm{v}^N$, it is clear that $\lambda^{-N}\mu^N_0(\bm{v}^N, \lambda) = \widetilde{\mu}^N_0(\bm{u}^N, \lambda)$, and furthermore
\begin{align}\label{eq:postrelate}
	\lambda^{-N}\mu^N(\bm{v}^N, \lambda) = \widetilde{\mu}^N(\bm{u}^N, \lambda).
\end{align}
Then equation $(\ref{eq:vicafin})$ can be written as 
\begin{align}\label{eq:vicafinn}
	\begin{split}
		&\mathop{\arg\min}\limits_{\widetilde{\nu}^N \in \widetilde{\mathcal{A}}} D_{{\text{KL}}} (\widetilde{\nu}^N \Vert \widetilde{\mu}^N)\\
		= &\mathop{\arg\min}\limits_{\widetilde{\nu}^N_1, \widetilde{\nu}^N_2}\int_{\mathcal{H}^N\times \mathbb{R}}	\log \frac{\widetilde{\nu}^N_1(\bm{u}^N)\widetilde{\nu}^N_2(\lambda)}{\widetilde{\mu}^N(\bm{u}^N, \lambda)}\widetilde{\nu}^N_1(\bm{u}^N)\widetilde{\nu}^N_2(\lambda)d\bm{u}^Nd\lambda \\
		= &\mathop{\arg\min}\limits_{\widetilde{\nu}^N_1, \widetilde{\nu}^N_2}\int_{\mathcal{H}^N\times \mathbb{R}}	\log \frac{\lambda^N\widetilde{\nu}^N_1(\bm{u}^N)\widetilde{\nu}^N_2(\lambda)}{\mu^N(\bm{v}^N, \lambda)}\lambda^N\widetilde{\nu}^N_1(\bm{u}^N)\widetilde{\nu}^N_2(\lambda)d\bm{v}^Nd\lambda.
	\end{split}
\end{align}

Next, we assume that there exist $\widetilde{\nu}^{*}_1(\bm{v}^N), \nu^{*}_1(\bm{u}^N)$ satisfying
\begin{align*}
	\widetilde{\nu}^{*}_1(\bm{v}^N) := \lambda^N \widetilde{\nu}^N_1(\bm{u}^N), \quad  \nu^{*}_1(\bm{u}^N) := \lambda^{-N}\nu^N_1(\bm{v}^N),
\end{align*}
and set
\begin{align*}
	\widetilde{\mathcal{A}}^{*} &:= \bigg \lbrace \widetilde{\nu}^N \ \bigg \vert \ \widetilde{\nu}^N(\bm{u}^N, \lambda) = \widetilde{\nu}^N_1(\bm{u}^N)\widetilde{\nu}^N_2(\lambda), \widetilde{\nu}^N_1(\bm{u}^N) = \lambda^{-N}\widetilde{\nu}^{*}_1(\bm{v}^N) \bigg \rbrace, \\
	\mathcal{A}^{*} &:= \bigg \lbrace \nu^N \ \bigg \vert \ \nu^N(\bm{v}^N, \lambda) = \nu^N_1(\bm{u}^N)\nu^N_2(\lambda), \nu^N_1(\bm{v}^N) = \lambda^N\nu^{*}_1(\bm{u}^N) \bigg \rbrace.
\end{align*}

Now we provide a theorem that illustrates the relationship between CP and NCP formulations in the finite-dimensional spaces.
\begin{theorem}\label{the:minimal}
	The minimal values of problems $(\ref{eq:vicafin})$ and $(\ref{eq:vincfin})$ are the same.
	That is
	\begin{align*}
	\min_{\widetilde{\nu}^N \in \widetilde{\mathcal{A}}} D_{{\text{KL}}} (\widetilde{\nu}^N \Vert \widetilde{\mu}^N) = \min_{\nu^N \in \mathcal{A}} D_{{\text{KL}}} (\nu^N \Vert \mu^N).
	\end{align*}
	Assume that $\nu^{\dagger} \in \mathcal{A}$ and $\widetilde{\nu}^{\dagger} \in \widetilde{\mathcal{A}}$ are one of the minimized points of these problems, respectively.
	If $\nu^{\dagger} \in \mathcal{A}^{*}, \widetilde{\nu}^{\dagger} \in \widetilde{\mathcal{A}}^{*}$ are satisfied, $\nu^{\dagger}$ and $\widetilde{\nu}^{\dagger}$ can be transformed into the solution of problems $(\ref{eq:vicafin})$ and $(\ref{eq:vincfin})$ by taking a reparameterization.
\end{theorem}

\begin{remark}
	\itshape
	Let $\widehat{\nu} = \widehat{\nu}^v \times \widehat{\nu}^{\lambda}$ be the measure given by Theorem $\ref{the:posterior}$, which is the solution of the NCP-iMFVI problem $(\ref{expre:problem})$.
	Let $\widehat{\nu}^N$ be the pushforward of the measure $\widehat{\nu}$ defined on space $\mathcal{H}^N_u \times \mathbb{R}$, i.e., $\widehat{\nu}^N = \widehat{\nu} \circ (P^N)^{-1}$.
	From Theorem $\ref{the:minimal}$, we know that $\widehat{\nu}^N \in \mathcal{A}^{*}$.
	That is, the measure calculated by Theorem $\ref{the:posterior}$ can be transformed to the solution of problem $(\ref{eq:vicafin})$ formulated by CP in finite-dimensional space.
\end{remark}

\begin{proof}
	Considering the following problems:
	\begin{align}
		\label{prob:caA}\min_{\widetilde{\nu}^N \in \widetilde{\mathcal{A}}^{*}} D_{{\text{KL}}} (\widetilde{\nu}^N\Vert \widetilde{\mu}^N), \\
		\label{prob:ncA}\min_{\nu^N \in \mathcal{A}^{*}} D_{{\text{KL}}} (\nu^N \Vert \mu^N).
	\end{align}
	For notational convenience, let 
	\begin{align*}
		\widetilde{\nu}^{\dagger}(\bm{u}^N, \lambda) = \widetilde{\nu}^{\dagger}_1(\bm{u}^N)\widetilde{\nu}^{\dagger}_2(\lambda) \in \widetilde{\mathcal{A}}^{*}, \quad
		\nu^{\dagger}(\bm{v}^N, \lambda) = \nu^{\dagger}_1(\bm{v}^N)\nu^{\dagger}_2(\lambda) \in \mathcal{A}^{*}
	\end{align*}
	be the solutions to problem $(\ref{prob:caA})$ and $(\ref{prob:ncA})$, respectively.
	Then we have
	\begin{align*}
		D_{{\text{KL}}} (\widetilde{\nu}^{\dagger} \Vert \widetilde{\mu}^N) 
		= &\int_{\mathcal{H}^N\times \mathbb{R}} \log \frac{\lambda^N\widetilde{\nu}^{\dagger}_1(\bm{u}^N)\widetilde{\nu}^{\dagger}_2(\lambda)}{\mu^N(\bm{v}^N, \lambda)}\lambda^N\widetilde{\nu}^{\dagger}_1(\bm{u}^N)\widetilde{\nu}^{\dagger}_2(\lambda)d\bm{v}^Nd\lambda \\
		= &\int_{\mathcal{H}^N\times \mathbb{R}} \log \frac{\widetilde{\nu}^{\dagger*}_1(\bm{v}^N)\widetilde{\nu}^{\dagger}_2(\lambda)}{\mu^N(\bm{v}^N, \lambda)}\widetilde{\nu}^{\dagger*}_1(\bm{v}^N)\widetilde{\nu}^{\dagger}_2(\lambda)d\bm{v}^Nd\lambda \\
		\geq &\int_{\mathcal{H}^N\times \mathbb{R}} \log \frac{\nu^{\dagger}_1(\bm{v}^N)\nu^{\dagger}_2(\lambda)}{\mu^N(\bm{v}^N, \lambda)}\nu^{\dagger}_1(\bm{v}^N)\nu^{\dagger}_2(\lambda)d\bm{v}^Nd\lambda \\
		= &\min_{\nu^N \in \mathcal{A}^{*}} D_{{\text{KL}}} (\nu^N \Vert \mu^N),
	\end{align*}
	where we use the equation $(\ref{eq:postrelate})$ in the first equation.
	Using the same strategy, we obtain that 
	\begin{align*}
		D_{{\text{KL}}} (\nu^{\dagger} \Vert \mu^N) \geq \min_{\widetilde{\nu}^N \in \widetilde{\mathcal{A}}^{*}} D_{{\text{KL}}} (\widetilde{\nu}^N\Vert \widetilde{\mu}^N).
	\end{align*}
	Combining these two inequalities, we have
	\begin{align*}
		\min_{\nu^N \in \mathcal{A}^{*}} D_{{\text{KL}}} (\nu^N \Vert \mu^N) \geq \min_{\widetilde{\nu}^N \in \widetilde{\mathcal{A}}^{*}} D_{{\text{KL}}} (\widetilde{\nu}^N\Vert \widetilde{\mu}^N) \geq \min_{\nu^N \in \mathcal{A}^{*}} D_{{\text{KL}}} (\nu^N \Vert \mu^N),
	\end{align*}
	which indicates
	\begin{align}\label{conclu:Astar}
		\min_{\widetilde{\nu}^N \in \widetilde{\mathcal{A}}^{*}} D_{{\text{KL}}} (\widetilde{\nu}^N\Vert \widetilde{\mu}^N) = \min_{\nu^N \in \mathcal{A}^{*}} D_{{\text{KL}}} (\nu^N \Vert \mu^N).
	\end{align}
	Now these two VI problems have the same minimal value.
	
	Due to the settings of prior measures and noise, measures $\widetilde{\nu}^N_1(d\bm{u}^N), \nu^N_1(d\bm{v}^N)$ are Gaussians by Theorem $\ref{the:posterior}$. 
	Since similar discussions are provided in Subsection $\ref{subsec2.5}$, we omit the details here.
	Without loss of generality, we assume that $\widetilde{\nu}^N_1 = \mathcal{N}(\widetilde{\bm{a}}, \widetilde{\bm{\Sigma}})$, and $\nu^N_1 = \mathcal{N}(\bm{a}, \bm{\Sigma})$, then we have
	\begin{align*}
		\widetilde{\nu}^N_1(\bm{u}^N) &= \frac{1}{(2\pi)^{\frac{N}{2}}\det(\widetilde{\bm{\Sigma}})^{\frac{1}{2}}}\exp \bigg (-\frac{1}{2}\lVert \widetilde{\bm{\Sigma}}^{-1/2}(\bm{u}^N - \widetilde{\bm{a}})\rVert^2 \bigg ) \\
		&= \frac{1}{(2\pi)^{\frac{N}{2}}\det(\widetilde{\bm{\Sigma}})^{\frac{1}{2}}}\exp \bigg (-\frac{1}{2}\lVert \lambda\widetilde{\bm{\Sigma}}^{-1/2}(\bm{v}^N - \widetilde{\bm{a}}/\lambda)\rVert^2 \bigg ),
	\end{align*}
	and
	\begin{align*}
		\nu^N_1(\bm{v}^N) &= \frac{1}{(2\pi)^{\frac{N}{2}}\det(\bm{\Sigma})^{\frac{1}{2}}}\exp \bigg (-\frac{1}{2}\lVert \bm{\Sigma}^{-1/2}(\bm{v}^N - \bm{a})\rVert^2 \bigg ) \\
		&= \frac{1}{(2\pi)^{\frac{N}{2}}\det(\bm{\Sigma})^{\frac{1}{2}}}\exp \bigg (-\frac{1}{2}\lVert \lambda^{-1}\bm{\Sigma}^{-1/2}(\bm{u}^N - \lambda\bm{a})\rVert^2 \bigg ).
	\end{align*} 
	Then measures $\widetilde{\nu}^{N\prime}_1 = \mathcal{N}(\widetilde{\bm{a}}/\lambda, \widetilde{\bm{\Sigma}}/\lambda^2)$, and $\nu^{N\prime}_1 = \mathcal{N}(\lambda\bm{a}, \lambda^2\bm{\Sigma})$ satisfy
	\begin{align*}
		\widetilde{\nu}^N_1(\bm{u}^N) = \lambda^{-N}\widetilde{\nu}^{N\prime}_1(\bm{v}^N) , \quad 
		\nu^N_1(\bm{v}^N) = \lambda^N\nu^{N\prime}_1(\bm{u}^N),
	\end{align*}
	which illustrates that
	\begin{align*}
		\widetilde{\nu}^N(\bm{u}^N ,\lambda) = \widetilde{\nu}^N_1(\bm{u}^N)\widetilde{\nu}^N_2(\lambda) \in \widetilde{\mathcal{A}}^{*}, \quad
		\nu^N(\bm{v}^N, \lambda) = \nu^N_1(\bm{v}^N)\nu^N_2(\lambda) \in \mathcal{A}^{*}.
	\end{align*}
	
	As a result, for all $\widetilde{\nu}^N(\bm{u}^N ,\lambda) \in \widetilde{\mathcal{A}}$ and $\nu^N(\bm{v}^N, \lambda) \in \mathcal{A}$, due to the measures $\widetilde{\nu}^N_1(d\bm{u}^N)$, and $\nu^N_1(d\bm{v}^N)$ are Gaussians, we have 
	\begin{align*}
		\widetilde{\nu}^N(\bm{u}^N ,\lambda) \in \widetilde{\mathcal{A}}^{*}, \quad 
		\nu^N(\bm{v}^N, \lambda) \in \mathcal{A}^{*}.
	\end{align*}
	That is, $\widetilde{\mathcal{A}}\subset \widetilde{\mathcal{A}}^{*}, \mathcal{A} \subset \mathcal{A}^{*}$.
	Together with $(\ref{conclu:Astar})$, this indicates that
	\begin{align*}
		\min_{\widetilde{\nu}^N \in \widetilde{\mathcal{A}}} D_{{\text{KL}}} (\widetilde{\nu}^N\Vert \widetilde{\mu}^N) = \min_{\nu^N \in \mathcal{A}} D_{{\text{KL}}} (\nu^N \Vert \mu^N).
	\end{align*}

	Then we have
	\begin{align*}
		\min_{\widetilde{\nu}^N \in \widetilde{\mathcal{A}}} D_{{\text{KL}}} (\widetilde{\nu}^N\Vert \widetilde{\mu}^N) 
		= &\int_{\mathcal{H}^N\times \mathbb{R}} \log \frac{\lambda^N\widetilde{\nu}^{\dagger}_1(\bm{u}^N)\widetilde{\nu}^{\dagger}_2(\lambda)}{\mu^N(\bm{v}^N, \lambda)}\lambda^N\widetilde{\nu}^{\dagger}_1(\bm{u}^N)\widetilde{\nu}^{\dagger}_2(\lambda)d\bm{v}^Nd\lambda \\
		= &\int_{\mathcal{H}^N\times \mathbb{R}} \log \frac{\nu^{\dagger}_1(\bm{v}^N)\nu^{\dagger}_2(\lambda)}{\mu^N(\bm{v}^N, \lambda)}\nu^{\dagger}_1(\bm{v}^N)\nu^{\dagger}_2(\lambda)d\bm{v}^Nd\lambda \\
		= &\min_{\nu^N \in \mathcal{A}} D_{{\text{KL}}} (\nu^N \Vert \mu^N).
	\end{align*}
	This indicates that the probability densities $\lambda^N\widetilde{\nu}^{\dagger}_1(\bm{u}^N)\widetilde{\nu}^{\dagger}_2(\lambda)$ and $\nu^{\dagger}_1(\bm{v}^N)\nu^{\dagger}_2(\lambda)$ are both one of the minimized points of the problems $(\ref{eq:vincfin})$.
	Using the same strategy, the probability densities $\lambda^{-N}\nu^{\dagger}_1(\bm{v}^N)\nu^{\dagger}_2(\lambda)$ and $\widetilde{\nu}^{\dagger}_1(\bm{u}^N)\widetilde{\nu}^{\dagger}_2(\lambda)$ are both one of the minimized points of problem $(\ref{eq:vicafinn})$.
	Combining these two conclusions, we obtain that:
	$\lambda^{-N}\nu^{\dagger}(\bm{v}^N, \lambda)$ and $\lambda^N\widetilde{\nu}^{\dagger}(\bm{u}^N, \lambda)$ are the minimized points of problems $(\ref{eq:vicafinn})$ and $(\ref{eq:vincfin})$, respectively, which can be transformed by $\nu^{\dagger}(\bm{v}^N, \lambda)$ and $\widetilde{\nu}^{\dagger}(\bm{u}^N, \lambda)$.
	That is, the minimized points of problems $(\ref{eq:vicafinn})$ and $(\ref{eq:vincfin})$ can be transformed into each other by taking the stated transformation.
	Here we complete the proof.
\end{proof}

We accomplish those two goals mentioned at the beginning of this subsection with the help of Theorem $\ref{the:minimal}$. 
As mentioned in the studies of \cite{bishop2006pattern, jia2021variational, Pinski2015SIAMMA}, 
we cannot expect that the two VI problems $(\ref{eq:vicafin})$ and $(\ref{eq:vincfin})$ have a unique solution. 
That is to say, the results stated in Theorem \ref{the:minimal} cannot ensure that all of the solutions to the two minimization problems can be transformed into each other. In this sense, the conclusions of Theorem \ref{the:minimal} are the best results we can hope for. 

We cannot employ the iMFVI theory under the CP formulation in the infinite-dimensional space due to the singularity issue recalled in Subsection \ref{subsec2.1}. 
In Subsections \ref{subsec2.2} and \ref{subsec2.3}, we illustrate that iMFVI theory can be applied under the NCP formulation without the critical singularity problems. 
In the finite-dimensional setting, the iMFVI theory (reduced to the corresponding finite-dimensional theory) can be constructed under both circumstances and, in addition, are intimately related to each other, as demonstrated in Theorem \ref{the:minimal}. 
As a result, it is natural to use NCP formulation to solve the hierarchical inference problem in the infinite-dimensional case. 

\subsection{Iterative algorithm}\label{subsec2.5}
We construct an iterative algorithm using Theorem $\ref{the:posterior}$, when the parameter $v$ belongs to the separable Hilbert space $\mathcal{H}_u$.
Before stating our algorithm, we need to introduce two operator norms and verify some properties of the operator $H^{*}H$.

Following \cite{reed2012methods}, let us assume that the operator $T_1$ is of trace-class, and the operator $T_2$ is a Hilbert-Schmidt operator, both of them are defined on a Hilbert space $\mathcal{H}_0$.
Then we define the operator norms:
\begin{align*}
	\lVert T_1\rVert_{\text{Tr}} = \text{Tr}\sqrt{T^{*}_1T_1} < \infty, \quad \lVert T_2\rVert_{\text{HS}} = \sqrt{\text{Tr}(T^{*}_2T_2)} < \infty,
\end{align*}
where $T^{*}_1, T^{*}_2$ represent the adjoint operators of $T_1$ and $T_2$, respectively. 
We denote $\text{Tr}(\cdot)$ by taking the trace of the operator.
Then we introduce some properties.

\begin{lemma}\label{lemma1}
	Since the operator $H:\mathcal{H}_u \rightarrow \mathbb{R}^{N_d}$ is a bounded linear operator, we have
	\begin{enumerate}
		\item $H^{*}H$ is a Hilbert-Schmidt operator.
		\item There exists a complete orthonormal basis $\lbrace \phi_k\rbrace^{\infty}_{k=1}$ for $\mathcal{H}_u$, which satisfies $H^{*}H\phi_k = \xi_k\phi_k$, and $\xi_k$ is the $k$-th eigenvalue of $H^{*}H$. 
		\item The Hilbert-Schmidt operator norm of $H^{*}H$ can be controlled by 
		\begin{align*}
			\lVert H^{*}H \rVert_{\text{HS}} \leq \sqrt{\lVert HH^{*} \rVert_{\text{Op}}}\lVert H^{*} \rVert_{\text{HS}},
		\end{align*}
		where $\lVert\cdot \rVert_{\text{Op}}$ denotes the operator norm.
	\end{enumerate}
\end{lemma}

\begin{proof}
	According to the definition of the Hilbert-Schmidt operator in Appendix C in \cite{da2014stochastic}, we have $\lVert H \rVert_{\text{HS}} = \lVert H^{*} \rVert_{\text{HS}} = \sum^{N_d}_{k=1}\lVert H^{*}f_k \rVert^2_{\mathcal{H}_u} < \infty$, where $\lbrace f_k \rbrace^{N_d}_{k=1}$ denotes the complete orthonormal bases defined on $\mathbb{R}^{N_d}$.
	The operator $H$ and $H^{*}$ are both the Hilbert-Schmidt operator.
	Based on Proposition C.4 in \cite{da2014stochastic}, operator $H^{*}H$ is of trace-class defined on $\mathcal{H}_u$, which is exactly a Hilbert-Schmidt operator.
	
	According to Theorem VI.22 \cite{reed2012methods}, we can precisely know that $H^{*}H$ is compact.
	Combining that $H^{*}H$ is self-adjoint, we can obtain the first conclusion immediately according to the Hilbert-Schmidt theorem (Theorem VI.16 \cite{reed2012methods}).
	
	For notational convenience, we will use the symbol $\langle \cdot, \cdot \rangle$ to denote the canonical inner product $\langle \cdot, \cdot \rangle_{\mathbb{R}^{N_d}}$ defined on the space $\mathbb{R}^{N_d}$.
	Taking $\lVert H^{*}H \rVert^2_{HS} = \text{Tr}(H^{*}HH^{*}H)$, for all $x \in \mathbb{R}^{N_d}$, we notice that
	\begin{align*}
		\langle HH^{*}x, x\rangle_{\mathbb{R}^{N_d}} \leq \lVert x\rVert \lVert HH^{*}x\rVert \leq \lVert x\rVert \lVert HH^{*}\rVert_{\text{Op}} \lVert x\rVert = \lVert HH^{*}\rVert_{\text{Op}} \langle x, x\rangle.
	\end{align*}
	That is $HH^{*} \leq \lVert HH^{*}\rVert_{\text{Op}}\text{I}$, where $\text{I}$ denotes the identity operator.	
	And for all $x \in \mathcal{H}_u$, it is clear that
	\begin{align*}
		\langle H^{*}HH^{*}Hx, x\rangle \leq \langle HH^{*}Hx, Hx\rangle \leq \lVert HH^{*}\rVert_{\text{Op}} \langle Hx, Hx\rangle.
	\end{align*}
	We derive that $H^{*}HH^{*}H \leq \lVert HH^{*}\rVert_{\text{Op}} H^{*}H$.
	For the term $\text{Tr}(H^{*}HH^{*}H)$, we have
	\begin{align*}
		\text{Tr}(H^{*}HH^{*}H) \leq \lVert HH^{*} \rVert_{\text{Op}} \text{Tr}(H^{*}H) = \lVert HH^{*} \rVert_{\text{Op}}\lVert H \rVert^2_{\text{HS}}.
	\end{align*}
	Then we derive
	\begin{align*}
		\lVert H^{*}H \rVert_{\text{HS}} \leq \sqrt{\lVert HH^{*} \rVert_{\text{Op}}}\lVert H^{*} \rVert_{\text{HS}}.
	\end{align*}
\end{proof}

\begin{remark}
	\itshape
	Following Lemma $\ref{lemma1}$, once we assume the operator $H:\mathcal{H}_u \rightarrow \mathbb{R}^{N_d}$ is a bounded linear operator, the conclusions can be obtained without any additional assumptions.
	However, if the range of $H$ changes to an infinite-dimensional space, the conclusions in Lemma $\ref{lemma1}$ still hold when the operator $H^{*}H$ is assumed to be a Hilbert-Schmidt operator, see \cite{engl1996regularization}.
\end{remark}

\textbf{Calculate $\Phi_v(v)$}: 
Applying the formula $(\ref{PotenFun1})$ in Theorem $\ref{the:posterior}$, we obtain
\begin{align*}
 \Phi_v(v) &= \int_{\mathbb{R}} \bigg(\frac{1}{2} \lVert \bm{d}-H\lambda v \rVert ^2_{\bm{\Gamma}_{\text{noise}}} \bigg)\nu^{\lambda}(d\lambda)+\text{Const} \\
 &= {\frac{1}{2} \bigg(\mathcal{C}_{\lambda} \lVert Hv\rVert^2_{\Gamma_{\text{noise}}} + (\lambda^{*})^2\lVert Hv\rVert^2_{\Gamma_{\text{noise}}} - 2 \langle\lambda^{*}Hv, \bm{d} \rangle_{\Gamma_{\text{noise}}} \bigg)+\text{Const}},
\end{align*}
where
\begin{align*}
 \lambda^{*} = \mathbb{E}^{\nu^{\lambda}}[\lambda] = \int_{\mathbb{R}}\lambda\nu^{\lambda}(d\lambda), \quad
 \mathcal{C}_{\lambda} = \mathbb{E}^{\nu^{\lambda}}[\lambda-\lambda^{*}]^2 = 
 \int_{\mathbb{R}} (\lambda-\lambda^{*})^2\nu^{\lambda}(d\lambda).
\end{align*}
We derive that
\begin{align*}
 {\frac{d\nu^v}{d\nu^v_0} \varpropto \exp \bigg(-\frac{1}{2} \bigg(\mathcal{C}_{\lambda} \lVert Hv\rVert^2_{\Gamma_{\text{noise}}} + (\lambda^{*})^2\lVert Hv\rVert^2_{\Gamma_{\text{noise}}} - 2\langle\lambda^{*}Hv, \bm{d}\rangle_{\Gamma_{\text{noise}}} \bigg) \bigg)}.
\end{align*}
On the basis of Bayes' formula, the probability measure $\nu^v = \mathcal{N}(v^{*}, \mathcal{C}_v)$ is Gaussian with
\begin{align}\label{post:v}
 { \mathcal{C}^{-1}_v=(\mathcal{C}_{\lambda}+(\lambda^{*})^2 )H^{*}\Gamma^{-1}_{\text{noise}}H+\mathcal{C}^{-1}_0, \quad
 v^{*}=\mathcal{C}_{v}(\lambda^{*}H^{*}\Gamma^{-1}_{\text{noise}}\bm{d})}.
\end{align}

\textbf{Calculate $\Phi_{\lambda}(\lambda)$}: 
Similarly, we have
\begin{align*}
 \Phi_{\lambda}(\lambda) = &\int_{\mathcal{H}_u}\frac{1}{2} \bigg( \lVert \bm{d}-H\lambda v\rVert ^2_{\bm{\Gamma}_{\text{noise}}} \bigg ) \nu^{v}(dv)+\text{Const} \\
 = &{\frac{\tau}{2}\int_{\mathcal{H}_u} \lVert \bm{d}-\lambda Hv^{*}\rVert ^2_{\Gamma_{\text{noise}}} + \lVert H(v-v^{*}) \rVert ^2_{\Gamma_{\text{noise}}} \lambda^2} \\
 &+ { 2\lambda \langle \bm{d}-\lambda Hv^{*},H(v-v^{*}) \rangle_{\Gamma_{\text{noise}}} \nu^vdv+\text{Const}}.
\end{align*}

Then we derive that
{
\begin{align}\label{eq:tr1}\nonumber
	\int_{\mathcal{H}_u} \lVert H(v-v^{*})\rVert^{2}_{\Gamma_{\text{noise}}}\nu^{v}(dv) 
	= &\int_{\mathcal{H}_u} \bigg \langle H(v-v^{*}), H(v-v^{*}) \bigg \rangle_{\Gamma_{\text{noise}}}\nu^{v}(dv) \nonumber \\
	= &\int_{\mathcal{H}_u} \bigg \langle H\sum^{\infty}_{k=1}(v_k - v^{*}_k)\phi_k, 
	H\sum^{\infty}_{j=1}(v_j - v^{*}_j)\phi_j \bigg \rangle_{\Gamma_{\text{noise}}} \nu^{v}(dv) \nonumber\\
	= &\sum^{\infty}_{k=1}\sum^{\infty}_{j=1}\int_{\mathcal{H}_u} \bigg \langle H(v_k - v^{*}_k)\phi_k, H(v_j - v^{*}_j)\phi_j \bigg \rangle_{\Gamma_{\text{noise}}} \nu^{v}(dv) \\
	= &\sum^{\infty}_{k=1}\sum^{\infty}_{j=1} \int_{\mathcal{H}_u}  (v_k - v^{*}_k)(v_j - v^{*}_j) \langle H^{*}H\phi_k, \phi_j \rangle_{\Gamma_{\text{noise}}}  \nu^{v}(dv) \nonumber.
\end{align}
}
Using the first conclusion of Lemma $\ref{lemma1}$, we have
{
\begin{align*}
	\sum^{\infty}_{k=1}\sum^{\infty}_{j=1}\langle H^{*}H\phi_k, \phi_j \rangle_{\Gamma_{\text{noise}}} 
	= \tau\sum^{\infty}_{k=1}\sum^{\infty}_{j=1}\langle \xi_k\phi_k, \phi_j \rangle 
	= \tau\sum^{\infty}_{k=1}\langle \xi_k\phi_k, \phi_k \rangle
	= \sum^{\infty}_{k=1}\tau\xi_k.
\end{align*}
}
Then equation $(\ref{eq:tr1})$ turns to
{
\begin{align*}
	\int_{\mathcal{H}_u} \lVert H(v-v^{*})\rVert^{2}_{\Gamma_{\text{noise}}}\nu^{v}(dv) 
	= &\sum^{\infty}_{k=1} \int_{\mathcal{H}_u}  \tau\xi_k(v_k - v^{*}_k)(v_k - v^{*}_k) \nu^{v}(dv)\\
	= &\sum^{\infty}_{k=1} \langle \tau\xi_k\phi_k, \mathcal{C}_v\phi_k \rangle = \sum^{\infty}_{k=1} \langle H^{*}\Gamma^{-1}_{\text{noise}}H\phi_k, \mathcal{C}_v\phi_k \rangle \\
	= &\sum^{\infty}_{k=1} \langle \mathcal{C}_vH^{*}\Gamma^{-1}_{\text{noise}}H\phi_k, \phi_k \rangle = \text{Tr}(\mathcal{C}_{v}H^{*}\Gamma^{-1}_{\text{noise}}H).
\end{align*}
}
In the last line, we use the property that operator $\mathcal{C}_v$ is self-adjoint.
According to Theorem VI.3 \cite{reed2012methods}, since $\mathcal{C}^{-1}_v$ is self-adjoint, we know that $\mathcal{C}_v$ is also self-adjoint.
Then the function $\Phi_{\lambda}(\lambda)$ can be written as 
{
\begin{align*}
 \Phi_{\lambda}(\lambda) = \frac{1}{2}\lVert \bm{d}-\lambda Hv^{*}\rVert^2_{\Gamma_{\text{noise}}}+\frac{1}{2}\text{Tr}(\mathcal{C}_vH^{*}\Gamma^{-1}_{\text{noise}}H)\lambda^2 + \text{Const},
\end{align*}
where
\begin{align*}
 v^{*} = \mathbb{E}^{\nu^{v}}[v] = \int_{\mathcal{H}_u}v\nu^{v}(dv), \quad
 \mathcal{C}_{v} = \int_{\mathcal{H}_u} (v - v^{*})\otimes (v - v^{*}) \nu^{v}(dv).
\end{align*}
This implies
\begin{align*}
 \frac{d\nu^{\lambda}}{d\nu^{\lambda}_0} \varpropto \exp \bigg (-\frac{1}{2} \bigg(\text{Tr}(\mathcal{C}_vH^{*}\Gamma^{-1}_{\text{noise}}H)+\lVert Hv^{*}\rVert^{2}_{\Gamma_{\text{noise}}} \bigg)\lambda^2{ + }\lambda v^{*}H^{*}\Gamma^{-1}_{\text{noise}}\bm{d} \bigg).
\end{align*}
Therefore, $\nu^{\lambda} = \mathcal{N}(\lambda^{*}, \mathcal{C}_{\lambda})$ is a Gaussian measure with
\begin{align}\label{post:lam}
 \mathcal{C}^{-1}_{\lambda}=\bigg (\text{Tr}(\mathcal{C}_vH^{*}\Gamma^{-1}_{\text{noise}}H)+\lVert Hv^{*}\rVert^{2}_{\Gamma_{\text{noise}}} \bigg )+\frac{1}{\sigma}, \quad
 \lambda^{*}=\mathcal{C}_{\lambda}\bigg ( v^{*}H^{*}\Gamma^{-1}_{\text{noise}}\bm{d}+\frac{1}{\sigma}\bar{\lambda}\bigg ).
\end{align}
}

We notice that the key point of calculating the posterior of $\lambda$ is to calculate {$\text{Tr}(\mathcal{C}_vH^{*}\Gamma^{-1}_{\text{noise}}H)$} efficiently.
Based on \cite{bui2012analysis}, we propose a theorem to transform the original form of {$\text{Tr}(\mathcal{C}_vH^{*}\Gamma^{-1}_{\text{noise}}H)$}, which is
{
\begin{align}\label{eq:tr}
	\begin{split}
		\text{Tr} \bigg(\mathcal{C}_vH^{*}\Gamma^{-1}_{\text{noise}}H \bigg) &= \text{Tr} \bigg( \bigg((\mathcal{C}_{\lambda}+(\lambda^{*})^2)H^{*}\Gamma^{-1}_{\text{noise}}H+\mathcal{C}^{-1}_0 \bigg)^{-1}H^{*}\Gamma^{-1}_{\text{noise}}H \bigg) \\
		&= \text{Tr} \bigg(\mathcal{C}^{1/2}_0 \bigg((\mathcal{C}_{\lambda}+(\lambda^{*})^2)\mathcal{C}^{1/2}_0H^{*}\Gamma^{-1}_{\text{noise}}H\mathcal{C}^{1/2}_0+\text{Id} \bigg)^{-1}\mathcal{C}^{1/2}_0H^{*}\Gamma^{-1}_{\text{noise}}H \bigg),
	\end{split}
\end{align}
}into a form that can be calculated conveniently.
And we shall discuss the advantage of the form stated in the following theorem more precisely in Subsection $\ref{subsec2.6}$.

\begin{theorem}\label{the:caltr}
	Let the operator $H^{*}H$ be a Hilbert-Schmidt operator defined on the Hilbert space $\mathcal{H}_u$, then we have
	{
	\begin{align*}
		\text{Tr} (\mathcal{C}_vH^{*}\Gamma^{-1}_{\text{noise}}H) = \text{Tr} \bigg( \bigg((\mathcal{C}_{\lambda}+(\lambda^{*})^2)\mathcal{C}^{1/2}_0H^{*}\Gamma^{-1}_{\text{noise}}H\mathcal{C}^{1/2}_0+\text{Id} \bigg)^{-1}\mathcal{C}^{1/2}_0H^{*}\Gamma^{-1}_{\text{noise}}H\mathcal{C}^{1/2}_0 \bigg) < \infty.
	\end{align*}
	}
\end{theorem} 

\begin{proof}
	Let us denote
	{
	\begin{align*}
		\mathcal{A} := \bigg((\mathcal{C}_{\lambda}+(\lambda^{*})^2)\mathcal{C}^{1/2}_0H^{*}\Gamma^{-1}_{\text{noise}}H\mathcal{C}^{1/2}_0+\text{Id} \bigg)^{-1}\negmedspace\mathcal{C}^{1/2}_0H^{*}\Gamma^{-1}_{\text{noise}}H,
		\quad \mathcal{B} := \mathcal{C}^{1/2}_0,
	\end{align*}
	}
	then we have
	\begin{align*}
		{\text{Tr} (\mathcal{C}_vH^{*}\Gamma^{-1}_{\text{noise}}H )} = \text{Tr} (\mathcal{B}\mathcal{A}).
	\end{align*}
	We aim to prove $\text{Tr} (\mathcal{B}\mathcal{A}) = \text{Tr} (\mathcal{A}\mathcal{B})$, that is
	{
		\begin{align*}
			\text{Tr} (\mathcal{C}_vH^{*}\Gamma^{-1}_{\text{noise}}H) = \text{Tr} \bigg( \bigg((\mathcal{C}_{\lambda}+(\lambda^{*})^2)\mathcal{C}^{1/2}_0H^{*}\Gamma^{-1}_{\text{noise}}H\mathcal{C}^{1/2}_0+\text{Id} \bigg)^{-1}\mathcal{C}^{1/2}_0H^{*}\Gamma^{-1}_{\text{noise}}H\mathcal{C}^{1/2}_0 \bigg).
	\end{align*}
	}
	According to Theorem VI.25(c) in \cite{reed2012methods}, we need to demonstrate that the operator $\mathcal{B}$ is bounded and $\mathcal{A}$ is of trace-class. 
	Since the operator $\mathcal{C}_0$ is assumed to be trace-class, we know that $\mathcal{B} = \mathcal{C}^{1/2}_0$ is a Hilbert-Schmidt operator.
	This demonstrates that operator $\mathcal{B}$ is bounded. 
	
	Our task now is to prove that $\mathcal{A}$ is of trace-class.
	Setting
	{
	\begin{align*}
		\mathcal{P} = \mathcal{C}^{1/2}_0H^{*}\Gamma^{-1}_{\text{noise}}H, \quad
		\mathcal{Q} = \bigg((\mathcal{C}_{\lambda}+(\lambda^{*})^2)\mathcal{C}^{1/2}_0H^{*}\Gamma^{-1}_{\text{noise}}H\mathcal{C}^{1/2}_0+\text{Id} \bigg)^{-1},
	\end{align*}
	}
	we know that $\mathcal{A} = \mathcal{Q}\mathcal{P}$.
	According to Theorem VI.19(b) in \cite{reed2012methods}, we need to prove { $\mathcal{P}= \mathcal{C}^{1/2}_0H^{*}\Gamma^{-1}_{\text{noise}}H$} is of trace-class, i.e., $\lVert \mathcal{P} \rVert_{\text{Tr}} < \infty $ ,and $\mathcal{Q}$ is a bounded linear operator.
	Noticing the property of $\mathcal{C}^{1/2}_0$, $\lVert \mathcal{C}^{1/2}_0\rVert_{\text{HS}} < \infty$ is clear. 
	Since $ H^{*}H$ is a Hilbert-Schmidt operator, we have
	{
	\begin{align}
		\lVert \mathcal{C}^{1/2}_0H^{*}\Gamma^{-1}_{\text{noise}}H\rVert_{\text{Tr}} \leqslant \tau\lVert \mathcal{C}^{1/2}_0\rVert_{\text{HS}}\lVert H^{*}H\rVert_{\text{HS}} < \infty,
	\end{align}
	}
	where we used the conclusion stated in Chapter VI \cite{reed2012methods}.
	This claims that { $\mathcal{P}= \mathcal{C}^{1/2}_0H^{*}\Gamma^{-1}_{\text{noise}}H$} is a trace-class operator. 
	
	Then we need to illustrate that the operator $\mathcal{Q}$ is bounded.
	Without loss of generality, we first show that $\mathcal{Q}^{-1}$ is injective:
	\begin{align*}
		\langle \mathcal{Q}^{-1}u, u \rangle_{\mathcal{H}_u} &= { \bigg \langle((\mathcal{C}_{\lambda}+(\lambda^{*})^2)\mathcal{C}^{1/2}_0H^{*}\Gamma^{-1}_{\text{noise}}H\mathcal{C}^{1/2}_0+\text{Id})u, u \bigg \rangle_{\mathcal{H}_u}} \\
		&= \tau(\mathcal{C}_{\lambda}+(\lambda^{*})^2) \langle H\mathcal{C}^{1/2}_0u, H\mathcal{C}^{1/2}_0u \rangle_{\mathcal{H}_u} + \langle u, u \rangle_{\mathcal{H}_u} \\
		&\geqslant \lVert u \rVert^{2}_{\mathcal{H}_u}.
	\end{align*}
	Hence we have $0 = \langle \mathcal{Q}^{-1}u, u\rangle_{\mathcal{H}_u} \geqslant \lVert u \rVert^{2}_{\mathcal{H}_u} \geqslant 0$, which indicates that the operator $\mathcal{Q}^{-1}$ is injective.
	To illustrate that $\mathcal{Q}^{-1}$ is surjective, we need to show that $\text{Ran}(\mathcal{Q}^{-1})$ is closed, and $\text{Ran}(\mathcal{Q}^{-1})^{\perp} = 0$.
	For all $y \in \overline{\text{Ran}(\mathcal{Q}^{-1})}$, there exists $\lbrace u_k \rbrace^{\infty}_{k=1} \subset \mathcal{H}_u$, such that $y = \lim_{k \rightarrow \infty}\mathcal{Q}^{-1}u_k$.
	Since
	\begin{align*}
		\lim_{k \rightarrow \infty}\lVert u_k - u_{k-1} \rVert^{2}_{\mathcal{H}_u} &\leq \lim_{k \rightarrow \infty}\langle \mathcal{Q}^{-1}(u_k - u_{k-1}), (u_k - u_{k-1}) \rangle_{\mathcal{H}_u} \\
		&\leq \lim_{k \rightarrow \infty}\lVert u_k - u_{k-1} \rVert^{2}_{\mathcal{H}_u}\lVert \mathcal{Q}^{-1}(u_k - u_{k-1}) \rVert^2_{\mathcal{H}_u} \\
		&= 0,
	\end{align*}
	we know that $\lbrace u_k \rbrace^{\infty}_{k=1}$ is a Cauchy sequence, that is, $\text{Ran}(\mathcal{Q}^{-1})$ is closed.
	For any $y \in \overline{\text{Ran}(\mathcal{Q}^{-1})}^{\perp}$, we assume that $\langle y, \mathcal{Q}^{-1}u \rangle = 0$, for all $u \in \mathcal{H}_u$.
	Let $y = u$, we obtain
	\begin{align*}
		\lVert u \rVert^{2}_{\mathcal{H}_u} \leq \langle \mathcal{Q}^{-1}u, u \rangle_{\mathcal{H}_u} = 0,
	\end{align*}
	which indicates $\overline{\text{Ran}(\mathcal{Q}^{-1})}^{\perp} = 0$.
	Then we derive that $\mathcal{Q}^{-1}$ is surjective.
	Since $\mathcal{Q}^{-1}$ is bounded, according to the inverse mapping theorem (Theorem III.11 in \cite{reed2012methods}), operator $\mathcal{Q}$ is a bounded linear operator defined on $\mathcal{H}_u$.
	That is, operator $\mathcal{A}$ is of trace-class.
	Here we complete the proof.
\end{proof}

It is worth mentioning that the infinite-dimensional formulation provides a general framework for conducting appropriate discretizations. 
Combining Theorem $\ref{the:posterior}$ and the discussion in this section, we can provide an iterative algorithm to calculate the posterior measures explicitly, see Algorithm $\ref{alg A}$.
In order to keep dimensional independence, the discretization of Algorithm $\ref{alg A}$ should be done carefully, e.g., the adjoint operator $H^{*}$ is usually not trivially equal to the transpose of the discrete approximation of $H$, see \cite{bui2012analysis, jia2021stein}.

\begin{algorithm}
\caption{The NCP-iMFVI algorithm}
\label{alg A}
\begin{algorithmic}[1]
\STATE{Initialize $\lambda_0 = \bar{\lambda}$, specify the tolerance $tol$, and set $k=1$;}
\REPEAT
\STATE{{ Calculate estimated posterior measure $\nu^v_k = \mathcal{N}(v^{*}_k, \mathcal{C}_{v_k})$, \\
where $\mathcal{C}^{-1}_{v_k} =  (\mathcal{C}^{k-1}_{\lambda}+\lambda^2_{k-1})H^{*}\Gamma^{-1}_{\text{noise}}H + \mathcal{C}^{-1}_{0}, \quad v^{*}_k = \mathcal{C}_{v_k}\lambda_{k-1}H^{*}\Gamma^{-1}_{\text{noise}}d$;}}
\STATE{{ Calculate $\text{Tr}(\mathcal{C}_{v_k}H^{*}\Gamma^{-1}_{\text{noise}}H)$;}}
\STATE{{Calculate estimated posterior measure $\nu^{\lambda}_k = \mathcal{N}(\lambda^{*}_k, \mathcal{C}_{\lambda_k})$, \\
where $\mathcal{C}^{-1}_{\lambda_k} = \bigg(\text{Tr}(\mathcal{C}_{v_k}H^{*}\Gamma^{-1}_{\text{noise}}H)+\lVert Hv_k\rVert^2_{\Gamma_{\text{noise}}} \bigg)+\frac{1}{\sigma}, 
\quad \lambda^{*}_k=\mathcal{C}_{\lambda_k}(d^{T}\Gamma^{-1}_{\text{noise}}Hv_k + \frac{\bar{\lambda}}{\sigma}$);}}
\UNTIL{$\max(\lVert \lambda_kv_k - \lambda_{k-1}v_{k-1} \rVert / \lVert \lambda_{k}v_{k}\rVert, 
	\lVert \lambda_{k} - \lambda_{k-1} \rVert / \lVert \lambda_{k-1}\rVert) \leqslant tol $;}
\STATE{Return the approximate probability measure $\nu(dv, d\lambda) = \nu^v_k(dv)\nu^{\lambda}_k(d\lambda)$.}
\end{algorithmic}
\end{algorithm}

\subsection{Some numerical details}\label{subsec2.6}
According to Theorem $\ref{the:caltr}$, we have a form of {$\text{Tr} (\mathcal{C}_vH^{*}\Gamma^{-1}_{\text{noise}}H)$}, which can be calculated more conveniently.
Based on \cite{bui2013computational}, with the help of a low rank approximation of operators, we provide an efficient method for calculating the trace of {$\mathcal{C}_vH^{*}\Gamma^{-1}_{\text{noise}}H$}.

We consider a finite-dimensional subspace $V_h$ of $L^2(\Omega)$ originating from a finite-element discretization with continuous Lagrange basis functions $\lbrace \phi_j \rbrace^{n}_{j=1}$, which correspond to the nodal points $\lbrace \bm{x}_j \rbrace^{n}_{j=1}$, such that
\begin{align*}
	\phi_j(\bm{x}_i) = \delta_{ij}, \quad \text{for} \ i, j \in \lbrace 1, \cdots, n\rbrace,
\end{align*}
where $\Omega$ is the domain of some specific problem.
For all $m \in L^2(\Omega)$, we define the approximation $m_h = \sum^{n}_{j=1}m_j\phi_j \in V_h$.
Then, we introduce the mass matrix $\bm{M}$ similar to \cite{bui2012analysis}.
For any $m_1, m_2 \in L^2(\Omega)$, observe that $\langle m_1, m_2 \rangle_{L^2(\Omega)} \approx \langle m_{1h}, m_{2h} \rangle = \langle \bm{m}_1, \bm{m}_2 \rangle_{\bm{M}} := \bm{m}^T_1\bm{M}\bm{m}_2$, where $\bm{M}$ is defined by
\begin{align*}
	\bm{M}_{ij} = \int_{\Omega}\phi_i(\bm{x})\phi_j(\bm{x})d\bm{x}.
\end{align*}
For simplicity of notation, we shall use the boldface symbol to denote the matrix representation of the operators.

Define ${\rho := \mathcal{C}_{\lambda}+(\lambda^{*})^2}$ is a constant.
Based on the Theorem $\ref{the:caltr}$, setting
\begin{align*}
 { \bm{\mathcal{G}} := \bm{\mathcal{C}}^{1/2}_0\bm{H^{*}\Gamma^{-1}_{\text{noise}}H\mathcal{C}}^{1/2}_0},
\end{align*}
we have
{
\begin{align*}
 \text{Tr} \bigg(\bm{\mathcal{C}_vH^{*}\Gamma^{-1}_{\text{noise}}H} \bigg) = \text{Tr} \bigg( (\rho\bm{\mathcal{G}}+\bm{I})^{-1}\bm{\mathcal{G}} \bigg).
\end{align*}
}
Let $\lbrace \xi_i, \bm{v}_i\rbrace^{n}_{i=1}$ be the eigenpairs of $\bm{\mathcal{G}}$, and $\bm{\Xi}=\diag(\xi_1, \cdots, \xi_n) \in \mathbb{R}^{n\times n}$, then denote $\bm{V} \in \mathbb{R}^{n\times n}$ such that its columns are the eigenvectors $\bm{v}_i$ of $\bm{\mathcal{G}}$. 
Now we replace $\bm{\mathcal{G}}$ by its spectral decomposition
\begin{align}\label{eq:tr calbm}
 \text{Tr} \bigg(\bm{\mathcal{C}_vH^{*}H} \bigg) = \text{Tr} \bigg((\rho \bm{V\Xi V^{\diamondsuit}+I)}^{-1}\bm{V\Xi V^{\diamondsuit}} \bigg),
\end{align}
where $\bm{V}^{\diamondsuit}$ is the adjoint of $\bm{V}$ defined as $\bm{V}^{\diamondsuit} = \bm{V}^T\bm{M}$, see \cite{bui2013computational}, and $\bm{M}$ is the mass matrix.

Since the eigenvalues of $\bm{\mathcal{G}}$ decay rapidly for many practical inverse problems \cite{bui2013computational}, it is considerable to construct a low-rank approximation of $\bm{\mathcal{G}}$ by computing the $r$ largest eigenvalues
\begin{align*}
 \bm{\mathcal{G}} = \bm{V_r\Xi_r V^{\diamondsuit}_r} + \mathcal{O} \bigg(\sum^{n}_{i=r+1}\xi_i \bigg),
\end{align*}
where $\bm{V_r} \in \mathbb{R}^{n\times r}$ contains $r$ eigenvectors of $\bm{\mathcal{G}}$ corresponding to the $r$ largest eigenvalues, and $\bm{\Xi_r} = \diag(\xi_1, \cdots, \xi_r) \in \mathbb{R}^{r\times r}$. We can use the Sherman-Morrison-Woodbury formula to simplify ($\ref{eq:tr calbm}$), then we have
\begin{align*}
 \bigg(\rho \bm{V\Xi V^{\diamondsuit}+I} \bigg)^{-1} = \bm{I - V_rD_rV^{\diamondsuit}_r} + \mathcal{O} \bigg(\sum^{n}_{i=r+1}\frac{\rho\xi_i}{\rho\xi_i+1} \bigg),
\end{align*}
where 
\begin{align*}
 \bm{D_r} := \diag(d_1, \cdots, d_r) = \diag(\rho\xi_1/(\rho\xi_1+1), \cdots, \rho\xi_r/(\rho\xi_r+1)).
\end{align*}
Thus, we have
\begin{align}
\begin{split}
 \text{Tr} \bigg(\bm{\mathcal{C}_vH^{*}H} \bigg) &= \frac{1}{\rho}\text{Tr} \bigg((\rho \bm{V\Xi V^{\diamondsuit} + I)}^{-1}\rho \bm{V\Xi V^{\diamondsuit}} \bigg) \\
 &= \frac{1}{\rho}\text{Tr} \bigg((\rho \bm{V\Xi V^{\diamondsuit} + I)}^{-1}(\rho \bm{V\Xi V^{\diamondsuit} + I - I)} \bigg) \\
 &= \frac{1}{\rho}\text{Tr} \bigg(\bm{I} - (\rho \bm{V\Xi V^{\diamondsuit} + I)}^{-1} \bigg) \\ 
 &= \frac{1}{\rho}\text{Tr} \bigg(\bm{I - (I - V_rD_rV^{\diamondsuit}_r} 
 + \mathcal{O} (\sum^{n}_{i=r+1}\frac{\rho\xi_i}{\rho\xi_i+1})) \bigg) \\
 &= \frac{1}{\rho}\text{Tr} \bigg(\bm{V_rD_rV^{\diamondsuit}_r} \bigg)
 - {\frac{1}{\rho}}\mathcal{O} \bigg(\sum^{n}_{i=r+1}\frac{{\rho}\xi_i}{\rho\xi_i+1} \bigg).
\end{split}
\end{align}

In order to obtain an accurate approximation of $\text{Tr} (\bm{\mathcal{C}_vH^{*}H})$, following the strategy in \cite{bui2013computational}, we have the following strategy:
{
The tail terms
\begin{align*}
	\sum^{n}_{i=r+1}\frac{\rho\xi_i}{\rho\xi_i+1}
\end{align*}
can be neglected when
\begin{align*}
	\rho\xi_i< 1,
\end{align*}
}

Removing tail terms, we have
\begin{align}
 \text{Tr} \bigg(\bm{\mathcal{C}_vH^{*}H} \bigg) \approx \frac{1}{\rho}\sum^{r}_{i=1}d_i = \sum^{r}_{i=1}\frac{\xi_i}{\rho\xi_i+1}.
\end{align}

\begin{remark}
	\itshape
		Usually, it is {computationally expensive} to calculate the eigenvalues directly if the eigenvalue problem {$\bm{H^{*}\Gamma^{-1}_{\text{noise}}Hv}_i = \tau\xi_i \bm{\mathcal{C}}^{-1}_0\bm{v}_i$} is of large-scale. 
		Matrix {$\bm{H^{*}\Gamma^{-1}_{\text{noise}}H} \in \mathbb{R}^{n \times n}$} is symmetric and positive-defined, $\bm{\mathcal{C}}^{-1}_0 \in \mathbb{R}^{n \times n}$ is symmetric positive defined, and $\lbrace \xi_i, \bm{v}_i \rbrace^{n}_{i=1}$ is the eigensystem.
		Since matrix { $\bm{H^{*}\Gamma^{-1}_{\text{noise}}H}$} is a large dense matrix, it is hardly to store such a matrix in our computational procedure.
		In order to calculate the eigenvalues, we shall employ a randomized SVD algorithm called double pass algorithm \cite{saibaba2016randomized} in which { $\bm{H^{*}\Gamma^{-1}_{\text{noise}}H}$} only operates on a random vector to save the computational resource without giving the explicit form of { $\bm{H^{*}\Gamma^{-1}_{\text{noise}}H}$}.
		This method will significantly reduce the computational cost required for eigenvalue decomposition.
		Under this circumstance, the computational procedure only requires $\mathcal{O}(k)$ matrix vector product, where $k$ denotes the number of eigenvalues we need to calculate.
\end{remark}

\section{Numerical Examples}\label{sec3}
We now present numerical simulations supporting the results in Section $\ref{sec2}$. 
Here we apply the general theory to three typical inverse problems, including two linear and one non-linear problems. 
In Subsection $\ref{subsec3.1}$, we consider a simple smooth model with Gaussian noise; meanwhile, we compare the computational results with the classical sampling method, such as the {Gibbs sampling method}. 
In Subsection $\ref{subsec3.2}$, we solve a large-scale inverse source problem of Helmholtz equation.
In Subsection $\ref{subsec3.3}$, we work with a non-linear inverse problem of Darcy-flow equation. 

\subsection{The simple elliptic equation with Gaussian noise}\label{subsec3.1}
\subsubsection{Basic settings}\label{subsec3.1.1}
Consider an inverse source problem of the elliptic equation
\begin{align}\label{prob1}
\begin{split}
 -\alpha \Delta w + w &= u \quad \text{in}\ \Omega, \\ 
 w &= 0 \quad \text{on}\ \partial \Omega,
\end{split}
\end{align}
where $\Omega = (0, 1) \subset \mathbb{R}$, $\alpha > 0$ is a positive constant. 
The forward operator is defined as follows:
\begin{align}
 Hu = (w(x_1), w(x_2), \cdots, w(x_{N_d}))^T,
\end{align}
where $u \in \mathcal{H}_u := L^2(\Omega)$, $w$ denotes the solution of $(\ref{prob1})$, and $x_i \in \Omega$ for $i = 1, \cdots, N_d$.
With these notations, the problem can be written abstractly as:
\begin{align}
 \bm{d} = Hu + \bm{\epsilon},
\end{align}
where $\bm{\epsilon} \sim \mathcal{N}(0, \bm{\Gamma}_{\text{noise}})$ is the random Gaussian noise.
In our implementations, the measurement points $\lbrace x_i \rbrace^{N_d}_{i=1}$ are taken at the coordinates $\lbrace i/20\rbrace^{20}_{i = 1}$. 
To avoid the inverse crime \cite{kaipio2006statistical}, we discretize the elliptic equation by the finite element method on a regular mesh (the grid points are uniformly distributed on the domain $\Omega$) with the number of grid points being equal to $10^4$. 
In our experiments, the prior measure of $u$ is a Gaussian probability measure $\mu^{u, \lambda}_0$ with mean zero and covariance $\lambda^2\mathcal{C}_0$. 
Without particular mention, the prior measure of the hyper-parameter $\lambda$ is a one-dimensional Gaussian probability measure $\mu^{\lambda}_0$, with { mean $\bar{\lambda} = 1$, and variance $\sigma = 10000$}.

For clarity, we list the specific choices for some parameters introduced in this subsection as follows: 
\begin{itemize}
	\item {Assume that 5$\%$ random Gaussian noise $\epsilon \sim \mathcal{N}(0, \bm{\Gamma}_{\text{noise}})$ is added, where $\bm{\Gamma}_{\text{noise}} = \tau^{-1}\textbf{I}$, and $\tau^{-1} = (0.05\max(\lvert Hu\rvert))^2$.}
	\item Let domain $\Omega$ be an interval $(0, 1)$ with $\partial \Omega = \lbrace 0, 1 \rbrace$. And the available data are assumed to be $\lbrace w(x_i) | i = 1, 2, \cdots, 20 \rbrace$.
	\item We assume that the data produced from the underlying true signal $u^{\dagger}(x) = 10 \cdot (\cos 4\pi x+1)$.
	\item The operator $\mathcal{C}_0$  is given by $\mathcal{C}_0 = (\text{I} - \alpha \Delta)^{-2},$ where $\alpha = 0.05$ is a fixed constant. Here the Laplace operator is defined on $\Omega$ with zero Neumann boundary condition.
	\item In order to avoid inverse crime \cite{kaipio2006statistical}, the data is generated on a fine mesh with the number of grid points equal to $10^4$. And we use different sizes of mesh $n = \lbrace 100, 200, 500, 700, 900 \rbrace$ in the inverse stage.
\end{itemize}

To illustrate the effectiveness of NCP-iMFVI, we compare it with the non-centered pCN within Gibbs sampling method proposed in \cite{chen2018dimension}. 
In the numerical experiment, we intend to compare the posterior mean function and covariance of $u$ obtained by both algorithms, thus showing that posterior distributions are very similar.
We assume that the hyper-parameter $\lambda$ defining the map $u = \lambda v$ with $v \sim \mu^v_0$, and $\lambda \sim \mu^{\lambda}_0$. 
The likelihood thus depends on both $\lambda$ and $v$, and the Bayes' formula takes the form:
\begin{align*}
 \frac{d\mu}{d\mu_0}(v, \lambda) \varpropto \exp (-\Phi(v, \lambda)).
\end{align*}
Because the Gibbs sampling method is a widely used algorithm, we won't discuss this method in this subsection.
This method is left in Appendix and provided as pseudocode.

\begin{figure}
	\centering
	\subfloat[Trace of non-cenerted Gibbs]{
		\includegraphics[ keepaspectratio=true, width=0.40\textwidth, clip=true]{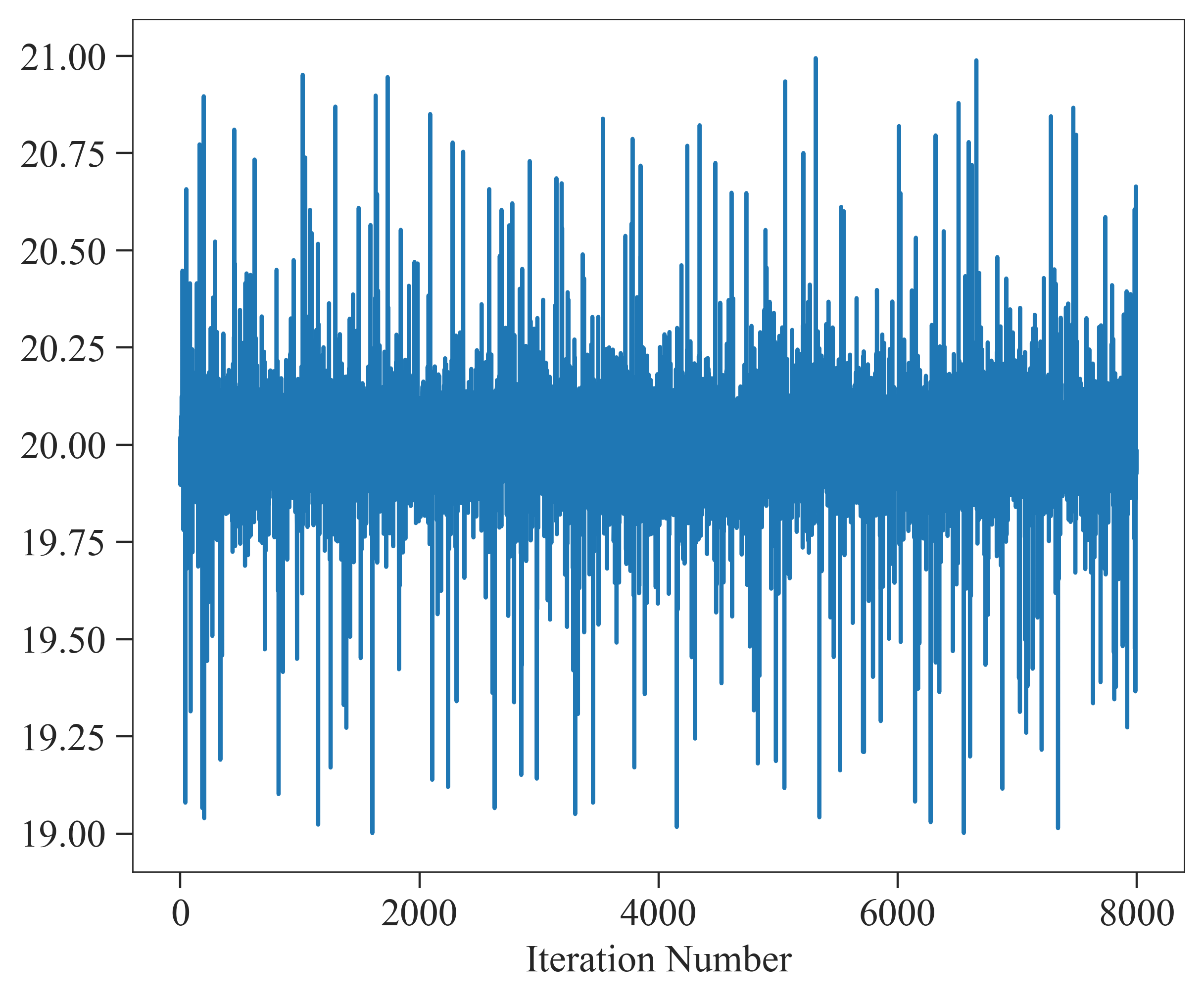}}
	\subfloat[Eigenvalues (Logarithm)]{
		\includegraphics[ keepaspectratio=true, width=0.395\textwidth, clip=true]{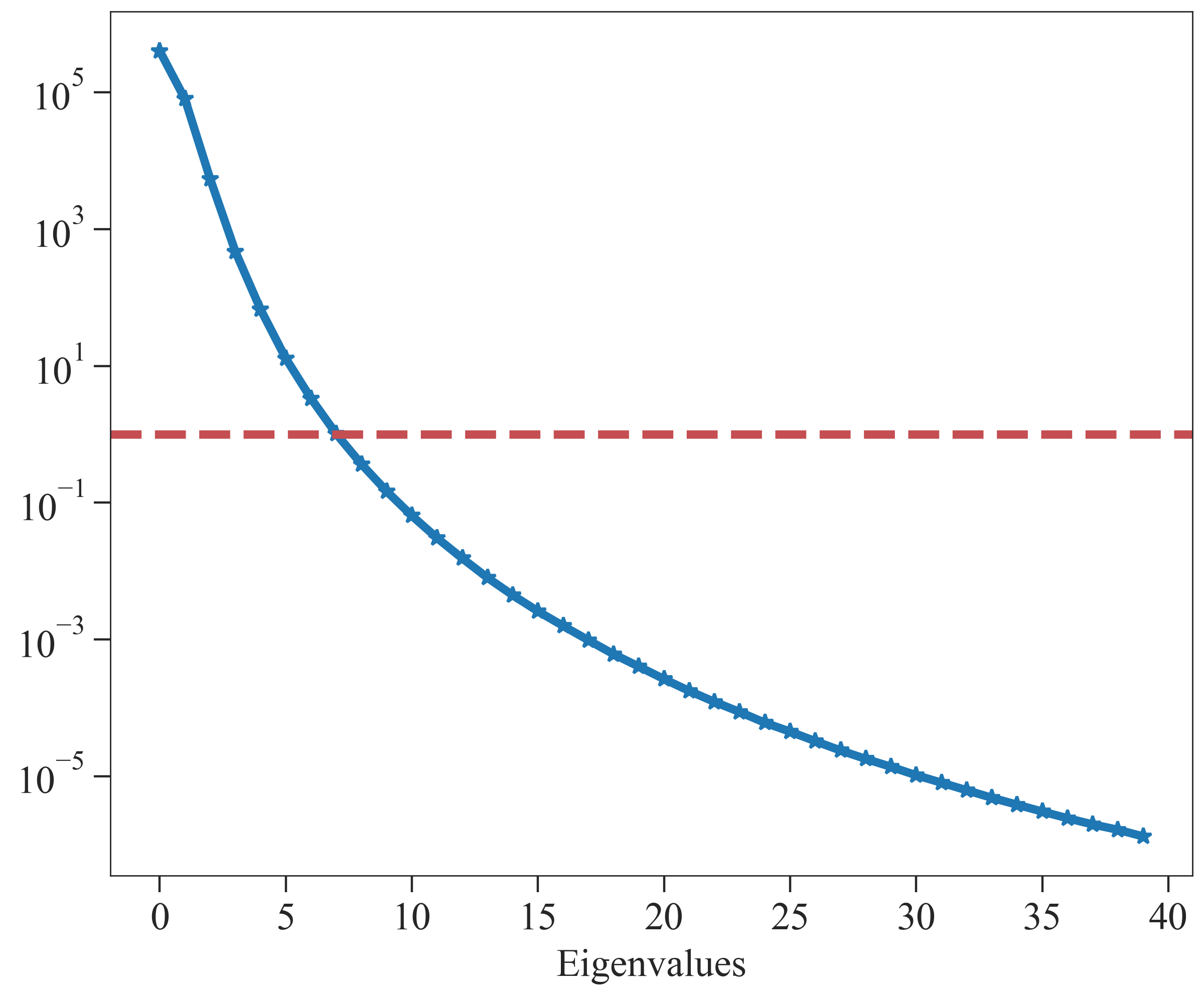}} \\

	\caption{\emph{\small 
			(a): The trace plot of non-centered Gibbs sampling method under the same mesh size of $100$; 
			(b): Logarithm of the eigenvalues $\lbrace \rho \xi_k \rbrace^{40}_{k=1}$. 
			The horizontal red dashed line $\rho \xi = 1$ shows the reference values for the truncation of the eigenvalues.}}		
	
	\label{fig:trace}
	
\end{figure}

\subsubsection{Numerical results}
In this subsection, we compare the computational results of NCP-iMFVI and non-centered Gibbs sampling methods.
Here we briefly discuss the computational cost of the methods in this subsection.

It is clear that the sampling method is computationally expensive.
In \cite{jin2010hierarchical, dunlop2017hierarchical}, the iteration number of the sampling method is chosen as $2\times 10^5$ and $4 \times 10^6$, respectively.
{ To ensure the computational accuracy, we generate $1 \times 10^6$ samples for sampling parameter $v$ and hyper-parameter $\lambda$ in this article.
The trace of the non-centered Gibbs sampling method is our concern, as shown in sub-figure (a) of Figure $\ref{fig:trace}$.
We see that the whole sampling procedure explores the entire sample space completely.
}

As for the NCP-iMFVI method, for calculating the posterior measure of $v$ in each iteration step, we need to solve one adjoint PDE (corresponding to calculate $H^{*}$), and solve $2N_{ite}$ numbers of PDEs (corresponding to calculate $\mathcal{C}_{v_k}$) with $N_{ite}$ being the maximum iteration number.
In the current settings, we assume that $N_{ite} = 10$.
Thus we need to solve $21$ PDEs to calculate the mean function of the posterior measure of $v$.
Next, based on Subsection $\ref{subsec2.5}$, it is unavoidable to calculate the eigenvalues and vectors when we calculate {$\text{Tr}(\mathcal{C}_{v_k}H^{*}\Gamma^{-1}_{\text{noise}}H))$}. 
In order to obtain an accurate approximation of the trace, we use the strategy stated in Subsection $\ref{subsec2.6}$.
{
That is, the tail terms can be neglected when $\rho\xi_k < 1$.
}
Under this prerequisite, in sub-figure (b) of Figure $\ref{fig:trace}$, we see that the horizontal red dashed line {$\rho\xi < 1$} shows the reference values for the truncation of the eigenvalues, and the number of the corresponding eigenvalues is less than $10$, which indicates that the maximum number of eigenvalues $N_{eig}$ can be taken as $10$.
By taking the double-pass algorithm in \cite{villa2021hippylib}, for calculating each eigenvalue, it is required to solve two forward problems and two adjoint problems.
As a result, we need to solve $4N_{eig}=40$ PDEs to calculate the eigenvalues.
At last, for solving the posterior measure of hyper-parameter $\lambda$, we need to calculate $2$ forward PDEs.
Each iteration step is required to solve $2N_{ite} + 4N_{eig} + 3 = 63$ PDEs.
Moreover, the VI method will converge in { $1500$} steps based on our experiment results, thus we choose the maximum iteration number {$N_{max}$} to be { $1500$}.
In summary, we need to calculate at most { $94500$} PDEs during the iterative procedure.
On the other hand, for the non-centered Gibbs sampling method, it is required to calculate $10^6$ PDEs.
We see that the computational cost of the NCP-iMFVI method is much less than the cost of the non-centered Gibbs sampling method.

{
Next, we want to illustrate that the NCP-iMFVI method provides a good estimate of the posterior measure of $u$ and $\lambda$.
Based on \cite{agapiou2014analysis}, the Gibbs sampling method provides a reliable approximation of the posterior measure of the variables according to $u$ and $\lambda$.
It is necessary to make comparisons between the estimated posterior measure of $u$ and $\lambda$ obtained by these two methods.

\begin{figure}
	\centering
	\subfloat[Comparison of $u$ I]{
		\includegraphics[ keepaspectratio=true, width=0.30\textwidth,  clip=true, trim=24pt 18pt 30pt 22pt]{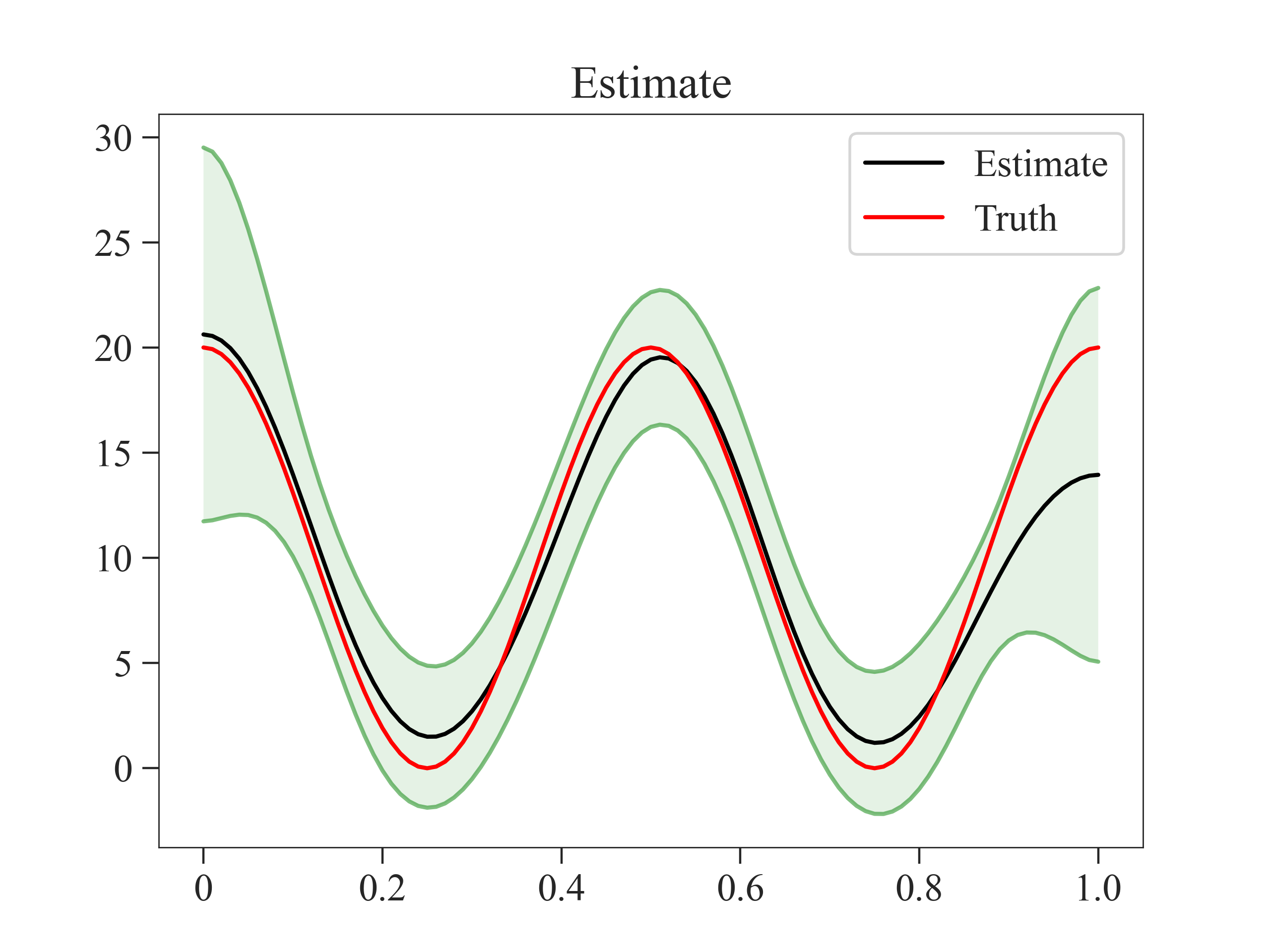}}
	\subfloat[Comparison of $u$ II]{
		\includegraphics[ keepaspectratio=true, width=0.30\textwidth,  clip=true, trim=24pt 18pt 30pt 22pt]{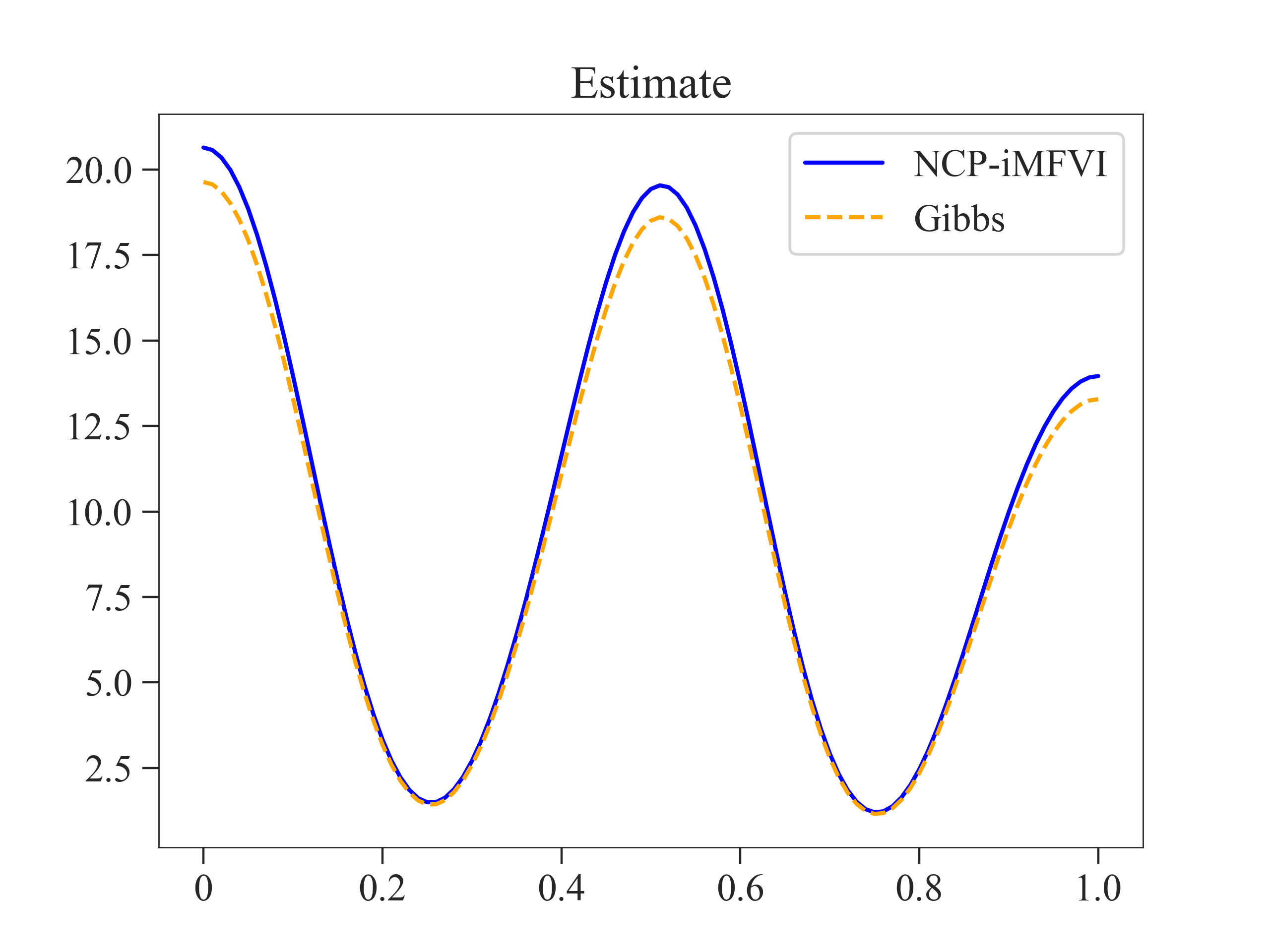}} 
	\subfloat[Comparsion of $\lambda$ densities]{
		\includegraphics[ keepaspectratio=true, width=0.270\textwidth, clip=true, trim=20pt 18pt 30pt 22pt]{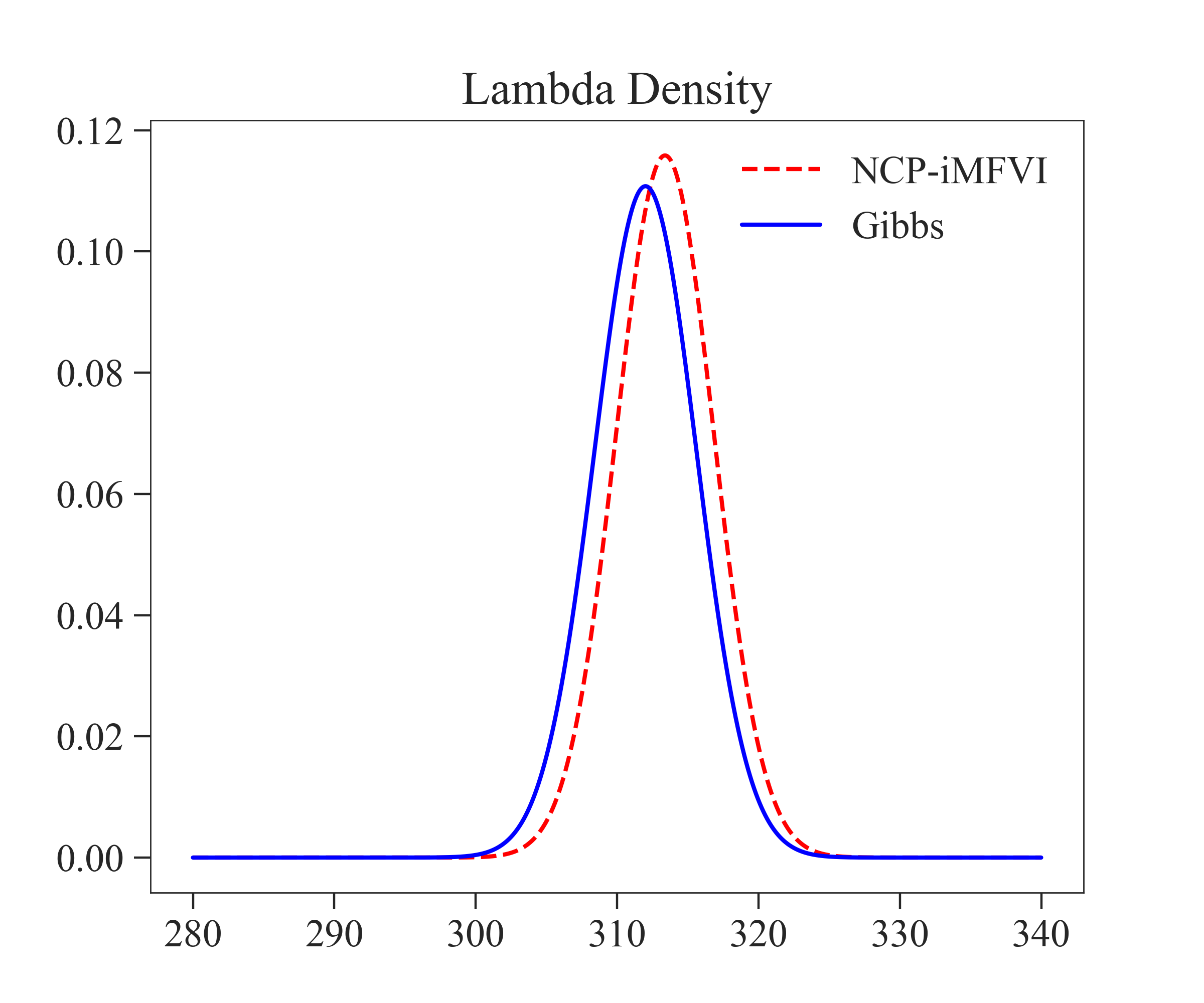}}

	\caption{\emph{\small 
			(a): The comparison of the {estimated posterior mean function} obtained by NCP-iMFVI and the background truth of $u$, respectively.
			The greed shade area represents the $95 \%$ credibility region of estimated posterior mean function; 
			(b): The comparison of the {estimated posterior mean function} obtained by NCP-iMFVI and Gibbs sampling of $u$, respectively; 
			(c): The comparison of estimated posterior density functions of $\lambda$ obtained by NCP-iMFVI and Gibbs sampling method. The red dashed line and the blue solid line represent the estimated posterior density obtained by these two methods, respectively.
	}}
	
	\label{fig:Error}
	
\end{figure}

\noindent \textbf{Discussion of $u$}:

As for the unknown parameter $u$, we care about whether NCP-iMFVI provides a good estimate of the posterior measure, i.e., a good estimate of the posterior mean and the covariance functions.

Discussion of estimated posterior mean function:
Firstly, in sub-figure (a) of Figure $\ref{fig:Error}$, the mean function of estimated posterior measure of $u$ obtained by the NCP-iMFVI method and the background truth are drawn in black solid and green dashed lines, respectively.
Two green lines represent the upper and lower bounds of the $95 \%$ credibility region of the estimated posterior mean function.
We see that the $95 \%$ credibility region includes the background truth, which reflects the uncertainty of the parameter.
In sub-figure (a) of Figure $\ref{fig:lamscom}$, we show the relative errors, which is defined by
\begin{align}\label{equ:relative}
	\text{relative error} = \frac{\lVert u_k - u^{\dagger}\rVert^2_2}{\lVert u^{\dagger}\rVert^2_2}.
\end{align}
The relative error curve illustrates that the iteration process converges within $100$ steps, and is stable around $3 \%$ at the end of the iteration.
The convergence speed is fast since the descending trend is rapid at first $20$ steps.
Based on the visual (sub-figure(a) of Figure $\ref{fig:Error}$) and quantitative ( relative errors shown in sub-figure (a) of Figure $\ref{fig:lamscom}$) evidence, we say that the NCP-iMFVI method provides an estimated posterior mean function of the parameter $u$ which is similar to the background truth.

More importantly, we provide numerical evidence to illustrate that the NCP-iMFVI method provides a good estimate of the mean function by the comparison with Gibbs sampling method.
we draw the comparison of the estimated posterior mean function obtained by two methods in sub-figure (b) of Figure $\ref{fig:Error}$, and the difference between them is visually small.
The relative error between them is given by
\begin{align}\label{equ:relatwo}
	\text{relative error} := \frac{\lVert u^{*}_{\text{N}} - u^{*}_{\text{G}} \rVert^2_{L^2}}{\lVert u^{*}_{\text{G}} \rVert^2_{L^2}} = 0.04977.
\end{align}
Hence, the estimated posterior mean function obtained by NCP-iMFVI is quantitatively similar to that by Gibbs sampling method.

As a result, combining visual (sub-figure (b) of Figure $\ref{fig:Error}$) and quantitative (relative error given in ($\ref{equ:relatwo}$)) evidence, the NCP-iMFVI method provides a good estimate of the posterior mean function.

\begin{figure}
	\centering
	
	\subfloat[Relative errors]{
		\includegraphics[keepaspectratio=true, width=0.35\textwidth, clip=true, trim=22pt 0pt 40pt 22pt]{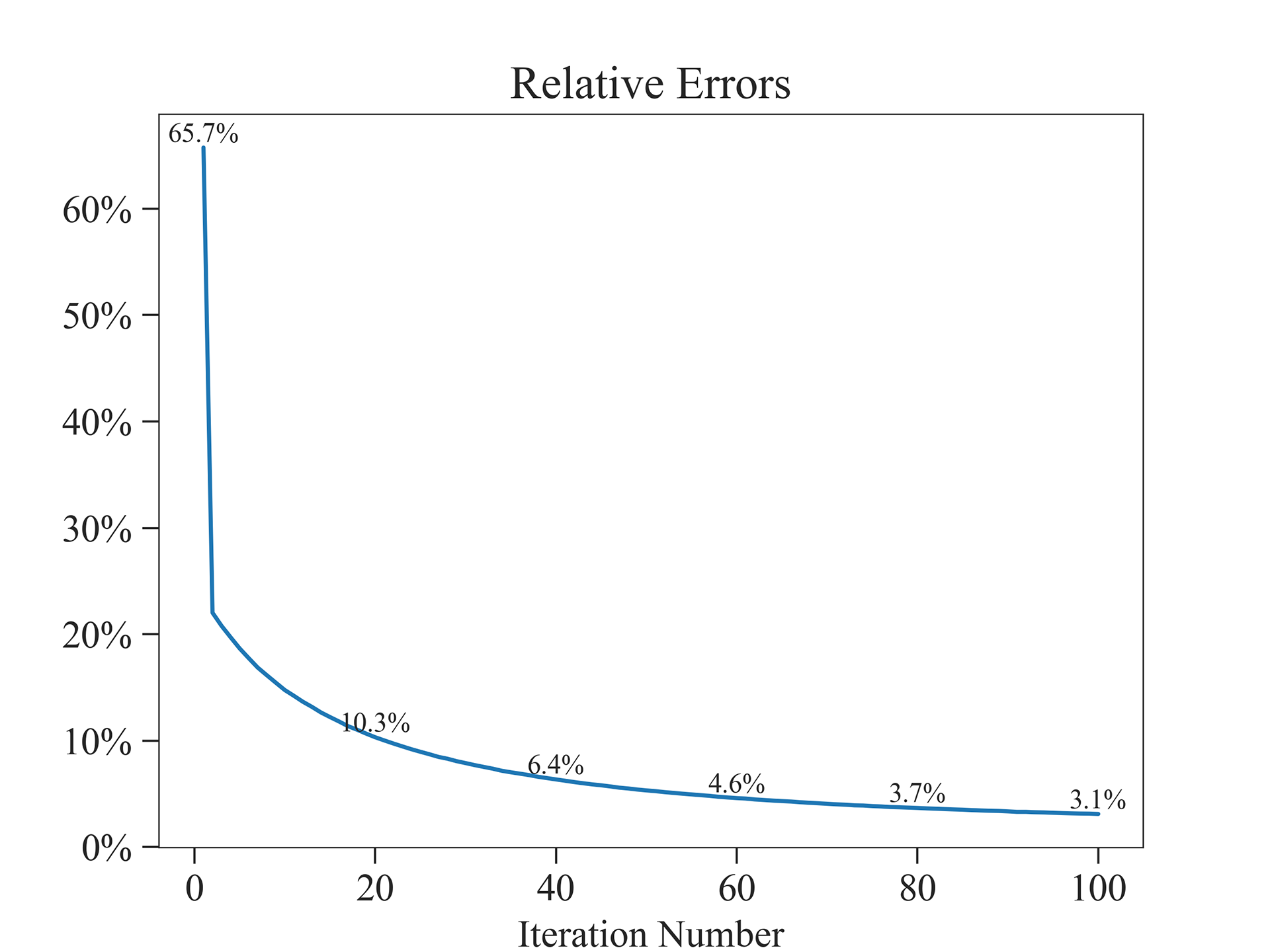}} 
	\subfloat[Step value of $\lambda$]{
		\includegraphics[keepaspectratio=true, width=0.35\textwidth, clip=true, trim=22pt 0pt 40pt 22pt]{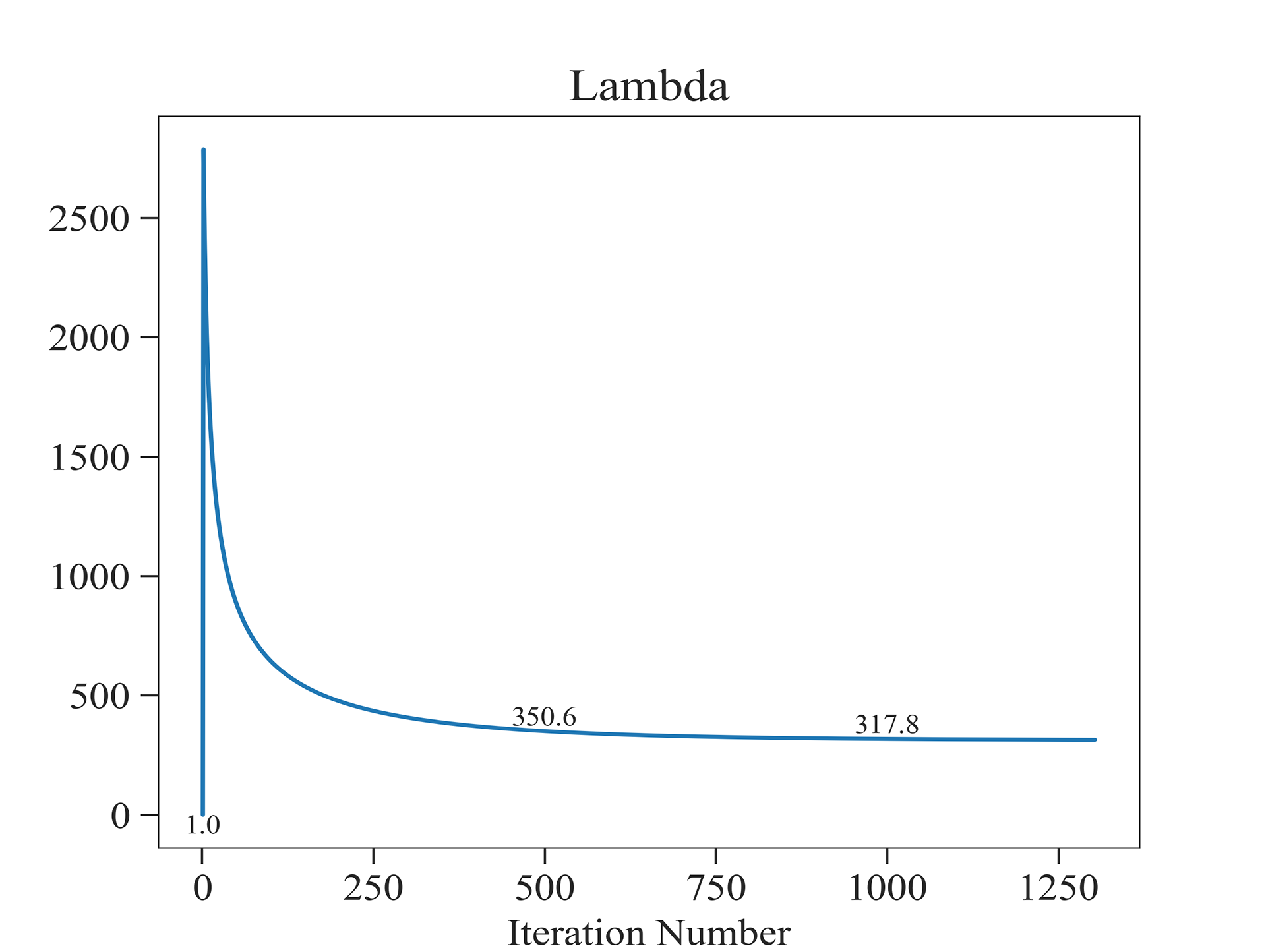}}

	\caption{\emph{\small { 
				(a): Relative errors of estimated posterior means of $u$ in the $L^2-norm$ of NCP-iMFVI method with size of mesh $n = 100$; 
				(b): The step values of $\lambda$ with size of mesh $n = 100$.}}}
	\label{fig:lamscom}
	
\end{figure}

Discussion of estimated posterior covariance function:
Secondly, we focus on the evidence to illustrate that the NCP-iMFVI method provides a good estimate of the covariance.
For the numerical convenience, we compare the estimated posterior covariance matrices, variance and covariance functions obtained by these two methods.

Before making the comparisons, we need to introduce the definition of covariance matrix, variance function and covariance function of the posterior covariance.
Consider a Gaussian random field $m$ on a domain $\Omega$ with mean $\bar{m}$ and the covariance function $c(\bm{x}, \bm{y})$ describing the covariance between $m(\bm{x})$ and $m(\bm{y})$:
\begin{align*}
	c(\bm{x}, \bm{y}) = \mathbb{E}((m(\bm{x})-\bar{m}(\bm{x}))(m(\bm{y})-\bar{m}(\bm{y}))), \quad \text{for} \ \bm{x}, \bm{y} \in \Omega.
\end{align*}
The corresponding covariance operator $\mathcal{C}$ is 
\begin{align*}
	(\mathcal{C}\phi^{\prime})(\bm{x}) = \int_{\Omega}c(\bm{x}, \bm{y})\phi^{\prime}(\bm{y})d\bm{y},
\end{align*}
where the functions $\phi^{\prime}$ are some sufficiently regular functions defined on $\Omega$.
To provide more numerical discussion of the posterior covariance, we introduce the matrix representation of the operator.
Let $V_h$ be the finite-dimensional subspace of $L^2(\Omega)$ originating from a finite element discretization with continuous Lagrange basis functions $\lbrace \phi_i\rbrace^n_{i=1}$, which correspond to the nodal points $\lbrace\bm{x}_j\rbrace^n_{j=1}$, such that $\phi_i(\bm{x}_j)=\delta_{ij}.$
The nodal vector of a function $m_h = \sum^n_{j=1}m_j\phi_j \in V_h$ is denoted by the boldface symbol $\bm{m} = (m_1, \cdots, m_n)^T$.
Inner products between nodal coefficient vectors are given by
\begin{align*}
	(\bm{m}_1, \bm{m}_2)_{\bm{M}} = \bm{m}^T_1\bm{M}\bm{m}_2,
\end{align*}
where the $ij$ element of mass matrix $\bm{M}$ is defined by
\begin{align*}
	M_{ij} = \int_{\Omega}\phi_i(\bm{x})\phi_j(\bm{x}) d\bm{x}.
\end{align*}
The matrix representation $\bm{c}$ of the operator $\mathcal{C}$ is given with respect to the Lagrange basis $\lbrace \phi_i\rbrace^n_{i=1}$ in $\mathbb{R}^n$ by
\begin{align*}
	\int_{\Omega} \phi_i\mathcal{C}\phi_j dx = \langle \bm{e}_i, \bm{c}\bm{e}_j\rangle_{\bm{M}},
\end{align*}
where $\bm{e}_j$ is the coordinate vector corresponding to the basis function $\phi_j$.

Let $\bm{\Phi(x)} = (\phi_1(\bm{x}), \cdots, \phi_n(\bm{x}))$, $\tilde{\bm{v}}_{k} = \bm{\mathcal{C}}^{1/2}_0\bm{v}_k$, and $\lbrace \xi_k, \bm{v}_k\rbrace^{n}_{k=1}$ be the eigenpairs of $\bm{\mathcal{G}}$, which is introduced in Subsection $\ref{subsec2.6}$.
In order to represent the posterior covariance operator $\mathcal{C}$, we need to calculate the posterior covariance field $c$ given by
\begin{align*}
	c(\bm{x}, \bm{y}) \approx c_{0}(\bm{x}, \bm{y}) - \sum^r_{k=1}d_k(\tilde{\bm{v}}_{kh}(\bm{x}))(\tilde{\bm{v}}_{kh}(\bm{y}))^T.
\end{align*}
Here $\tilde{\bm{v}}_{kh}(\bm{x}) = \bm{\Phi(x)}^T\tilde{\bm{v}}_{k}$, $d_k = \xi_i/\xi_i + 1$, $c_{0}(\bm{x}, \bm{y})$ is the covariance function of the prior covariance operator $\mathcal{C}_0 = (\text{I} - \alpha\Delta)^{-2}$, based on the setting in Subsection $\ref{subsec3.1.1}$.
Hence, the variance function $\lbrace c(x_i, x_i)\rbrace^{n}_{i=1}$ of the covariance field $c(\bm{x}, \bm{y})$ is calculated on the pairs of points $\lbrace (x_i, x_{i})\rbrace^{n}_{i=1}$.
And the covariance function $\lbrace c(x_i, x_{i+k})\rbrace^{n-k}_{i=1}$ of the covariance field $c(\bm{x}, \bm{y})$ is calculated on the pairs of points $\lbrace (x_i, x_{i+k})\rbrace^{n-k}_{i=1}$, where $k$ is an integer.
Readers can seek more details in \cite{bui2013computational}.

\begin{figure}
	\centering
	
	\subfloat[NCP-iMFVI]{
		\includegraphics[ keepaspectratio=true, width=0.32\textwidth, clip=true]{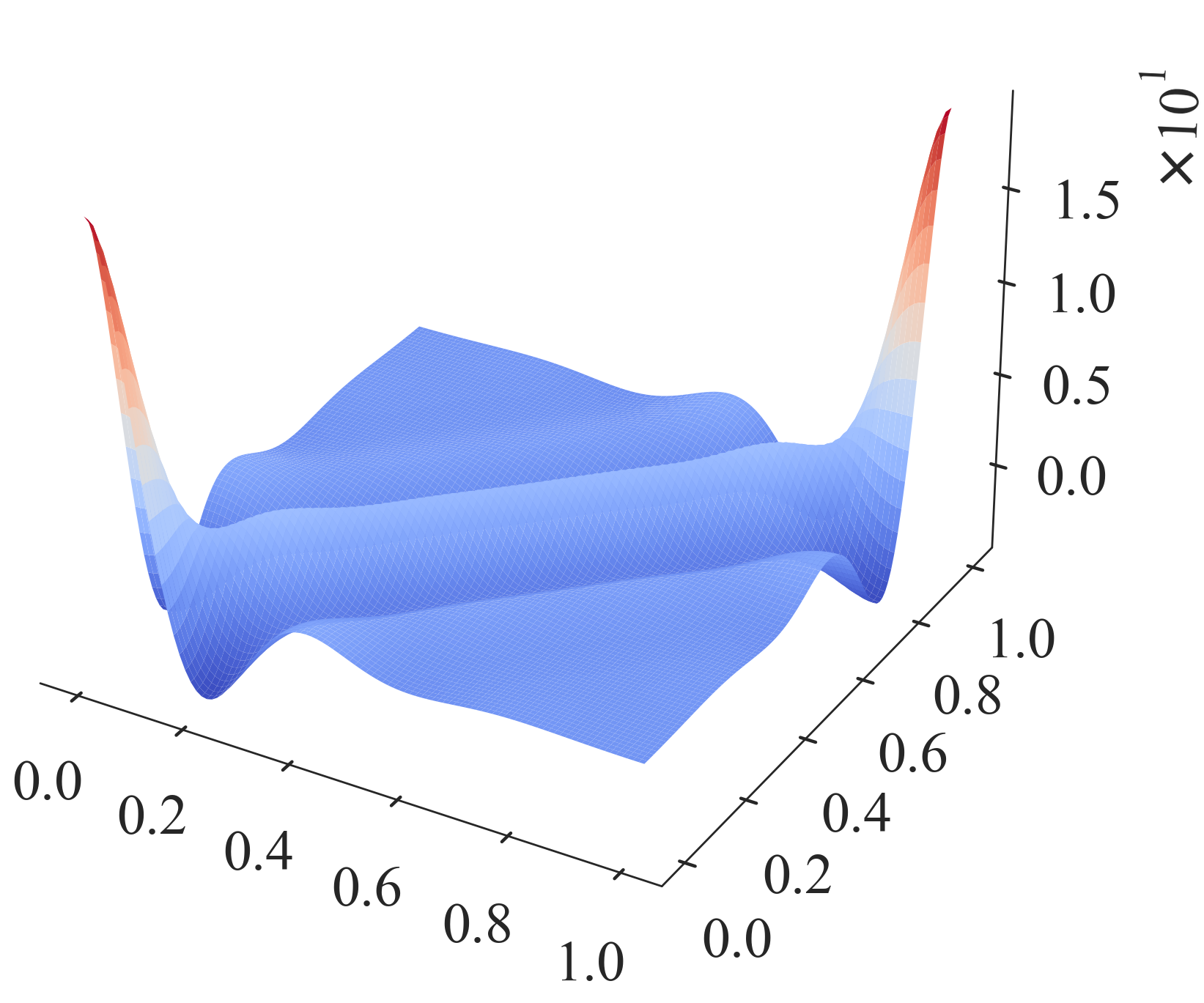}} 
	\subfloat[non-centered Gibbs sampling]{
		\includegraphics[ keepaspectratio=true, width=0.32\textwidth, clip=true]{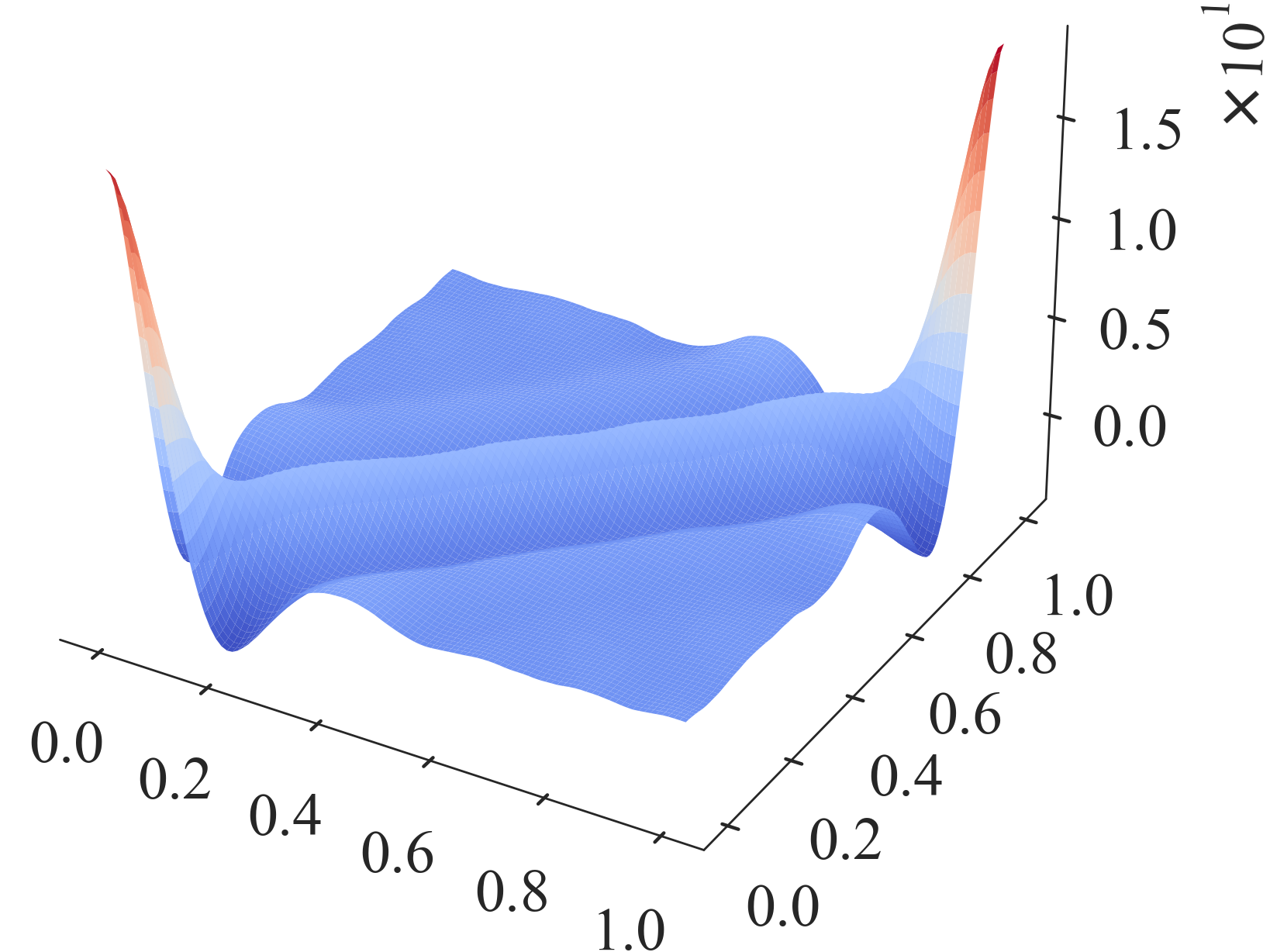}} 
	\subfloat[Difference]{
		\includegraphics[ keepaspectratio=true, width=0.32\textwidth, clip=true]{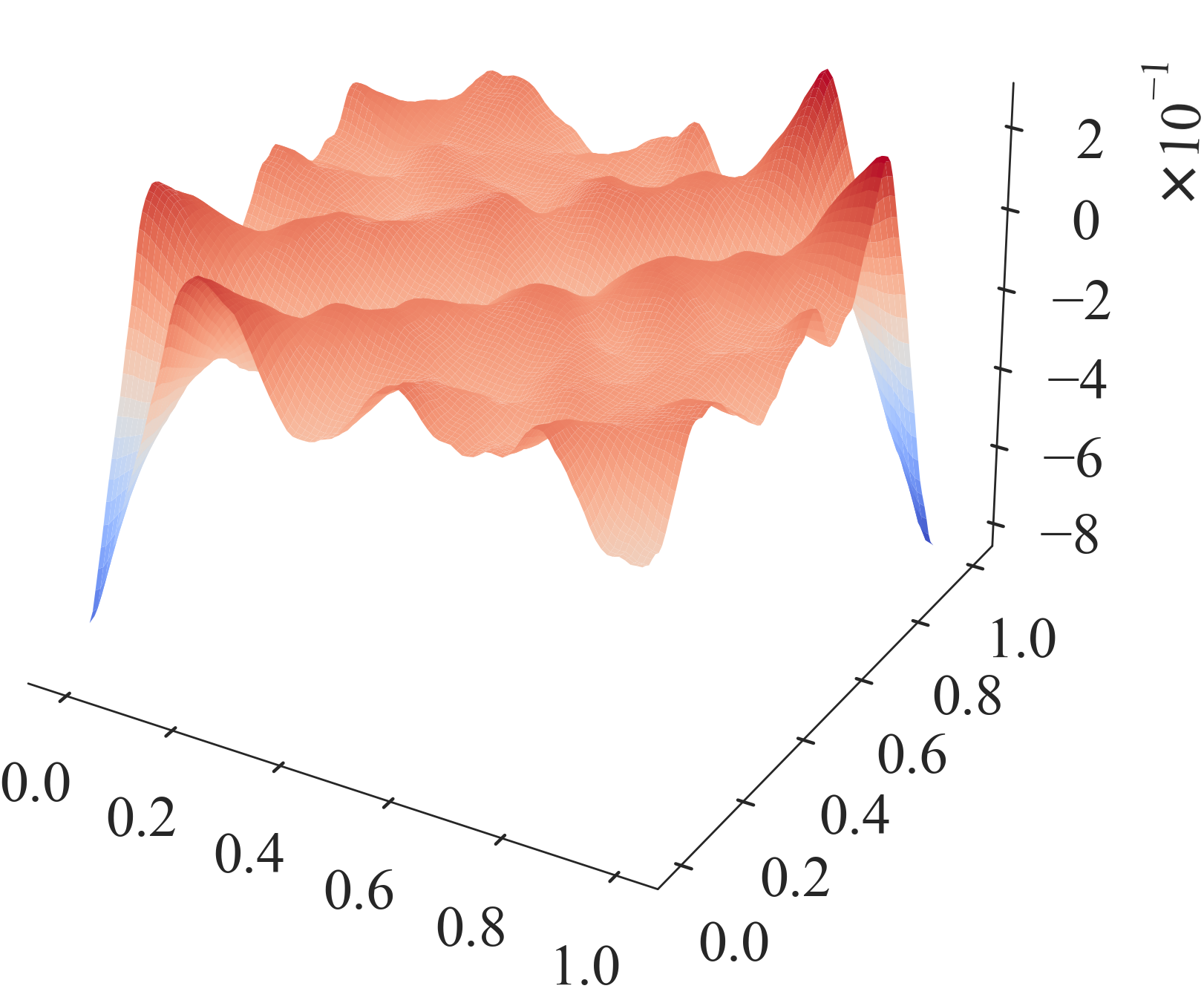}} \\

	\caption{\emph{\small {The comparison of posterior covariance of $u$ with the mesh size $n=100$ respectively. (a): The covariance given by NCP-iMFVI method; (b): The covariance given by the non-centered Gibbs sampling method; (c): The difference between the covariance obtained by NCP-iMFVI and Gibbs sampling methods.}}}
	\label{fig:Covariance}
	
\end{figure}

Then we can make the comparisons of the covariance matrices, variance and covariance functions.
Let us denote $\mathcal{N}(u^{*}_{\text{N}}, \mathcal{C}_{\text{N}})$ and $\mathcal{N}(u^{*}_{\text{G}}, \mathcal{C}_{\text{G}})$ be the estimated posterior measures of $u$ obtained by NCP-iMFVI and Gibbs sampling methods, respectively.
We show the comparisons of the following covariance matrices, variance, and covariance functions:
\begin{itemize}
	\item $\bm{c}_N$, and $\bm{c}_G$, which are the matrix representations of the covariance operators $\mathcal{C}_{\text{N}}$ and $\mathcal{C}_{\text{G}}$, respectively.
	\item $\lbrace c_N(x_i, x_i)\rbrace^{n}_{i=1}$ and $\lbrace c_G(x_i, x_i)\rbrace^{n}_{i=1}$, which are the variance functions of the covariance operators $\mathcal{C}_{\text{N}}$ and $\mathcal{C}_{\text{G}}$, respectively.
	\item $\lbrace c_N(x_i, x_{i+20}) \rbrace^{n-20}_{i=1}$ and $\lbrace c_G(x_i, x_{i+20}) \rbrace^{n-20}_{i=1}$, which are the covariance functions calculated on the pairs of points $\lbrace (x_i, x_{i+20})\rbrace^{n-20}_{i=1}$.
	\item $\lbrace c_N(x_i, x_{i+40}) \rbrace^{n-40}_{i=1}$ and $\lbrace c_G(x_i, x_{i+40}) \rbrace^{n-40}_{i=1}$, which are the covariance functions calculated on the pairs of points $\lbrace (x_i, x_{i+40})\rbrace^{n-40}_{i=1}$.
\end{itemize}
Here the mesh size $n = 100$.
We calculate the relative errors according to these functions with the following definitions:
\begin{itemize}
	\item The relative error between $\bm{c}_N$, and $\bm{c}_G$ is given by	
		\begin{align*}
		\text{relative error} = \frac{\lVert \bm{c}_N - \bm{c}_G \rVert^2_F}{\lVert \bm{c}_N \rVert^2_F},
		\end{align*}
		where $\lVert \cdot \rVert_F$ denotes the Frobenius-norm defined by
		\begin{align*}
			\lVert A \rVert_F = \bigg (\sum^{n}_{i, j=1} a_{ij}^2 \bigg )^{1/2}.
		\end{align*}
	\item The relative errors between the variance functions and the covariance functions are given by:
	\begin{align*}
		\text{relative error of variance} &= \frac{\sum^{n}_{i=1}\big (c_N(x_i, x_i) - c_G(x_i, x_i)\big )^2}{\sum^{n}_{i=1}\big (c_G(x_i, x_i)\big)^2}, \\
		\text{relative error of covariance} &= \frac{\sum^{n-k}_{i=1}\big(c_N(x_i, x_{i+k}) - c_G(x_i, x_{i+k})\big)^2}{\sum^{n-k}_{i=1}\big(c_G(x_i, x_{i+k})\big)^2},
	\end{align*}
	and $k$ is an integer.
\end{itemize}
\begin{table}
	\renewcommand{\arraystretch}{1.5}
	\centering
		\caption{\emph{\small The relative errors between the covariance matrix, variance function, and covariance functions.}} \label{table:relative}
	\begin{tabular}{c|cccc}
		\hline $\text{Function} $ & $\bm{c}$  & $\lbrace c(x_i, x_i)\rbrace^{n}_{i=1}$ & $\lbrace c(x_i, x_{i+20}) \rbrace^{n-20}_{i=1}$ & $\lbrace c(x_i, x_{i+50}) \rbrace^{n-50}_{i=1}$  \\
		\hline $\text{Relative Error}$ & $0.0860$  & $0.0688$ & $0.1152$ & $0.1514$ \\
		\hline
	\end{tabular}
\end{table}
According to the results shown in Table $\ref{table:relative}$, the relative errors are small.
This indicates that the covariance matrices, variance and covariance functions obtained by these two methods are quantitatively similar to each other.

In Figure $\ref{fig:Covariance}$, we draw the covariance matrix $\bm{c}_N$, and $\bm{c}_G$ in the sub-figures (a) and (b), and sub-figure (c) shows the difference $\bm{c}_N - \bm{c}_G$.
The covariance obtained by NCP-iMFVI is similar to that obtained by the non-centered Gibbs sampling method, and the difference between them is visually small.
Furthermore, we provide a detailed comparison of the variance and covariances functions in Figure $\ref{fig:Variance}$.
In all the sub-figures of Figure $\ref{fig:Variance}$, the variance and covariance functions obtained by NCP-iMFVI and Gibbs sampling are drawn in blue solid and the orange dashed lines.
In sub-figure (a) of Figure $\ref{fig:Variance}$, we show the variance function calculated on all the mesh point pairs $\lbrace (x_i, x_{i})\rbrace^{n}_{i=1}$ with $n = 100$.
In sub-figures (b) and (c), we show the covariance functions calculated on the pairs of points $\lbrace (x_i, x_{i+20})\rbrace^{n-20}_{i=1}$, and $\lbrace (x_i, x_{i+40})\rbrace^{n-40}_{i=1}$, respectively.
The covariance matrices, variance and covariance functions obtained by the NCP-iMFVI are visually similar to those obtained by the Gibbs sampling method.
The relative errors between the covariance matrices and variance, covariance functions given in Table $\ref{table:relative}$ are small, which means the posterior covariance obtained by NCP-iMFVI and Gibbs sampling methods are quantitatively similar to each other.
Furthermore, as is seen in the sub-figure (a) of Figure $\ref{fig:Error}$, the $95 \%$ credibility region of the estimated posterior mean contains the background truth, which indicates that the Bayesian setup is meaningful and in accordance with the frequentist theoretical investigations of the posterior consistency \cite{wang2019frequentist, zhang2020convergence}. 
Overall, we can say that the NCP-iMFVI method provides a good estimate of the posterior covariance based on the visual (Figures $\ref{fig:Covariance}$ and $\ref{fig:Variance}$) and quantitative (Table $\ref{table:relative}$) evidence.

\begin{figure}
	\centering
	
	\subfloat[Variance]{
		\includegraphics[ keepaspectratio=true, width=0.30\textwidth, clip=true, trim=24pt 18pt 30pt 22pt]{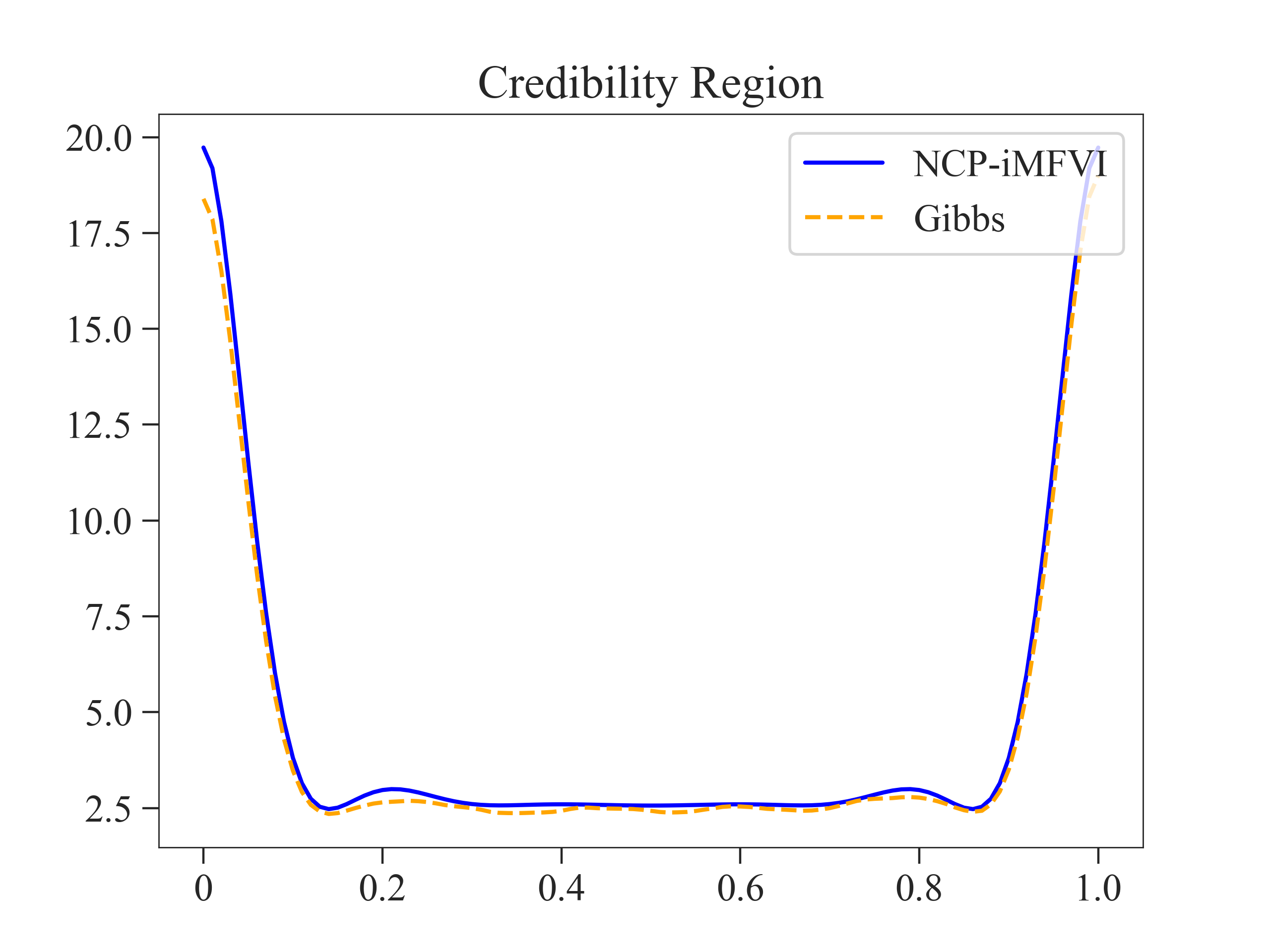}}
	\subfloat[Covariance]{
		\includegraphics[ keepaspectratio=true, width=0.30\textwidth, clip=true, trim=24pt 18pt 30pt 22pt]{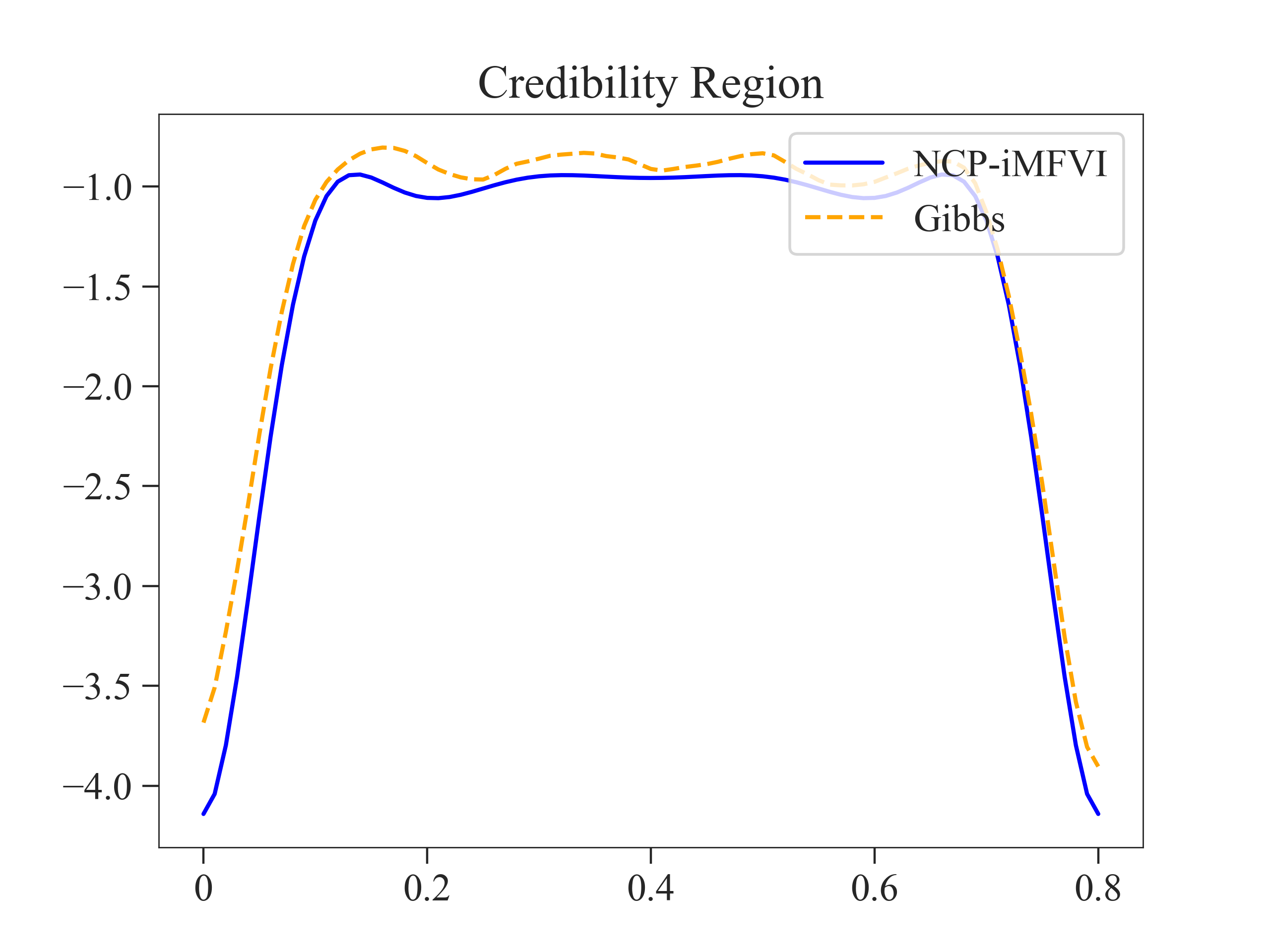}} 
	\subfloat[Covariance]{
		\includegraphics[ keepaspectratio=true, width=0.30\textwidth, clip=true, trim=24pt 18pt 30pt 22pt]{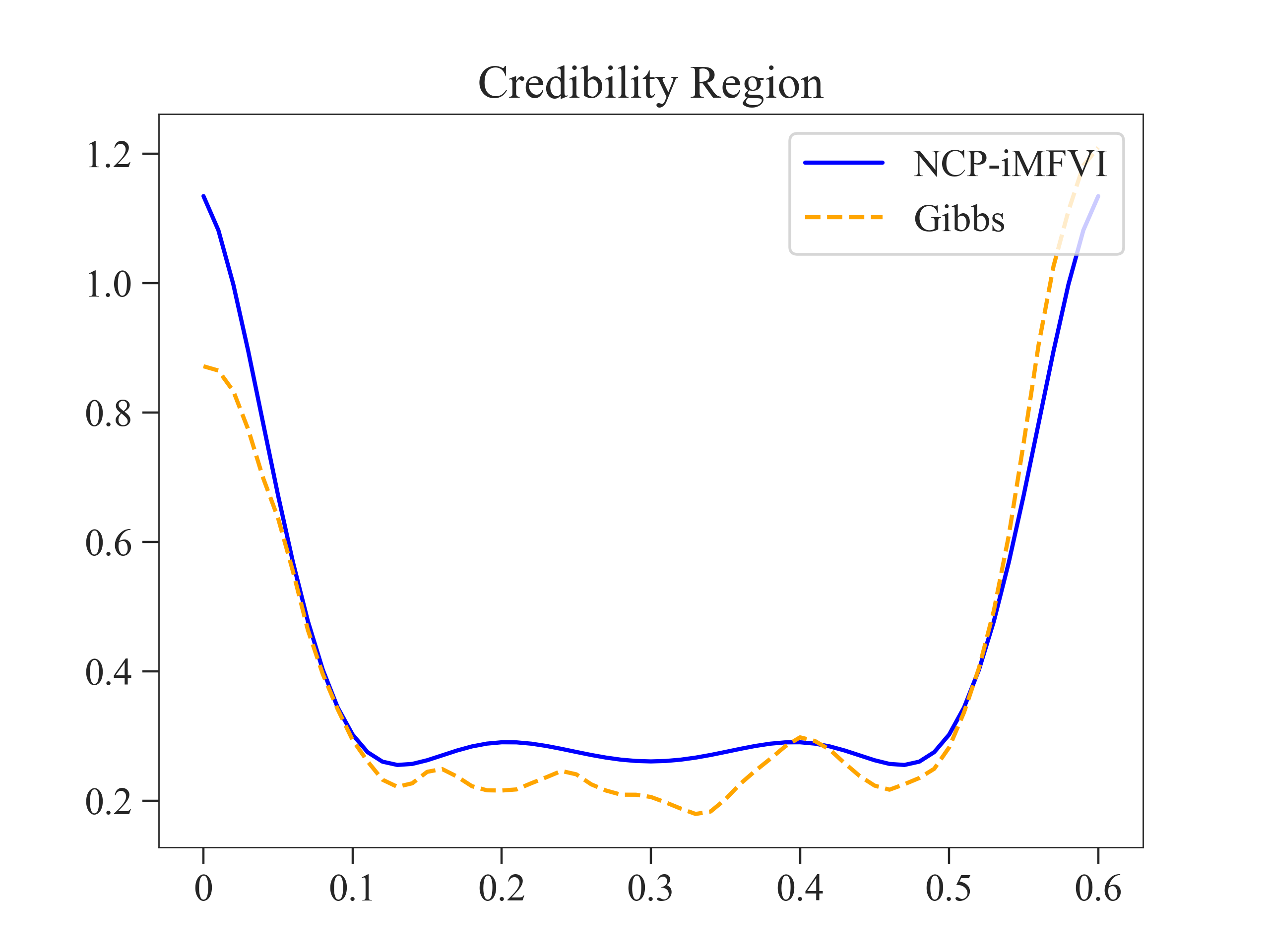}}

	\caption{\emph{\small { The estimated variance and covariance functions obtained by the NCP-iMFVI (blue solid line), and Gibbs sampling method (orange dashed line).
				(a): The variance function $\lbrace c(x_i, x_i)\rbrace^{100}_{i=1}$ on all the mesh point pairs $\lbrace (x_i, x_{i})\rbrace^{100}_{i=1}$;
				(b): The covariance function $\lbrace c(x_i, x_{i+20}) \rbrace^{80}_{i=1}$ on the mesh points $\lbrace (x_i, x_{i+20}) \rbrace^{80}_{i=1}$;
				(c): The covariance function $\lbrace c(x_i, x_{i+40}) \rbrace^{60}_{i=1}$ on the mesh points $\lbrace (x_i, x_{i+40}) \rbrace^{60}_{i=1}$}}.}
	\label{fig:Variance}
	
\end{figure}

\noindent \textbf{Discussion of $\lambda$}:

In the following, we provide numerical evidence to illustrate that the difference between posterior densities of $\lambda$ obtained by NCP-iMFVI and Gibbs sampling methods is small.
In sub-figure (b) of Figure $\ref{fig:lamscom}$, we draw the step values of $\lambda$, and $\lambda$ is finally converged to $313.387$.
At the beginning of the iteration process, the convergence speed is rapid, which is similar to the convergence speed of the relative error curve shown in sub-figure (a) of Figure $\ref{fig:lamscom}$.
Let measure $\nu^{\lambda}_{\text{N}} = \mathcal{N}(\lambda^{*}_{\text{N}}, \sigma_{\lambda_{\text{N}}})$ be the estimated posterior measure of $\lambda$ obtained by the NCP-iMFVI method, and measure $\nu^{\lambda}_{\text{G}} = \mathcal{N}(\lambda^{*}_{\text{G}}, \sigma_{\lambda_{\text{G}}})$ be the the estimated posterior measure obtained by the Gibbs sampling method, where we have
\begin{align*}
	\lambda^{*}_{\text{N}} = 313.387, \ \sigma_{\lambda_{\text{N}}} = 11.861, \\
	\lambda^{*}_{\text{G}} = 312.006, \  \sigma_{\lambda_{\text{G}}} = 12.972.
\end{align*}
The estimated posterior density functions of $\lambda$ are shown in the sub-figure (b) of Figure $\ref{fig:Error}$, where the red dashed and blue solid lines represent the density functionf obtained by NCP-iMFVI and Gibbs sampling method, respectively.
We see that the estimated posterior density obtained by NCP-iMFVI is visually similar to that of the Gibbs sampling.

For providing more numerical comparison between $\nu^{\lambda}_{\text{N}}$ and $\nu^{\lambda}_{\text{G}}$, we calculate the KL divergence of these two measures:
\begin{align*}
	D_{{\text{KL}}} (\nu^{\lambda}_{\text{N}} \Vert \nu^{\lambda}_{\text{G}}) &= \mathbb{E}^{\nu^{\lambda}_{\text{N}}}\log \bigg(\frac{d\nu^{\lambda}_{\text{N}}}{d\nu^{\lambda}_{\text{G}}} \bigg) \\
	&= \ln \sqrt{\sigma_{\lambda_{\text{G}}} / \sigma_{\lambda_{\text{N}}}} + \frac{\sigma_{\lambda_{\text{N}}} - \sigma_{\lambda_{\text{G}}}}{2 \sigma_{\lambda_{\text{G}}}}
	+ \frac{(\lambda^{*}_{\text{N}} - \lambda^{*}_{\text{G}})^2}{2 \sigma_{\lambda_{\text{G}}}} = 0.07546.
\end{align*}
The KL divergence between these two densities is quantitatively small.

Since the Gibbs sampling method provides a reliable estimate of the hyper-parameter of $\lambda$ based on \cite{agapiou2014analysis}, we say that the NCP-iMFVI method provides a reliable estimate of $\lambda$ as well, according to the visual (sub-figure (b) of Figure $\ref{fig:lamscom}$) and quantitative (KL divergence) evidence.
}

\begin{table}
	\renewcommand{\arraystretch}{1.5}
	\centering
	{
		\caption{\emph{\small The posterior mean and variance of $\lambda$ under different discrete meshes.}} \label{table:lambda}
		\begin{tabular}{c|ccccc}
			\hline $\text{Mesh level}$ & $100$  & $300$ & $500$ & $700$ & $900$    \\
			\hline $\text{Mean}$ & $313.38665$ & $313.01847$ & $313.10529$ & $313.11057$ & $313.15050$ \\
			\hline $\text{Variance}$ & $11.86120$  & $11.83329$ & $11.83986$ & $11.84026$ & $11.84328$ \\
			\hline
		\end{tabular}
	}
\end{table}

{
Furthermore, because of taking the reparameterization $u = \lambda v$, we expect that the posterior measure of $u$ is the same even if we take different prior measures of $v$.
We want to show that the posterior mean of $\lambda$ is self-adjustable when the prior measure of $v$ changes.
Based on the settings in Subsection $\ref{subsec3.1.1}$, the prior measure of $v$ is $\mu^v_0 = \mathcal{N}(0, \mathcal{C}_0)$, where $\mathcal{C}_0 = (\text{I} - \alpha\Delta)^{-2}$.
Then sub-figure (a) of Figure $\ref{fig:lamhalf}$ shows the posterior measure of $\lambda$, and the posterior mean ${\lambda^{*}_1}$ is around $313$.
Next, we set the prior measure of $v$ as ${(\mu^{v}_0)^{\prime}: = \mathcal{N}(0, 4\mathcal{C}_0)}$.
Then sub-figure (b) of Figure $\ref{fig:lamhalf}$ shows the posterior measure of $\lambda$, and the posterior mean ${\lambda^{*}_2}$ is around $176$.
We can see that ${\lambda^{*}_2 \approx 0.5\times \lambda^{*}_1}$.
That means the posterior mean of $\lambda$ is self-adjustable, which is in line with our theorized expectations.      

\noindent \textbf{Summary}:

Taking the Bayesian approach, a good approximation of posterior measure of $u$ not only needs a good posterior mean function, but also a good estimate of the covariance which can quantify the uncertainty of the parameter $u$.
As we mentioned in the discussion of $u$, the estimated posterior mean function of $u$ obtained by NCP-iMFVI is similar to that by Gibbs sampling, based on the visual (sub-figure (b) of Figure $\ref{fig:Error}$) and quantitative (relative error calculated in ($\ref{equ:relatwo}$)) evidence.
On the other hand, $95 \%$ credibility region of the estimated mean function includes background truth.
Combining Figures $\ref{fig:Covariance}$ and $\ref{fig:Variance}$, we see that posterior covariance matrices, variance and covariance functions obtained by these two methods are visually similar to each other.
And Table $\ref{table:relative}$ shows that the relative errors between the covariance matrices, variance, and covariance functions are quantitatively small, which indicates that the posterior covariance obtained by NCP-iMFVI method is as good as it by Gibbs sampling method.
Hence, NCP-iMFVI method provides a reliable estimate of the posterior measure of $u$ based on the visual and quantitative evidence.

For the parameter $\lambda$, we draw the posterior densities obtained by these two methods, and a large part overlap with each other.
The KL divergence between these two distributions is small, which means that they are quantitatively similar to each other.
This evidence confirms that  NCP-iMFVI method provides a good estimate of the posterior measure of $\lambda$.

Overall, the proposed NCP-iMFVI method obtains a good approximation of the posterior measures for both $u$ and $\lambda$.

\begin{figure}
	\centering
	
	\subfloat[Posterior density with $\mathcal{N}(0, \mathcal{C}_0)$]{
		\includegraphics[ keepaspectratio=true, width=0.35\textwidth, clip=true, trim=5pt 5pt 5pt 5pt]{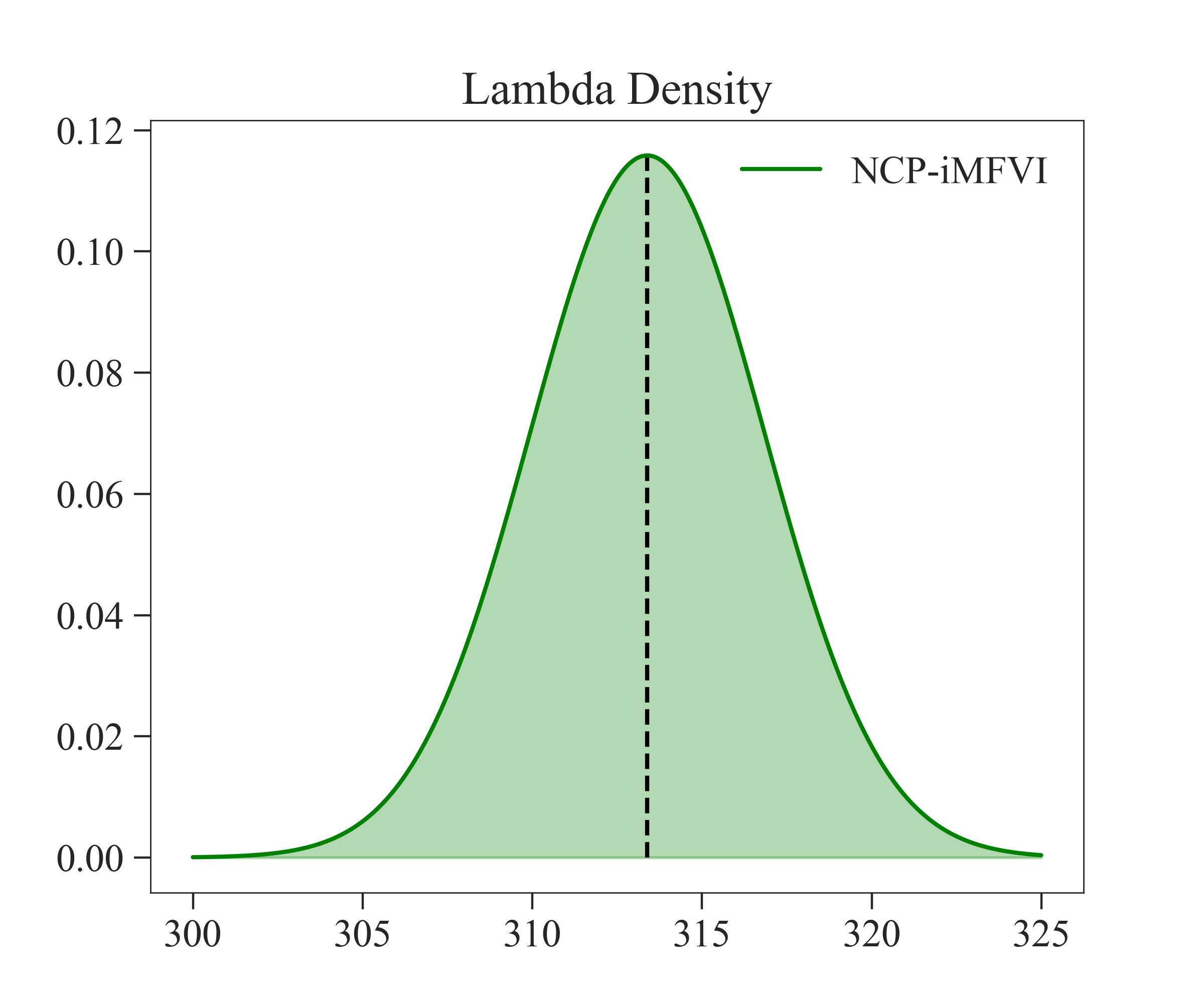}}
	\subfloat[Posterior density with $\mathcal{N}(0, 4\mathcal{C}_0)$]{
		\includegraphics[ keepaspectratio=true, width=0.35\textwidth, clip=true, trim=5pt 5pt 5pt 5pt]{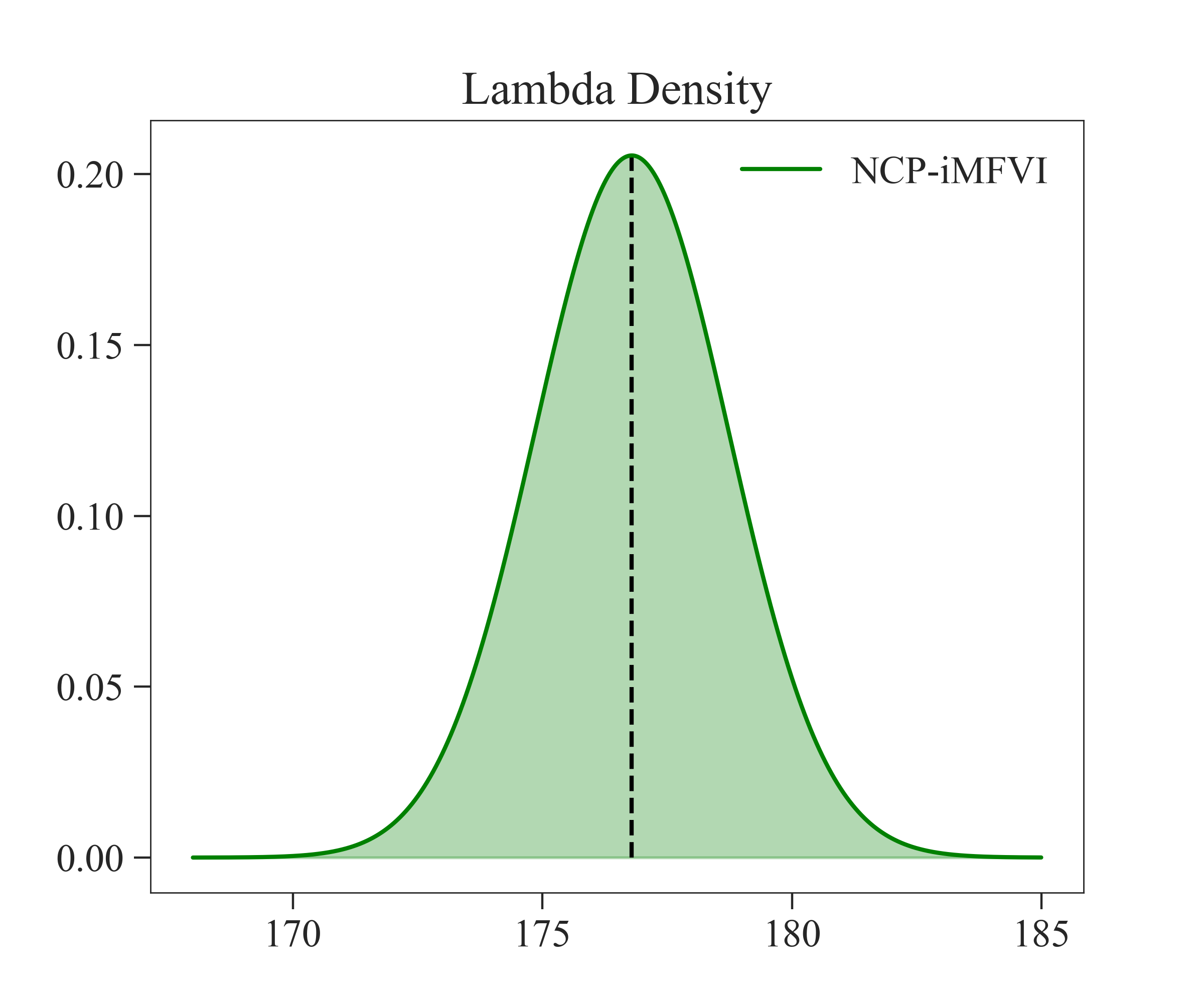}}

	\caption{\emph{\small { Comparison of posterior measures of $\lambda$ calculated by difference prior measures of $v$. (a): the posterior measure of $\lambda$ calculated by prior measure $\mu^v_0 = \mathcal{N}(0, \mathcal{C}_0)$; (b): the posterior measure of $\lambda$ calculated by prior measure $\mu^v_0 = \mathcal{N}(0, 4\mathcal{C}_0)$. }}}
	\label{fig:lamhalf}
	
\end{figure}

\noindent \textbf{Mesh independence}:

At last, we illustrate the mesh independence of the NCP-iMFVI method, as expected for the ``Bayesianize-then-discretize'' approach.
We define the step norm in the $L^2$-norm as follows:
\begin{align*}
	\text{the k-th step norm} = \lVert u_{k+1} - u_{k} \rVert^2 / \lVert u_{k} \rVert^2.
\end{align*}
In sub-figure (a) of Figure $\ref{fig:inde}$, we draw the step norms computed by the ``Discretize-then-Bayesianize'' approach with different discretized dimensions, and we see that this approach can hardly keep the infinite-dimensional natural.
Readers can find more details about this approach in \cite{jin2010hierarchical}.
We can see that the step norms decay rapidly when the dimension grows.
This indicates that the convergence speed depends highly on the discretized dimension.
On the contrary, the logarithm curves of step norms under the NCP-iMFVI method are almost the same for different discretizations, as shown in sub-figure (b) of Figure $\ref{fig:inde}$.
Moreover, we show the estimated posterior mean and covariance of hyper-parameter $\lambda$ obtained by each discrete level $ n = \lbrace 100, 300, 500, 700, 900 \rbrace$ in Table $\ref{table:lambda}$.
The estimated posterior means of $\lambda$ in all discretized dimensions are around $313$.
Under different discrete levels, the estimated posterior densities of $\lambda$ are quantitatively similar to each other.
Thus we say that the iteration process is not affected by discretized dimensions, demonstrating that the NCP-iMFVI method has mesh independence property.
}

\begin{figure}
	\centering
	\subfloat[Step norms(Logarithm)]{
		\includegraphics[keepaspectratio=true, width=0.35\textwidth, clip=true ]{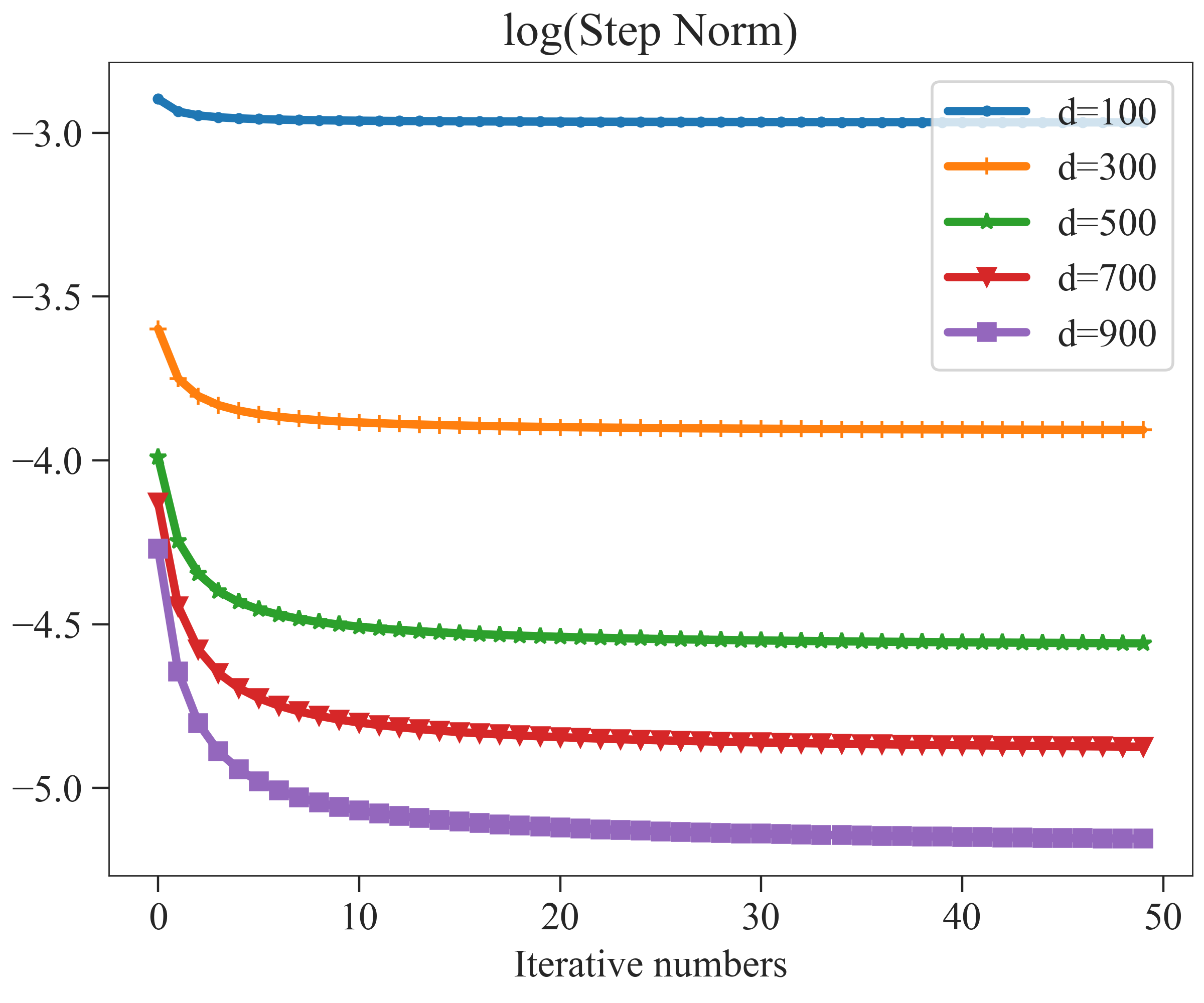}} 
	\subfloat[Step norms(Logarithm)]{
		\includegraphics[ keepaspectratio=true, width=0.35 \textwidth, clip=true]{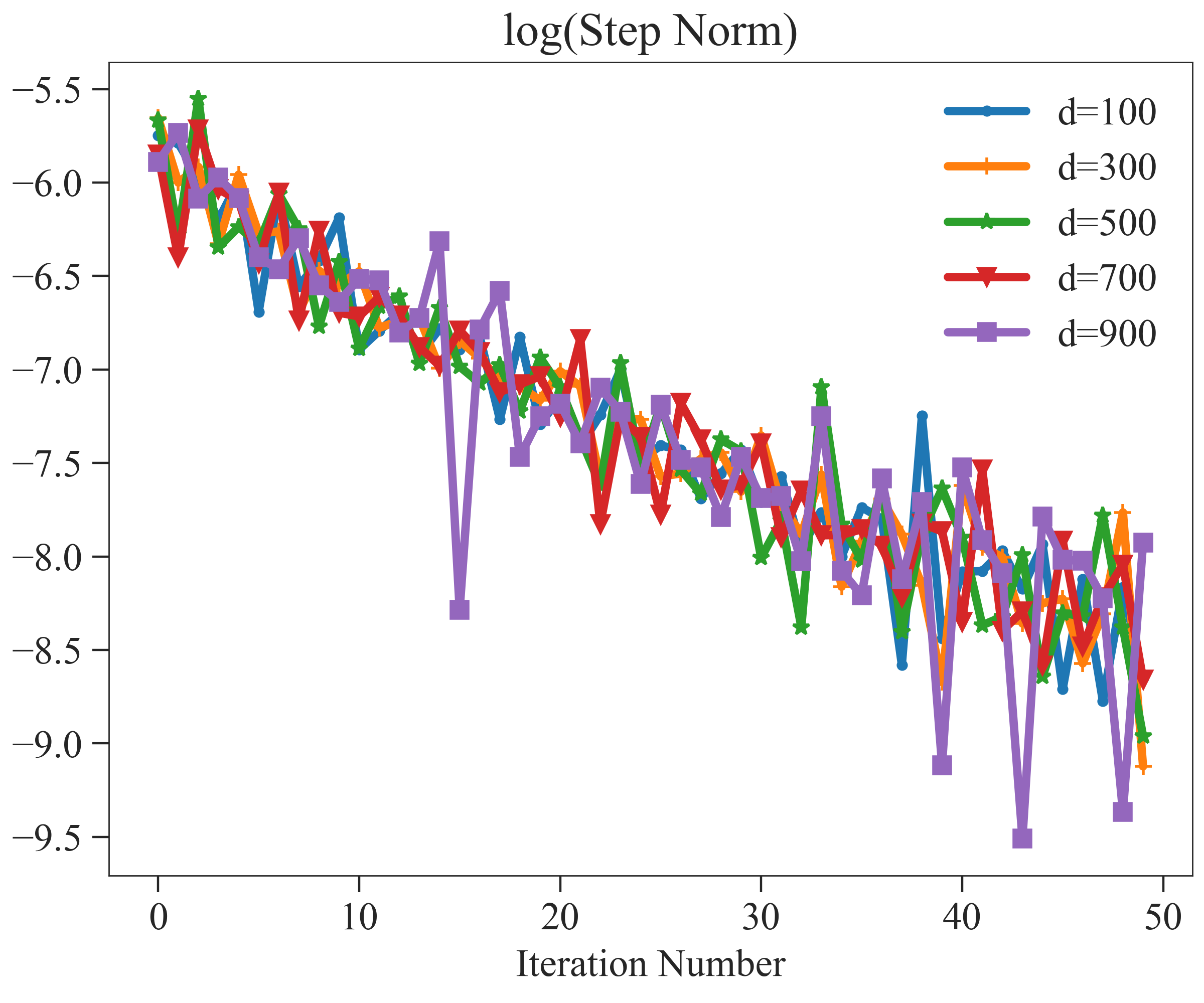}} \\

	\caption{\emph{\small 
			(a): Logarithm of the step norms computed by `` Discretize-then-Bayesianize'' approach with different discretized dimensions $n = \lbrace 100, 300, 500, 700 ,900 \rbrace$;
			(b): Logarithm of the step norms computed by NCP-iMFVI method with different discretized dimensions $n = \lbrace 100, 300, 500, 700 ,900 \rbrace$.}}		
	
	\label{fig:inde}
	
\end{figure}

~\
\subsection{Inverse source problem of Helmholtz equation}\label{subsec3.2}

\subsubsection{Basic settings}
The inverse source problem we studied in this section is borrowed from \cite{Bao_2010, Bao_2015}, which determines the unknown current density function from measurements of the radiates fields at multiple wavenumbers.
Let us consider the two-dimensional Helmholtz equation:
\begin{align*}
 \Delta w + \kappa ^2 w &= u \quad \text{in} \ \mathbb{R}^2, \\
 \partial_r w - i\kappa w &= o(r^{-1/2}) \quad \text{as} \ r=\lvert x \rvert \rightarrow \infty,
\end{align*}
where $\kappa$ is the wavenumber, $w$ is the acoustic field, and $u$ is the source supported in a bounded domain $\Omega = (0, 1)^2$. 
For this two-dimensional case, to simulate the problem defined on $\mathbb{R}^2$, we use the uniaxial perfectly matched layer (PML) technique to truncate the whole space $\mathbb{R}^2$ into a bounded domain. 
Denoting $\bm{x} = (x, y) \in \mathbb{R}^2$, let $D$ be a rectangle containing $\Omega$ and let $d_1, d_2$ be the thickness of the PML layers along the $x$ and $y$ coordinates, respectively. 
Denote by $\partial D$ the boundary of the domain $D$. Let $s_1(x) = 1 + i\sigma_1(x)$ and $s_2(y) = 1 + i\sigma_2(y)$ be the model medium property, where $\sigma_j$ are the positive continuous even functions and satisfy $\sigma_j(x) = 0$ in $\Omega$. 
Readers can seek more details about PML in \cite{Bao_2010, jia2019recursive}.
Following the general idea in designing PML absorbing layers, we may deduce the truncated PML problem: find the PML solution from the following system
\begin{align}
	\begin{split}
		\nabla \cdot (s\nabla w) + \kappa ^2s_1s_2w &= u \quad \text{in} \ D, \\
		w &= 0 \quad \text{on} \ \partial D,
	\end{split}
\end{align}
where $s = \diag(s_2(y) / s_1(x), s_1(x) / s_2(y))$ is a diagonal matrix. The forward operator related to $\kappa$ is defined by the Helmholtz equation $H_{\kappa}(u) = (w(\bm{x}_1), \cdots, w(\bm{x}_{N_d}))^T$ with $\lbrace \bm{x}_i \rbrace ^{N_d}_{i=1} \in \partial \Omega$ and $u \in \mathcal{H}_u := L^2(\Omega)$. 
Since we consider the multi-frequency case, i.e., a series of wavenumbers $0 < \kappa_1 < \kappa_2 < \cdots < \kappa_{N} < \infty$ are considered, then the forward operator has the following form:
\begin{align}
 Hu = (H_{\kappa_1}(u), \cdots, H_{\kappa_N}(u)) \in \mathbb{R}^{N_d \times N}.
\end{align}
Similar to the simple smooth model, with these notations, we can write an abstract form of this problem:
\begin{align}
 \bm{d} = Hu + \bm{\epsilon},
\end{align}
where $\bm{\epsilon} \sim \mathcal{N}(0, \bm{\Gamma}_{\text{noise}})$ is the random Gaussian noise.
Parameter $v$ and hyper-parameter $\lambda$ are generated by the Gaussian measures $\mathcal{N}(0, \mathcal{C}_0)$ and $\mathcal{N}(\bar{\lambda}, \sigma)$, where{ $\bar{\lambda} = 1, \sigma = 100$}, respectively.

\begin{figure}
	\centering
	
	\subfloat[Background truth]{
		\includegraphics[ keepaspectratio=true, width=0.32\textwidth, clip=true]{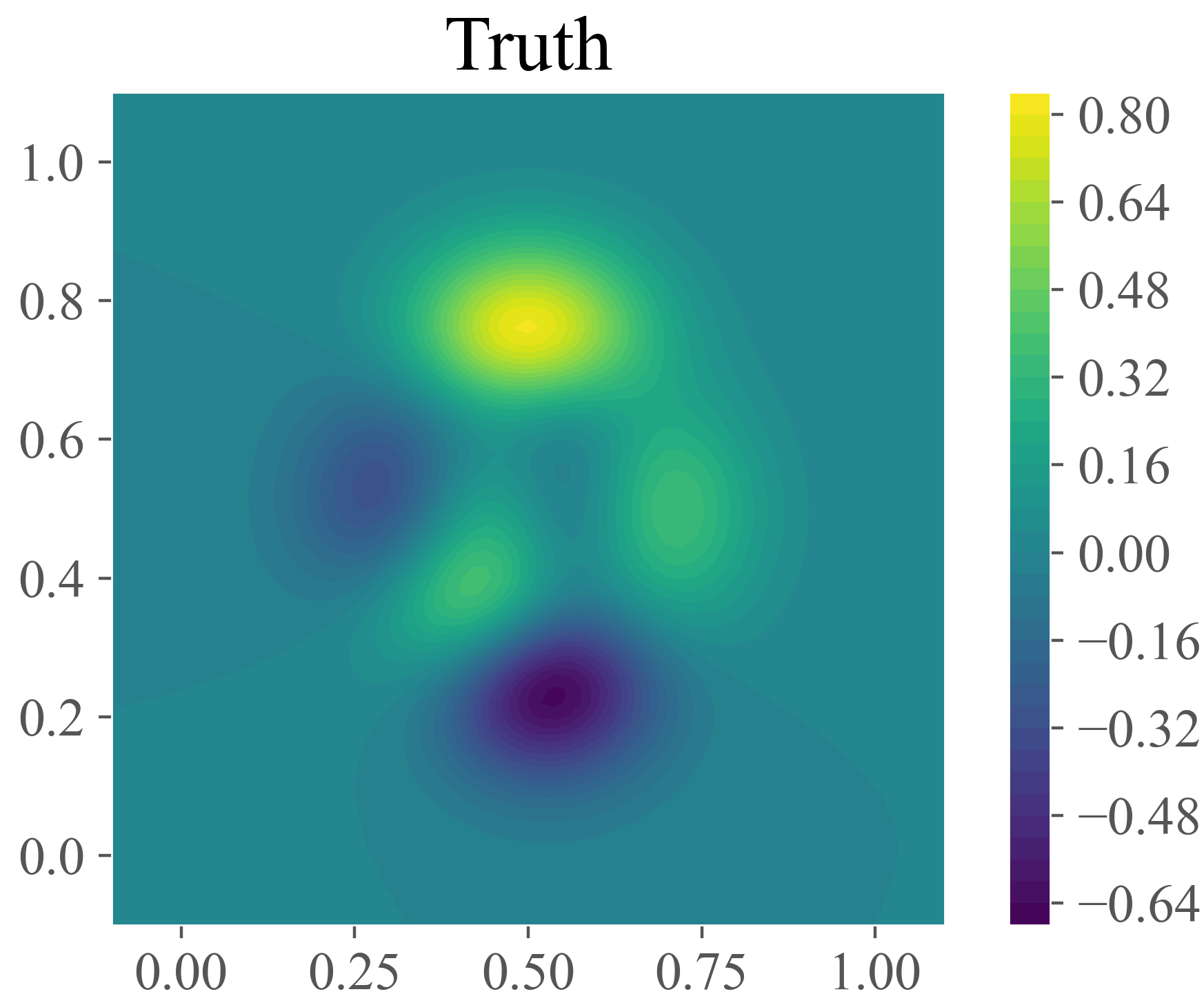}} 
	\subfloat[Estimated mean]{
		\includegraphics[ keepaspectratio=true, width=0.32\textwidth, clip=true]{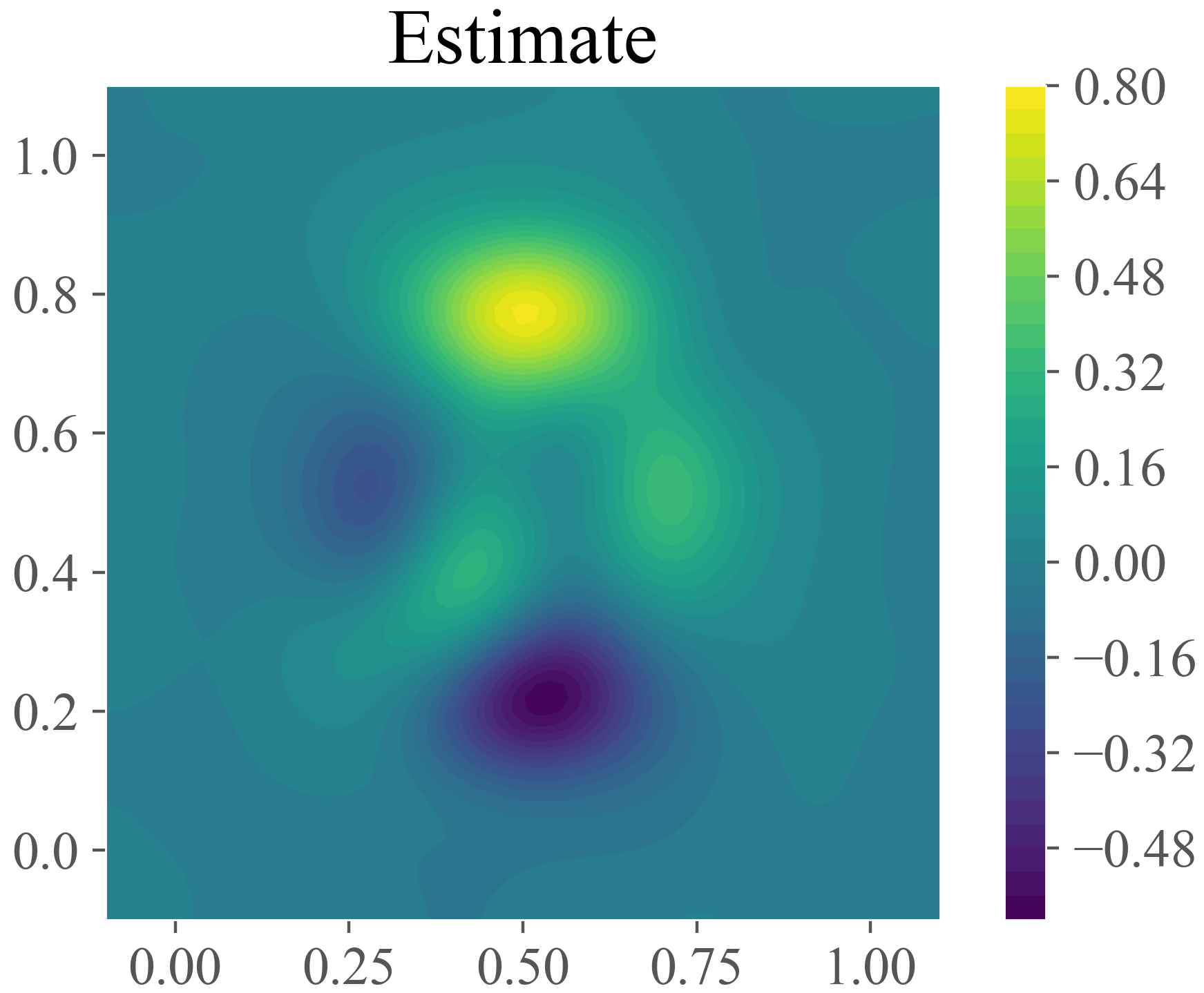}} 
	\subfloat[Estimated variance]{
		\includegraphics[ keepaspectratio=true, width=0.32\textwidth, clip=true]{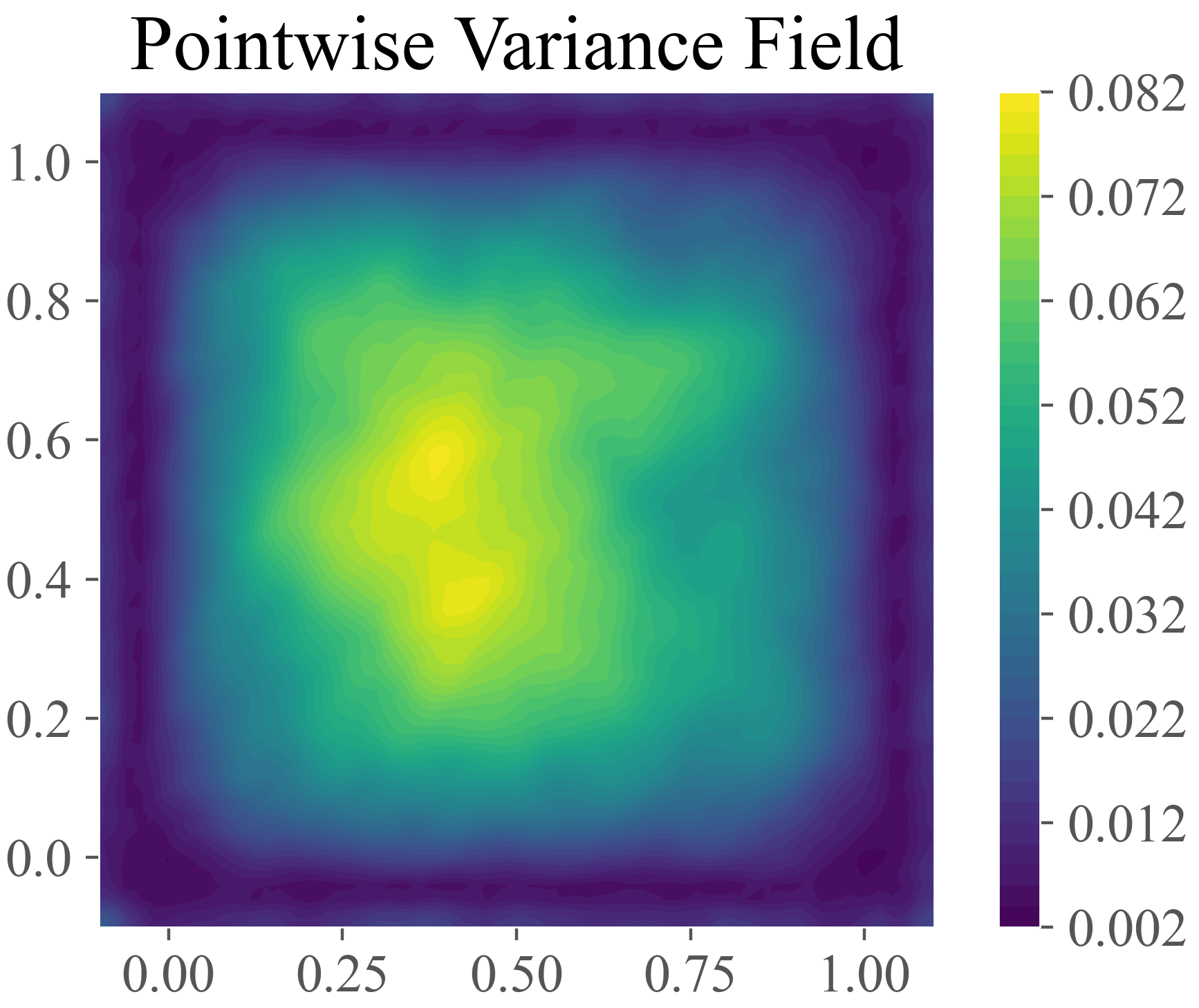}}
	
	\caption{\emph{\small (a): The background truth of $u$; (b): The estimated posterior mean function of $u$ obtained by NCP-iMFVI; (c): The estimated point-wise variance field of posterior measure of $u$ obtained by NCP-iMFVI. }}
	
	\label{fig:Helmcomparison}
	
\end{figure}

For clarity, we list the specific choices for some parameters introduced in this subsection as follows: 
\begin{itemize}
	\item {Assume that 5$\%$ random Gaussian noise $\epsilon \sim \mathcal{N}(0, \bm{\Gamma}_{\text{noise}})$ is added, where $\bm{\Gamma}_{\text{noise}} = \tau^{-1}\textbf{I}$, and $\tau^{-1} = (0.05\max(\lvert Hu\rvert))^2$.}
	\item We assume that the data produced from the underlying true signal
	\begin{align*}
		u^{\dagger} = \ &0.3(4-6x_1)^2u_0(x_1, x_2) - 0.03u_0(x_1, x_2) \\
		&- ((1.2x_1-0.6)-(6x_1-3)^3-(6x_2-3)^5)u_0(x_1, x_2),
	\end{align*}
	where $u_0(x_1, x_2) = \exp(-(6x_1-3)^2 - (6x_2-2)^2)$.
	\item The operator $\mathcal{C}_0$ is given by $\mathcal{C}_0 = (\text{I} - \alpha\Delta)^{-2}$, where $\alpha = 0.05$ is a fixed constant. 
	Here the Laplace operator is defined on $\Omega$ with zero Neumann boundary condition. 
	\item Let domain $\Omega$ be a bounded area $(0, 1)^2$. 
	And the available data is taken on the boundary uniformly for $80$ points.
	\item The wavenumber series are specified with { $\kappa_j = 0.5 \pi j (j = 1, 2, \cdots ,20)$}, and the thickness of PML layers is set as $d_1 = d_2 = 0.1$.
	\item In order to avoid the inverse crime, a fine mesh with the number of grid points equal to $500 \times 500$ is employed for generating the data. 
	For the inverse stage, the mesh with the number of grid points equal to { $60 \times 60$} is employed as the wave numbers are below 30 according to \cite{wong2011exact}.
\end{itemize}

\begin{remark}\label{remark3}
	\itshape
		To employ sampling-type methods such as the MCMC method, researchers carefully parameterize the unknown source function to reduce the dimension, e.g., assume that the sources are point sources, then parameterize the source function by numbers, locations, and amplitudes. 
		If we employ MCMC algorithm \cite{cotter2013, feng2018adaptive} in our settings, the computational complexity is unacceptable for two reasons: calculation with many wavenumbers is needed for multi-frequency problems, and a large number of samples need to be generated for each wavenumber. 
		Similar to \cite{jia2021variational}, we could hardly compare our NCP-iMFVI method with the MCMC sampling method in this case.
\end{remark}

\subsubsection{Numerical results}

\begin{figure}
	\centering
	
	\subfloat[Credibility region $\text{I}$]{
		\includegraphics[ keepaspectratio=true, width=0.32\textwidth, clip=true]{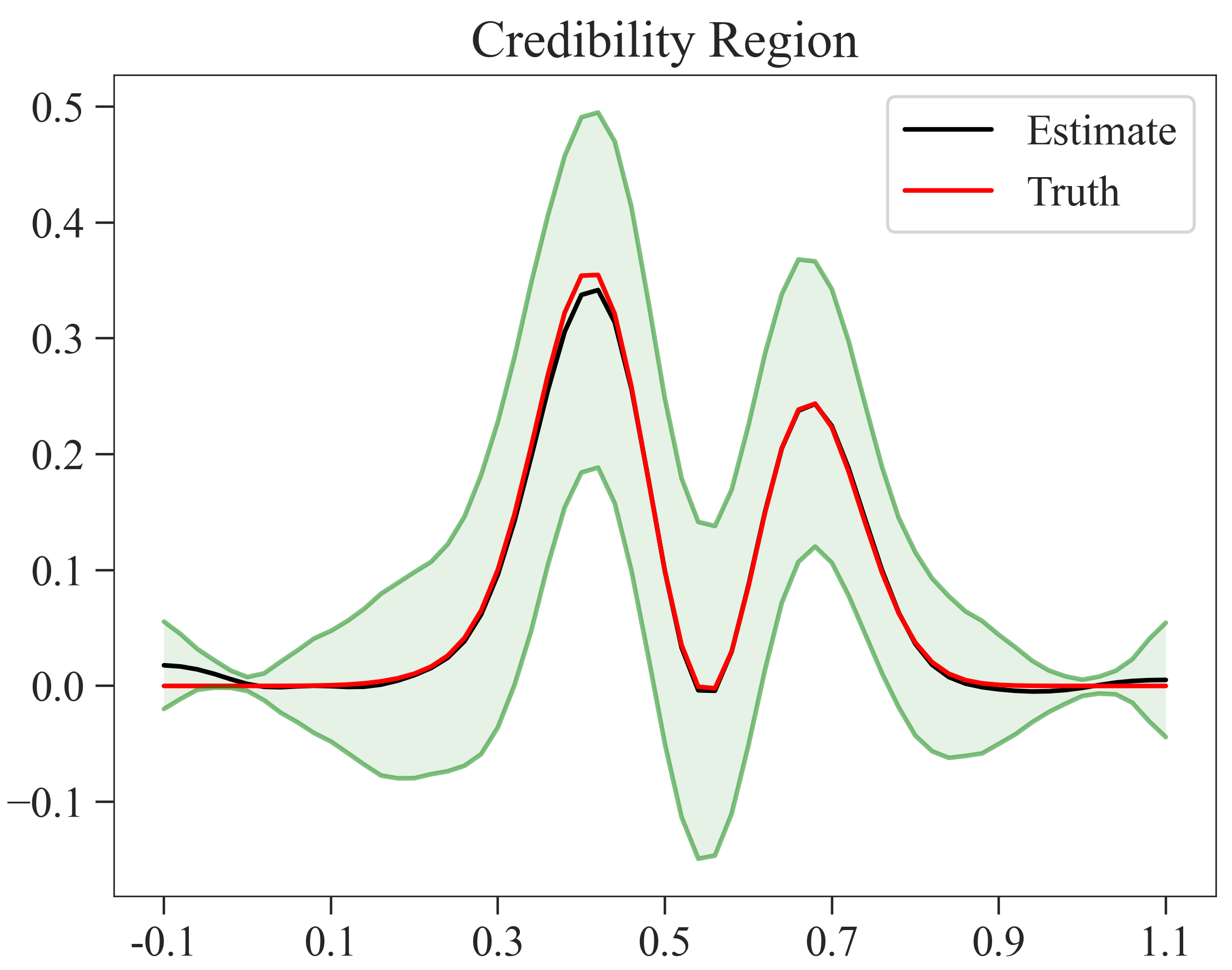}} 
	\subfloat[Credibility region $\text{II}$]{
		\includegraphics[ keepaspectratio=true, width=0.32\textwidth, clip=true]{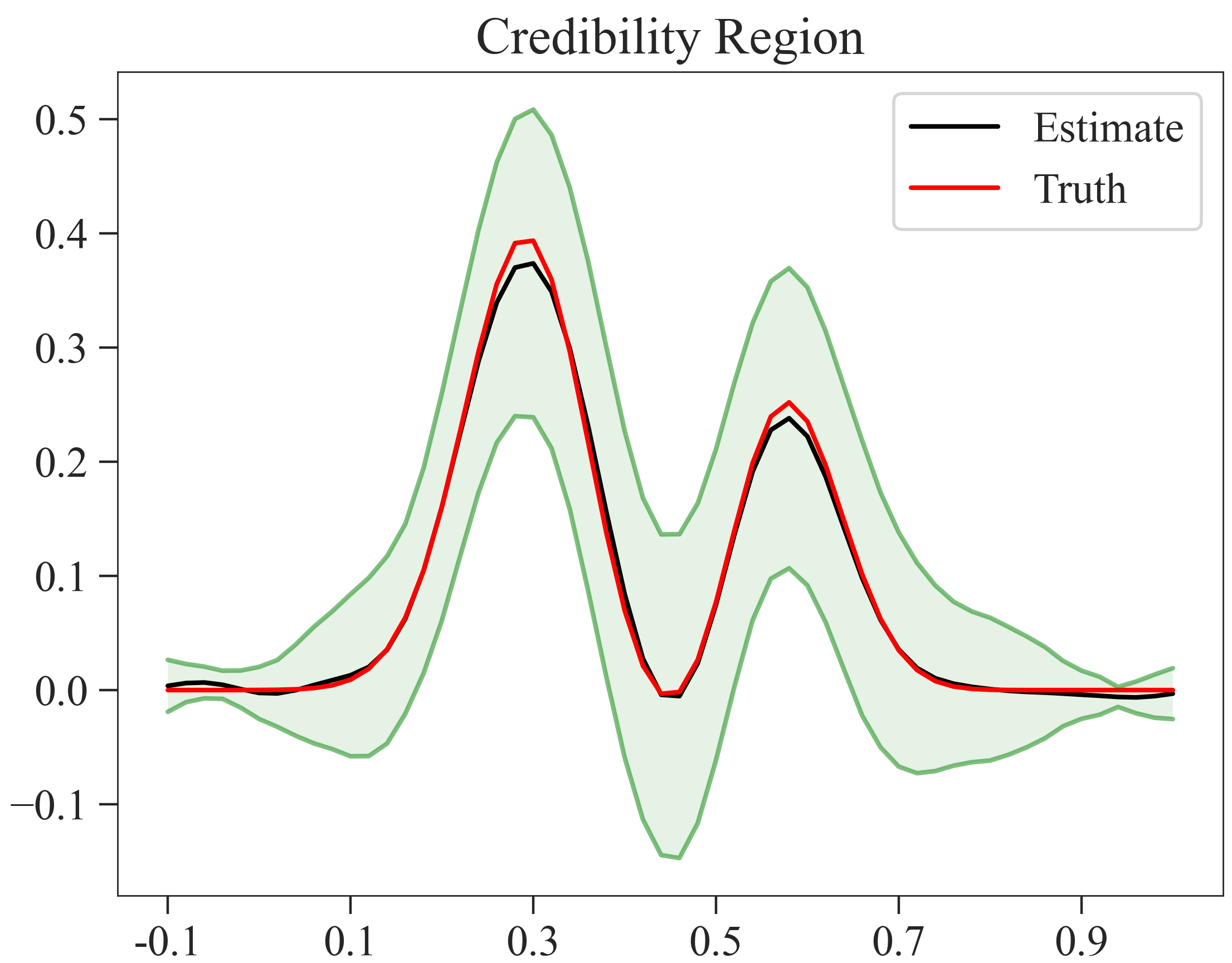}} 
	\subfloat[Credibility region $\text{III}$]{
		\includegraphics[ keepaspectratio=true, width=0.31\textwidth, clip=true]{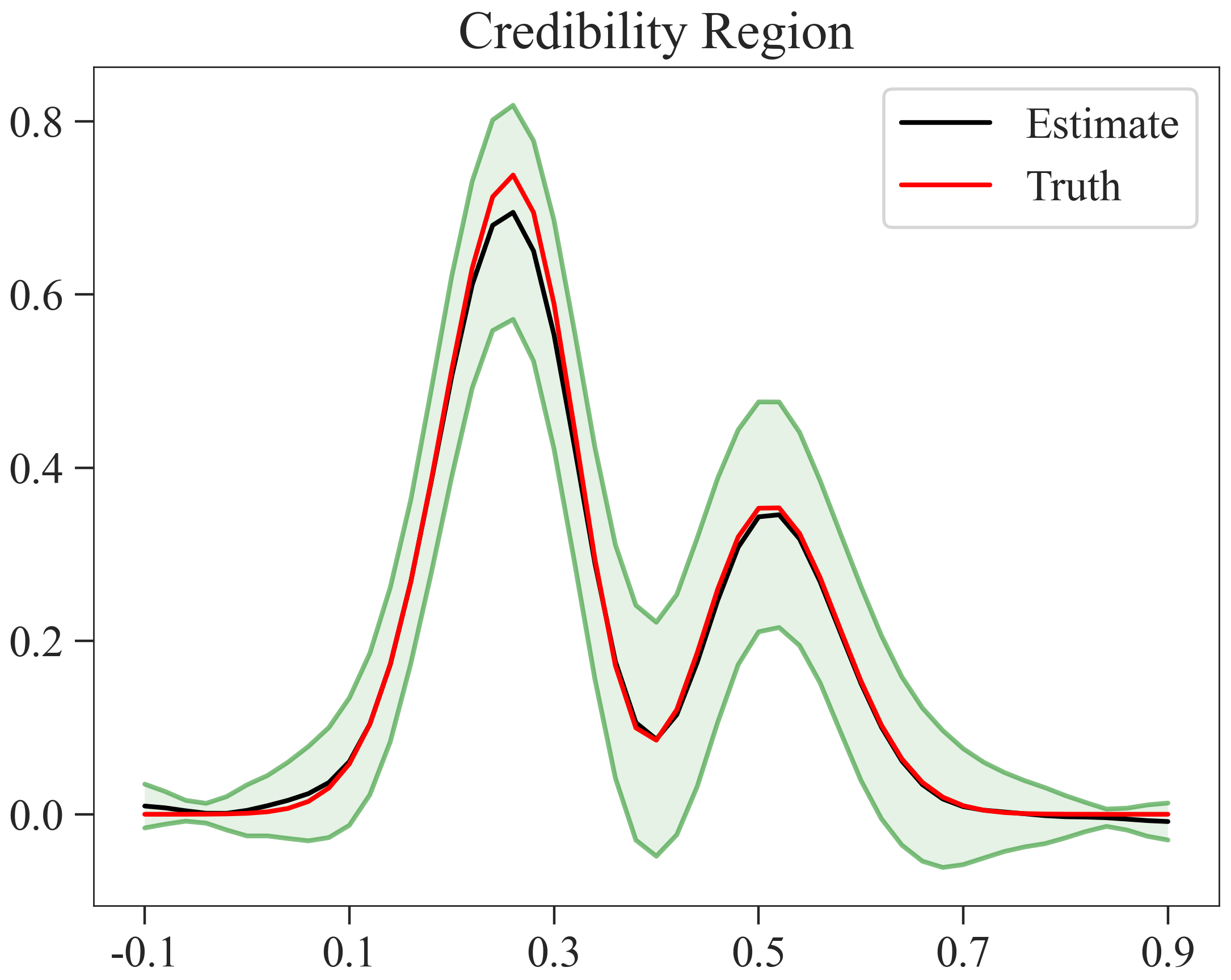}} 
	
	\caption{\emph{\small {
				The $95 \%$ credibility region of the estimated posterior mean function represented by the green shade area.
				(a): The black line represents the estimated mean $\lbrace u(x_i, x_i) \rbrace^{60}_{i=1}$, and red line represents the true function $\lbrace u^{\dagger}(x_i, x_i) \rbrace^{60}_{i=1}$.
				(b): The black line represents the estimated mean $\lbrace u(x_i, x_{i+5}) \rbrace^{55}_{i=1}$, and red line represents the true function $\lbrace u^{\dagger}(x_i, x_{i+5}) \rbrace^{55}_{i=1}$.
				(c): The black line represents the estimated mean $\lbrace u(x_i, x_{i+10}) \rbrace^{50}_{i=1}$, and red line represents the true function $\lbrace u^{\dagger}(x_i, x_{i+10}) \rbrace^{50}_{i=1}$. }}}
	
	\label{fig:Helmvar}
	
\end{figure}

{
\noindent \textbf{Discussion of $u$}:

Firstly, in the sub-figures (a) and (b) of Figure $\ref{fig:Helmcomparison}$, we show the background truth and the estimated posterior mean function of $u$.
The estimated posterior mean function of $u$ is visually similar to the true function.
In sub-figure (a) of Figure $\ref{fig:Helmrelative}$, we draw the relative error curve calculated in $L^2$-norm, defined in $(\ref{equ:relative})$.
The relative error curve illustrates that the convergence speed is fast since the descending trend is rapid at first $20$ steps, and the relative error is stable around $0.6 \%$ at the end of the iteration steps.
This provides quantitative evidence that the estimated posterior mean function is similar to the background truth of $u$.
As a result, combining the visual (sub-figures (a), (b) of Figure $\ref{fig:Helmcomparison}$) and quantitative (relative errors shown in sub-figure (a) of Figure $\ref{fig:Helmrelative}$) evidence, we say that the NCP-iMFVI method provides an estimated posterior mean function of the parameter $u$ which is similar to the background truth.

Secondly, we provide some discussions of the estimated posterior covariance functions about the parameter $u$.
In sub-figure (c) of Figure $\ref{fig:Helmcomparison}$, we draw the point-wise variance field of the posterior measure $u$.
To provide detailed evidence, we draw the comparisons of the estimated mean function and back-ground truth calculated by different mesh points in Figure $\ref{fig:Helmvar}$, which are given by $\lbrace (x_i, x_i) \rbrace^{n}_{i=1}$, $\lbrace (x_i, x_{i+5}) \rbrace^{n-5}_{i=1}$, and $\lbrace (x_i, x_{i+10}) \rbrace^{n-10}_{i=1}$($n = 60$ according to the mesh size $= 60 \times 60$).
And the green shade area represents the $95 \%$ credibility region according to estimated posterior mean function.
In the sub-figures (a), (b), and (c) of Figure $\ref{fig:Helmvar}$, the black lines represent the estimated mean $\lbrace u(x_i, x_i) \rbrace^{60}_{i=1}$, $\lbrace u(x_i, x_{i+5}) \rbrace^{55}_{i=1}$, and $\lbrace u(x_i, x_{i+10}) \rbrace^{50}_{i=1}$, respectively.
While red lines represent the true function $\lbrace u^{\dagger}(x_i, x_i) \rbrace^{60}_{i=1}$, $\lbrace u^{\dagger}(x_i, x_{i+5}) \rbrace^{55}_{i=1}$, and $\lbrace u^{\dagger}(x_i, x_{i+10}) \rbrace^{50}_{i=1}$, respectively.
We see that the credibility region contains background truth, which indicates that the Bayesian setup is meaningful and in accordance with the frequentist theoretical investigations of the posterior consistency \cite{wang2019frequentist, zhang2020convergence}. 
Combining with the variance shown in sub-figure (c) of Figure $\ref{fig:Helmcomparison}$, we say that the NCP-iMFVI method quantifies the uncertainties of the parameter $u$.

\noindent \textbf{Discussion of $\lambda$}:

In sub-figure (b) of Figure $\ref{fig:Helmrelative}$, we draw the step values of $\lambda$.
At the beginning of the iteration process, the descending trend is rapid, which is similar to the trend of the relative errors shown in sub-figure (a) of Figure $\ref{fig:Helmrelative}$.
During the iteration process, we see that the descending trend of $\lambda$ gradually becomes gentle.
Recalling the stop criteria (Step 6) of Algorithm $\ref{alg A}$, the step error of $\lambda$ is small when the iteration process is stopped.
This indicates that the hyper-parameter $\lambda$ is converged.
$\lambda$ is finally converged to $40.014$ within $50$ steps based on sub-figure (b) of Figure $\ref{fig:Helmrelative}$.

\noindent \textbf{Mesh independence}:

At last, we illustrate the mesh independence of the NCP-iMFVI method, as expected for the ``Bayesianize-then-discretize'' approach.
We show the step norm curves obtained by each discrete level $ n = \lbrace 3600, 4225, 4900, 5625, 6400 \rbrace$ in sub-figure (c) of Figure $\ref{fig:Helmrelative}$.
Although there are differences between each step norm curves, they perform the same descending trend for the different discretized dimensions.
Thus we say that the convergence speed is not affected by discretized dimensions, demonstrating that the NCP-iMFVI method has mesh independence property.
}

\begin{figure}
	\centering
	
	\subfloat[Relative errors]{
		\includegraphics[ keepaspectratio=true, width=0.324\textwidth, clip=true]{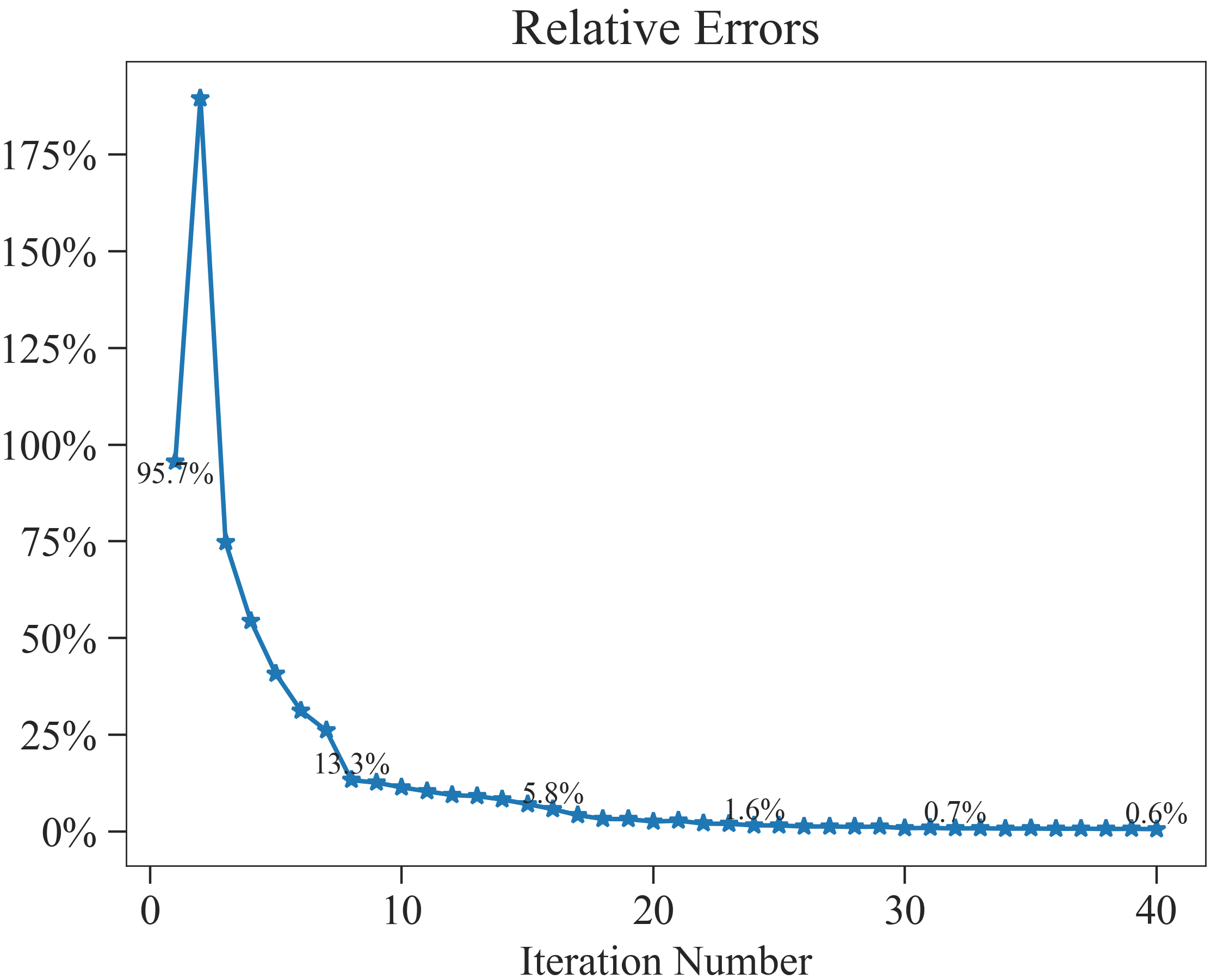}} 
	\subfloat[Step values of $\lambda$]{
		\includegraphics[ keepaspectratio=true, width=0.3095\textwidth, clip=true]{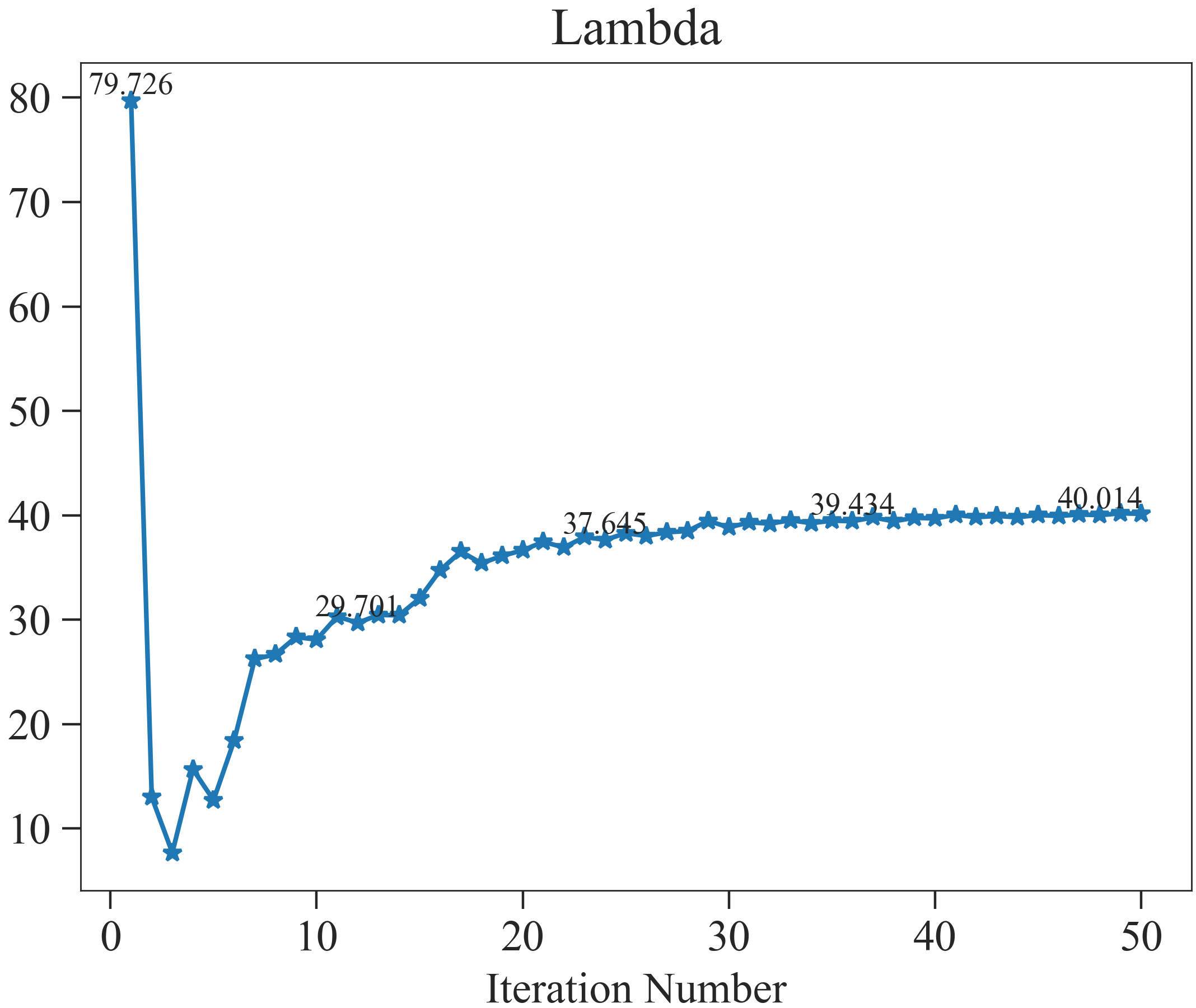}}
	\subfloat[Step norms]{
		\includegraphics[ keepaspectratio=true, width=0.313\textwidth, clip=true, trim=3pt 4pt 2pt 0pt]{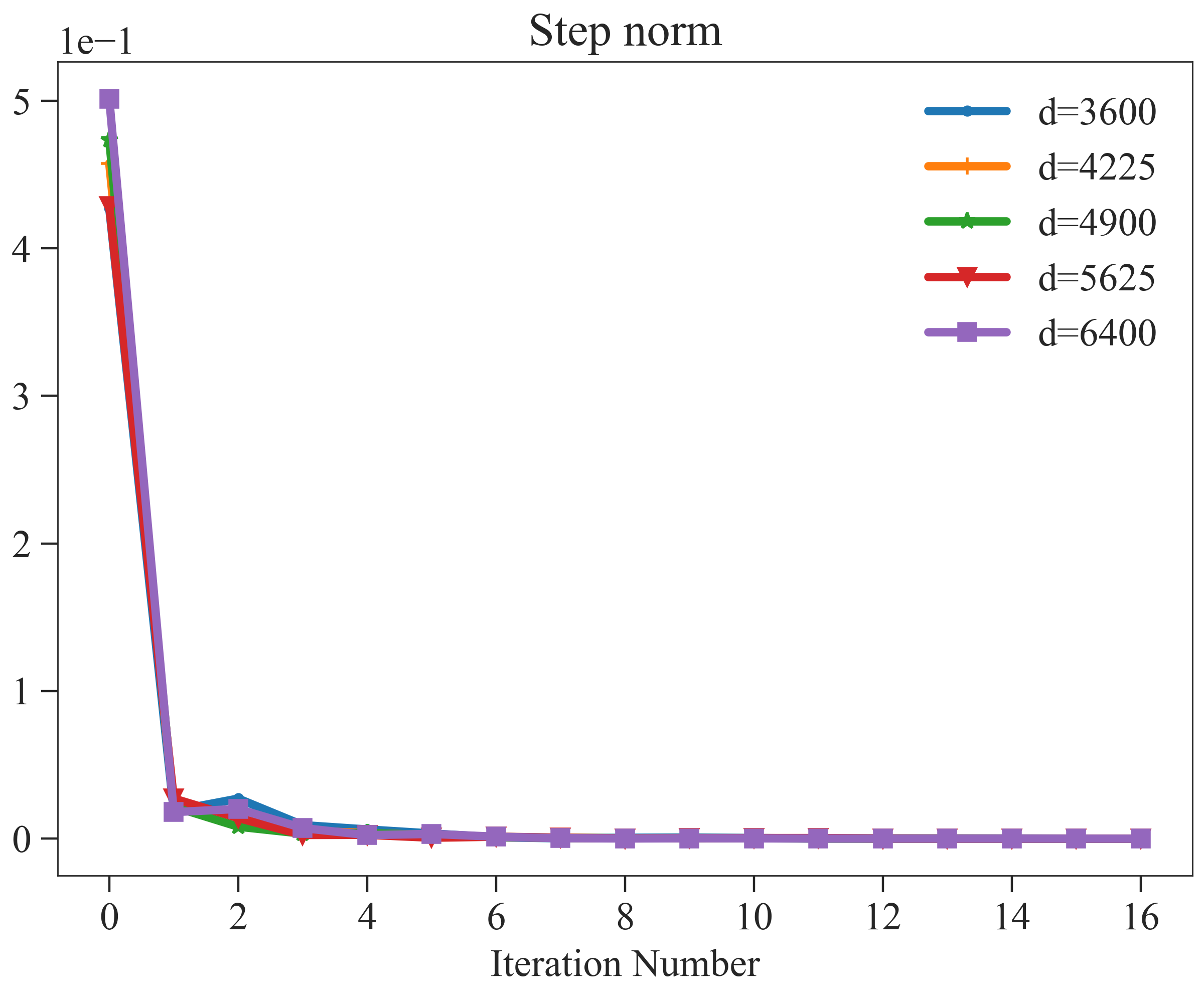}} 
	
	\caption{\emph{\small (a): Relative error of the estimated posterior means in $L^2$-norm under the mesh size { $60 \times 60$}; (b): The step values of $\lambda$ obtained by NCP-iMFVI; (c): Logarithm of the step norms computed by NCP-iMFVI method with different discretized dimensions $n = \lbrace3600, 4225, 4900, 5625, 6400 \rbrace$. }}
	
	\label{fig:Helmrelative}
	
\end{figure}

\subsection{Non-linear inverse problem with steady-state Darcy flow equation}\label{subsec3.3}
\subsubsection{Basic settings}
In this section, we concentrate on the inverse problem of estimating the permeability distribution in a porous medium from a discrete set of pressure measurements, studied in \cite{calvetti2018iterative}. Consider the following steady-state Darcy flow equation:
\begin{align}
\begin{split}
 -\nabla \cdot (e^u\nabla w) &= f \quad x \in \Omega, \\
 w &= 0 \quad x \in \partial \Omega,
\end{split}
\end{align}
where $f \in H^{-1}(\Omega)$ is the source function,  $u \in \mathcal{X} := L^{\infty}(\Omega)$ called log-permeability for the computational area $\Omega = (0, 1)^2$, and denote $\bm{x} = (x, y) \in \mathbb{R}^2$. 
Then the forward operator has the following form:
\begin{align}
 Hu = (w(\bm{x}_1), w(\bm{x}_2), \cdots, w(\bm{x}_{N_d}))^T,
\end{align}
where $x_i \in \Omega$ for $i = 1, \cdots, N_d$. 

\begin{figure}
	\centering
	
	\subfloat[Background truth]{
		\includegraphics[ keepaspectratio=true, width=0.32\textwidth, clip=true]{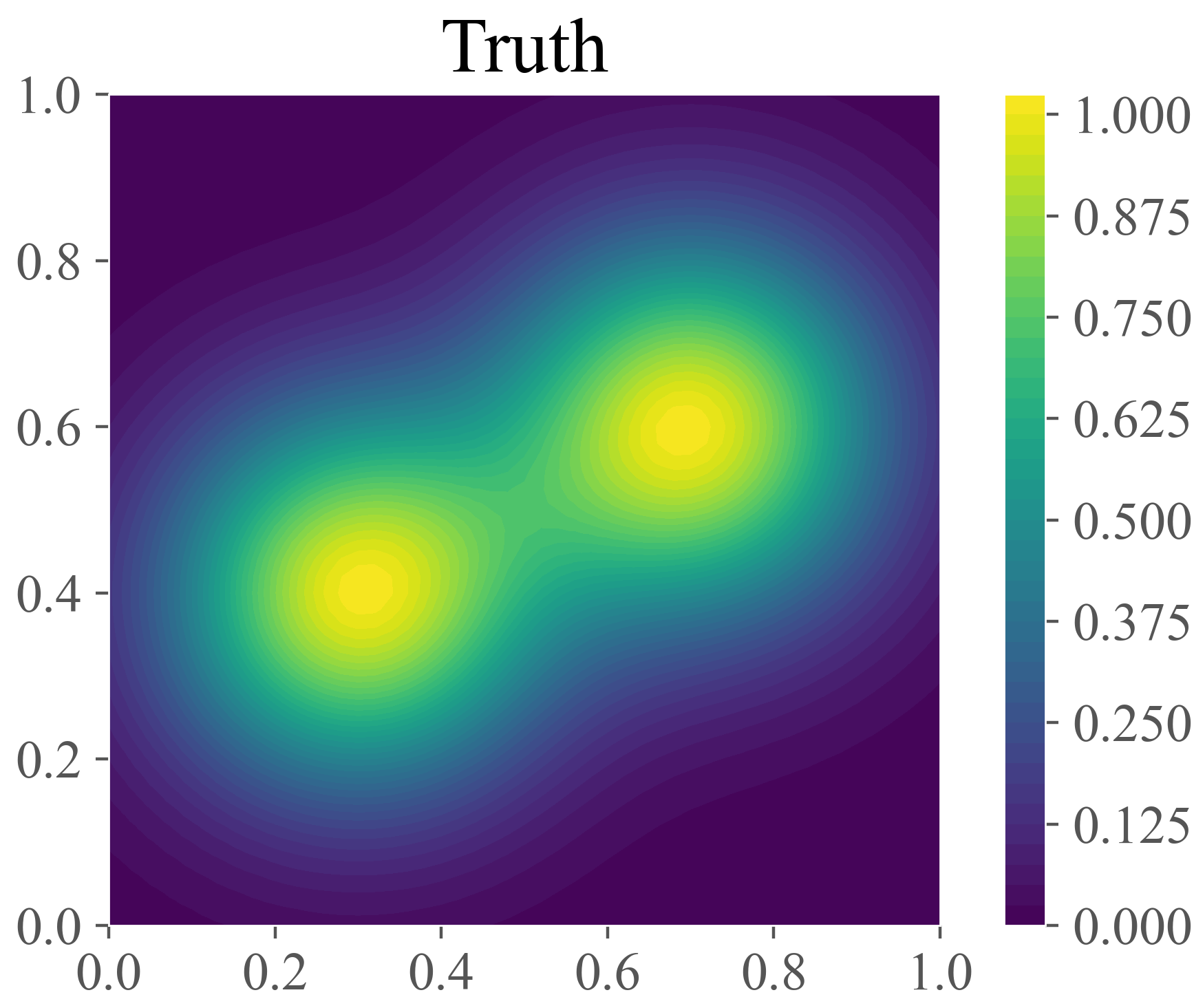}} 
	\subfloat[Estimated mean]{
		\includegraphics[ keepaspectratio=true, width=0.31\textwidth, clip=true]{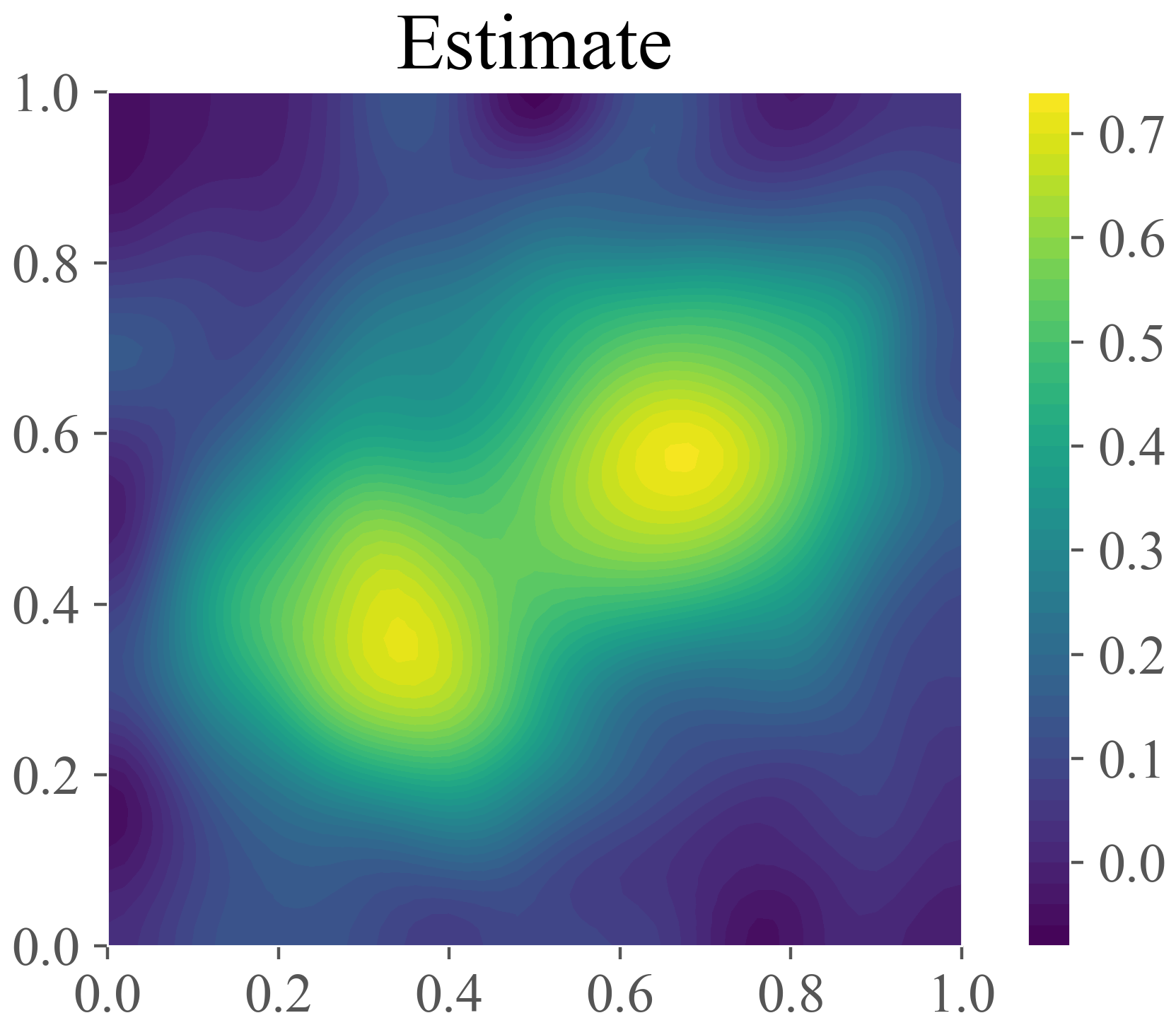}} 
	\subfloat[Estimated variance]{
		\includegraphics[ keepaspectratio=true, width=0.32\textwidth, clip=true]{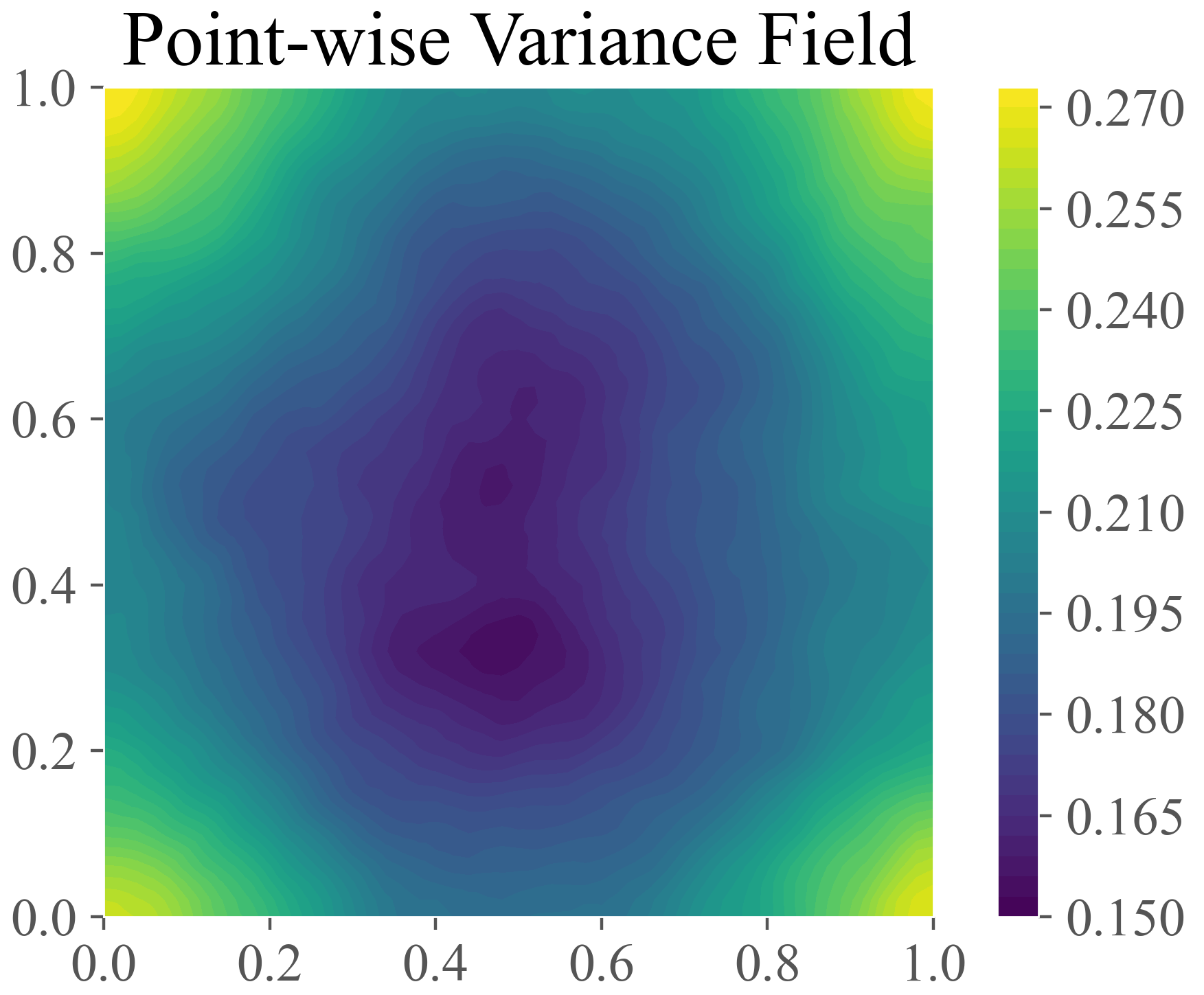}}
	
	\caption{\emph{\small (a): The background truth of $u$; (b): The estimated posterior mean function of $u$ obtained by NCP-iMFVI; (c): The estimated variance function of posterior measure of $u$ obtained by NCP-iMFVI.}}
	
	\label{fig:Darcycomparison}
	
\end{figure}

Let us assume that the operator $\mathcal{F}$ is Fr\'{e}chet differentiable, and $w = \mathcal{F}(u)$. 
{
We linearize $\mathcal{F}(u)$ around $u^{*}$ to obtain
\begin{align}
 w \approx \mathcal{F}(u^{*}) + F^{\prime}(u - u^{*}),
\end{align}
where $F^{\prime}$ is the Fr\'{e}chet derivative of $\mathcal{F}(u)$ evaluated at $u^{*}$. Consequently, we transform the non-linear problem into the linear form, and we are able to employ our NCP-iMFVI method. 

We now denote $\delta w := F^{\prime}(\delta u)$, $w_0 = \mathcal{F}(u_{\text{MAP}})$, and $\delta u := u - u^{*} = \lambda v$.
Then we obtain the linearized equation as follows:
\begin{align}
\begin{split}
 -\nabla \cdot (e^{u^{*}}\nabla w_0 \cdot \delta u) &= \nabla \cdot (e^{u^{*}} \nabla \delta w) \quad x \in \Omega, \\
 \delta u &= 0 \quad x \in \partial \Omega,
\end{split}
\end{align}
and the abstract form turns to 
\begin{align}
\begin{split}
 \bm{d} &= \mathcal{F}(u^{*}) + \delta w + \bm{\epsilon},
\end{split}
\end{align}
where $\bm{\epsilon} \sim \mathcal{N}(0, \bm{\Gamma}_{\text{noise}})$ is the random Gaussian noise.}
The parameters $v$ and $\lambda$ are generated by the Gaussian measures $\mathcal{N}(0, \mathcal{C}_0)$ and $\mathcal{N}(\bar{\lambda}, \sigma)$, where { $\bar{\lambda} = 1, \sigma = 100$}, respectively.
For clarity, we list the specific choices for some parameters introduced in this subsection as follows: 
\begin{itemize}
	\item {Assume that 5$\%$ random Gaussian noise $\epsilon \sim \mathcal{N}(0, \bm{\Gamma}_{\text{noise}})$ is added, where $\bm{\Gamma}_{\text{noise}} = \tau^{-1}\textbf{I}$, and $\tau^{-1} = (0.05\max(\lvert Hu\rvert))^2$.}
	\item We assume that the data produced from the underlying log-permeability
	\begin{align*}
		u^{\dagger} = \ &\exp(-20(x_1-0.3)^2 - 20(x_2-0.4)^2) \\
		&+ \exp(-20(x_1-0.7)^2 - 20(x_2-0.6)^2).
	\end{align*}
	\item { $\mathcal{F}(u)$ is linearized around $u^{*} = 0$}.
	\item Let domain $\Omega$ be a bounded area $(0, 1)^2$. The available data is discretized by the finite element method on a regular mesh with the number of grid points equal to $20 \times 20$.
	\item The operator $\mathcal{C}_0$ is given by $\mathcal{C}_0 = (\text{I} - \alpha\Delta)^{-2}$, where $\alpha = 0.05$ is a fixed constant. Here the Laplace operator is defined on $\Omega$ with zero Neumann boundary condition. 
	\item In order to avoid the inverse crime, a fine mesh with the number of grid points equal to $500 \times 500$ is employed for generating the data. For the inversion, a mesh with a number of grid points equal to {$50 \times 50$} is employed.
\end{itemize}

\begin{figure}
	\centering
	
	\subfloat[Credibility region $\text{I}$]{
		\includegraphics[ keepaspectratio=true, width=0.32\textwidth, clip=true, trim=24pt 18pt 30pt 22pt]{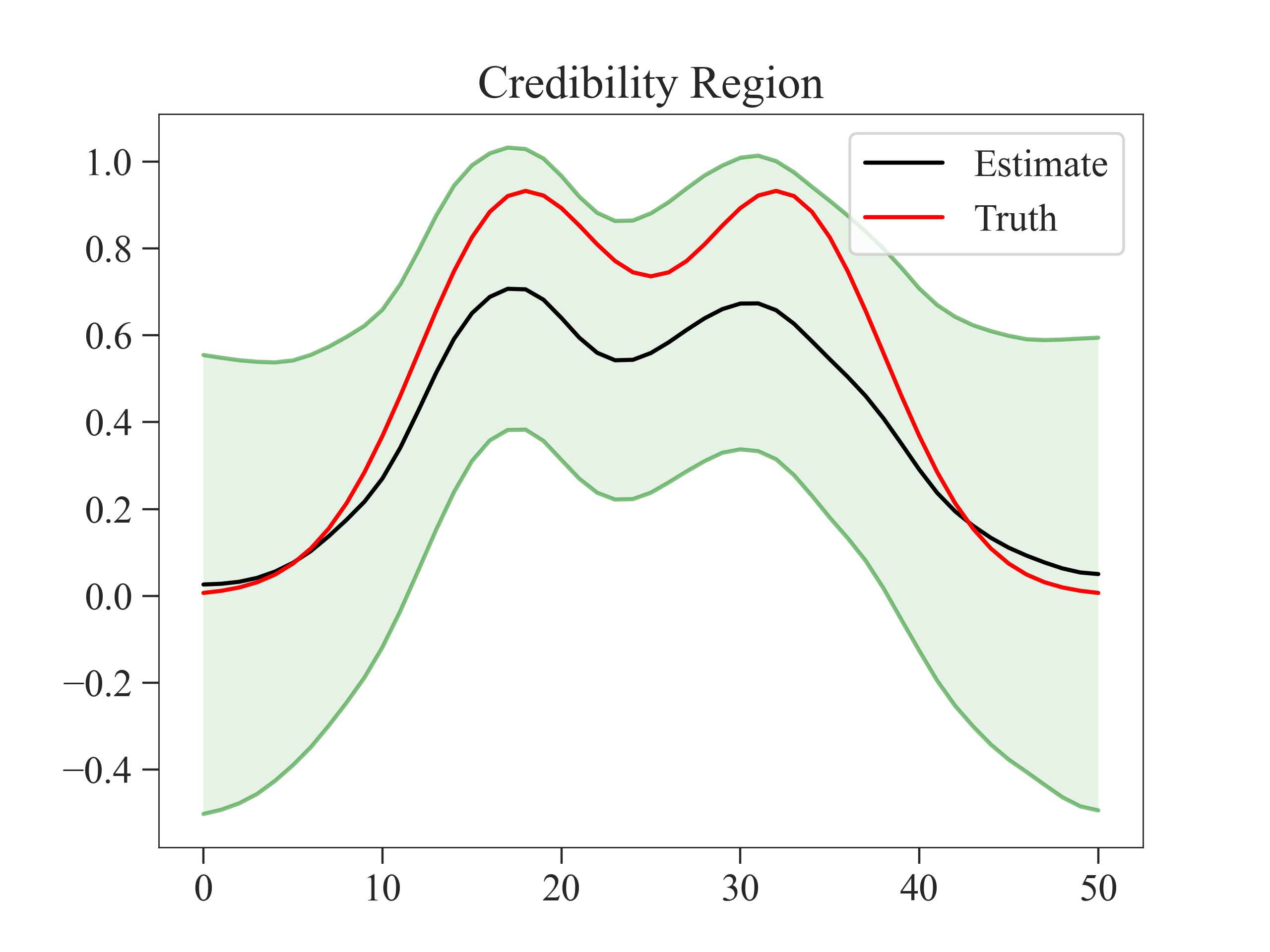}} 
	\subfloat[Credibility region $\text{II}$]{
		\includegraphics[ keepaspectratio=true, width=0.32\textwidth, clip=true, trim=24pt 18pt 30pt 22pt]{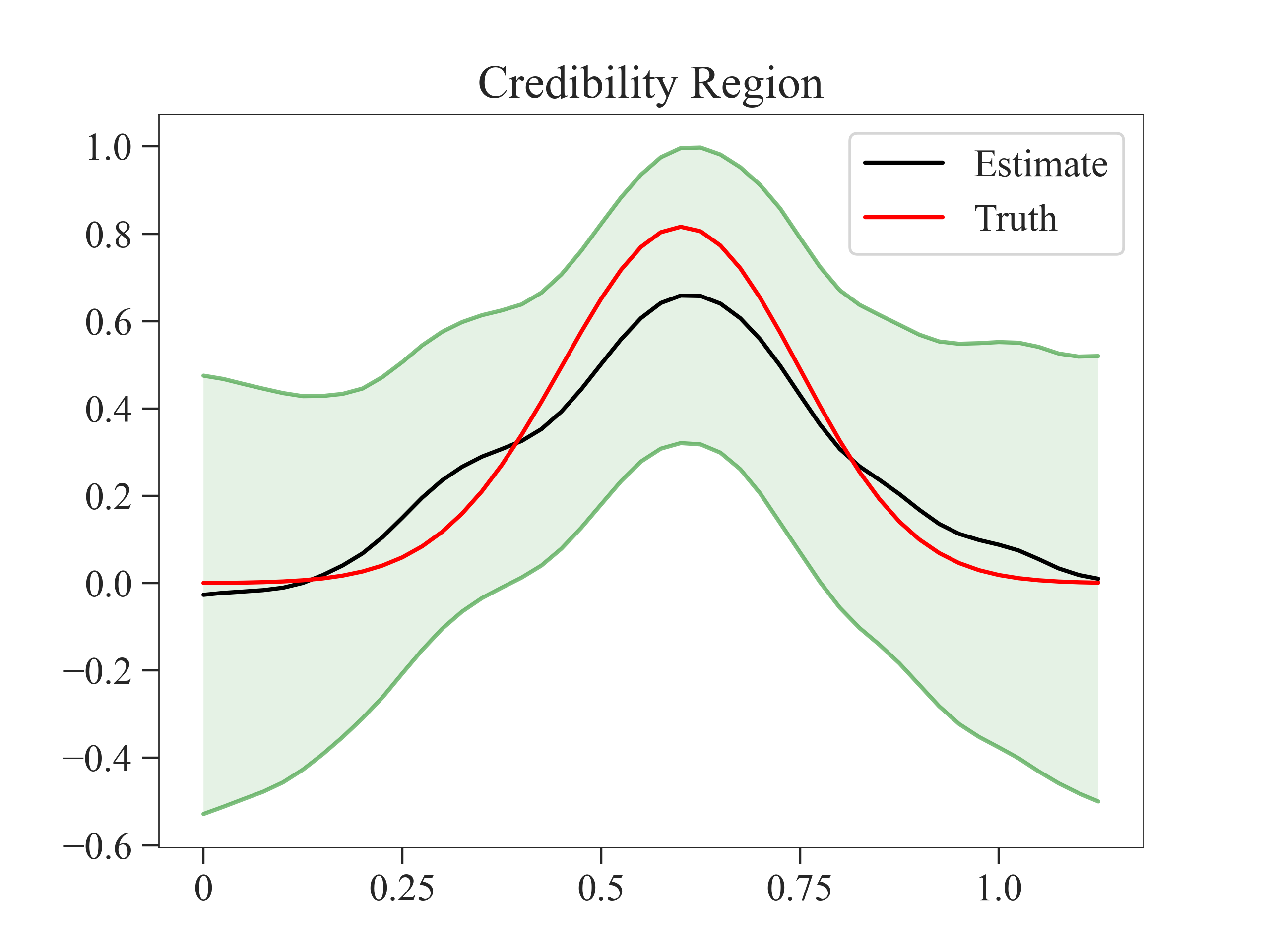}} 
	\subfloat[Credibility region $\text{III}$]{
		\includegraphics[ keepaspectratio=true, width=0.32\textwidth, clip=true, trim=24pt 18pt 30pt 22pt]{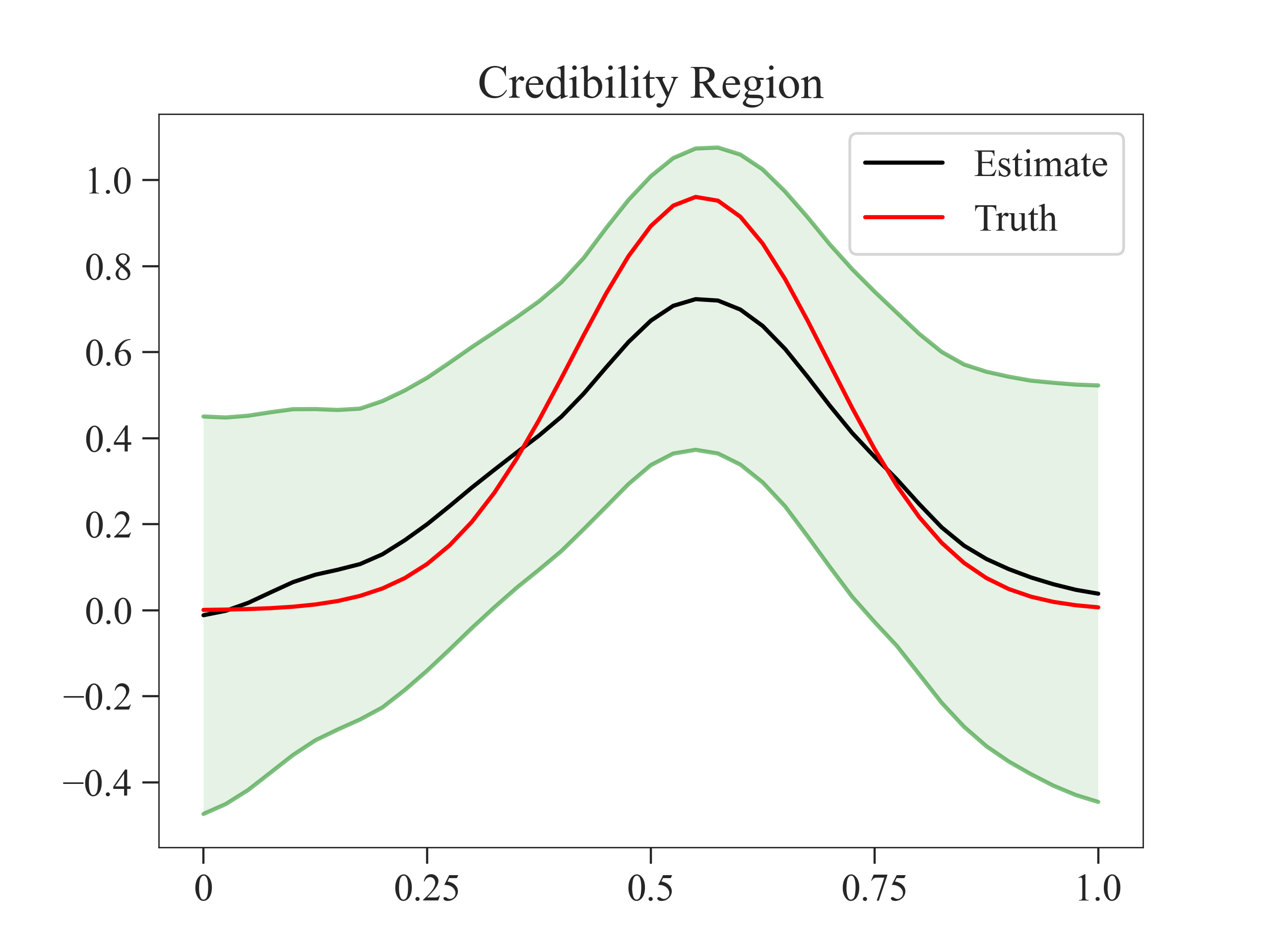}} 
	
	\caption{\emph{\small {
				The $95 \%$ credibility region of the estimated posterior mean function represented by the green shade area.
				(a): The black line represents the estimated mean $\lbrace u(x_i, x_i) \rbrace^{50}_{i=1}$, and red line represents the true function $\lbrace u^{\dagger}(x_i, x_i) \rbrace^{50}_{i=1}$.
				(b): The black line represents the estimated mean $\lbrace u(x_i, x_{i+5}) \rbrace^{45}_{i=1}$, and red line represents the true function $\lbrace u^{\dagger}(x_i, x_{i+5}) \rbrace^{45}_{i=1}$.
				(c): The black line represents the estimated mean $\lbrace u(x_i, x_{i+10}) \rbrace^{40}_{i=1}$, and red line represents the true function $\lbrace u^{\dagger}(x_i, x_{i+10}) \rbrace^{40}_{i=1}$. }}}
	
	\label{fig:Darcyvar}
	
\end{figure}

\subsubsection{Numerical results}

{
\noindent \textbf{Discussion of $u$}:

Firstly, in the sub-figures (a) and (b) of Figure $\ref{fig:Darcycomparison}$, we show the background truth and estimated posterior mean function of $u$.
The estimated posterior mean function of $u$ is visually similar to the true function.
In sub-figure (a) of Figure $\ref{fig:Darcyrelative}$, we draw the relative error curve calculated in $L^2$-norm, defined in $(\ref{equ:relative})$.
The relative error curve illustrates that the convergence speed is fast since the descending trend is rapid at first $10$ steps, and the relative error is stable around $8 \%$ at the end of the iteration steps.
This provides quantitative evidence that the estimated posterior mean function is similar to the background truth of $u$.
As a result, combining the visual (sub-figures (a), (b) of Figure $\ref{fig:Darcycomparison}$) and quantitative (relative errors shown in sub-figure (a) of Figure $\ref{fig:Darcyrelative}$) evidence, we say that the NCP-iMFVI method provides an estimated posterior mean function of the parameter $u$ which is similar to the background truth.

Secondly, we provide some discussions of the estimated posterior covariance functions about the parameter $u$.
In sub-figure (c) of Figure $\ref{fig:Darcycomparison}$, we draw the point-wise variance field of the posterior measure $u$.
To provide detailed evidence, we draw the comparisons of the estimated mean function and back-ground truth calculated by different mesh points in Figure $\ref{fig:Darcyvar}$, which are given by $\lbrace (x_i, x_i) \rbrace^{n}_{i=1}$, $\lbrace (x_i, x_{i+5}) \rbrace^{n-5}_{i=1}$, and $\lbrace (x_i, x_{i+10}) \rbrace^{n-10}_{i=1}$($n = 50$ according to the mesh size $= 50 \times 50$).
And the green shade area represents the $95 \%$ credibility region according to estimated posterior mean function.
In the sub-figures (a), (b), and (c) of Figure $\ref{fig:Darcyvar}$, the black lines represent the estimated mean $\lbrace u(x_i, x_i) \rbrace^{50}_{i=1}$, $\lbrace u(x_i, x_{i+5}) \rbrace^{45}_{i=1}$, and $\lbrace u(x_i, x_{i+10}) \rbrace^{40}_{i=1}$, respectively.
While red lines represent the true function $\lbrace u^{\dagger}(x_i, x_i) \rbrace^{50}_{i=1}$, $\lbrace u^{\dagger}(x_i, x_{i+5}) \rbrace^{45}_{i=1}$, and $\lbrace u^{\dagger}(x_i, x_{i+10}) \rbrace^{40}_{i=1}$, respectively.
We see that the credibility region contains background truth, which indicates that the Bayesian setup is meaningful and in accordance with the frequentist theoretical investigations of the posterior consistency \cite{wang2019frequentist, zhang2020convergence}. 
Combining with the variance function, we say that the proposed NCP-iMFVI method quantifies the uncertainties of the parameter $u$.

\noindent \textbf{Discussion of $\lambda$}:

In sub-figure (b) of Figure $\ref{fig:Darcyrelative}$, we draw the step values of $\lambda$.
At the beginning of the iteration process, the descending trend is rapid, which is similar to the trend of the relative errors shown in sub-figure (a) of Figure $\ref{fig:Darcyrelative}$.
During the iteration process, we see that the descending trend of $\lambda$ gradually becomes gentle.
Recalling the stop criteria (Step 6) of Algorithm $\ref{alg A}$, the step error of $\lambda$ is small when the iteration process is stopped.
This indicates that the hyper-parameter $\lambda$ is converged.
$\lambda$ is finally converged to $36.711$ within $126$ steps based on sub-figure (b) of Figure $\ref{fig:Darcyrelative}$.

\noindent \textbf{Mesh independence}:

At last, we illustrate the mesh independence of the NCP-iMFVI method, as expected for the ``Bayesianize-then-discretize'' approach.
We show the step norm curves obtained by each discrete level $ n = \lbrace 1600, 2025, 2500, 3025, 3600 \rbrace$ in sub-figure (c) of Figure $\ref{fig:Darcyrelative}$.
Although there are differences between each step norm curves, they perform the same descending trend for the different discretized dimensions.
Thus we say that the convergence speed is not affected by discretized dimensions, demonstrating that the NCP-iMFVI method has mesh independence property.

}

\begin{figure}
	\centering
	
	\subfloat[Relative errors]{
		\includegraphics[ keepaspectratio=true, width=0.32\textwidth, clip=true]{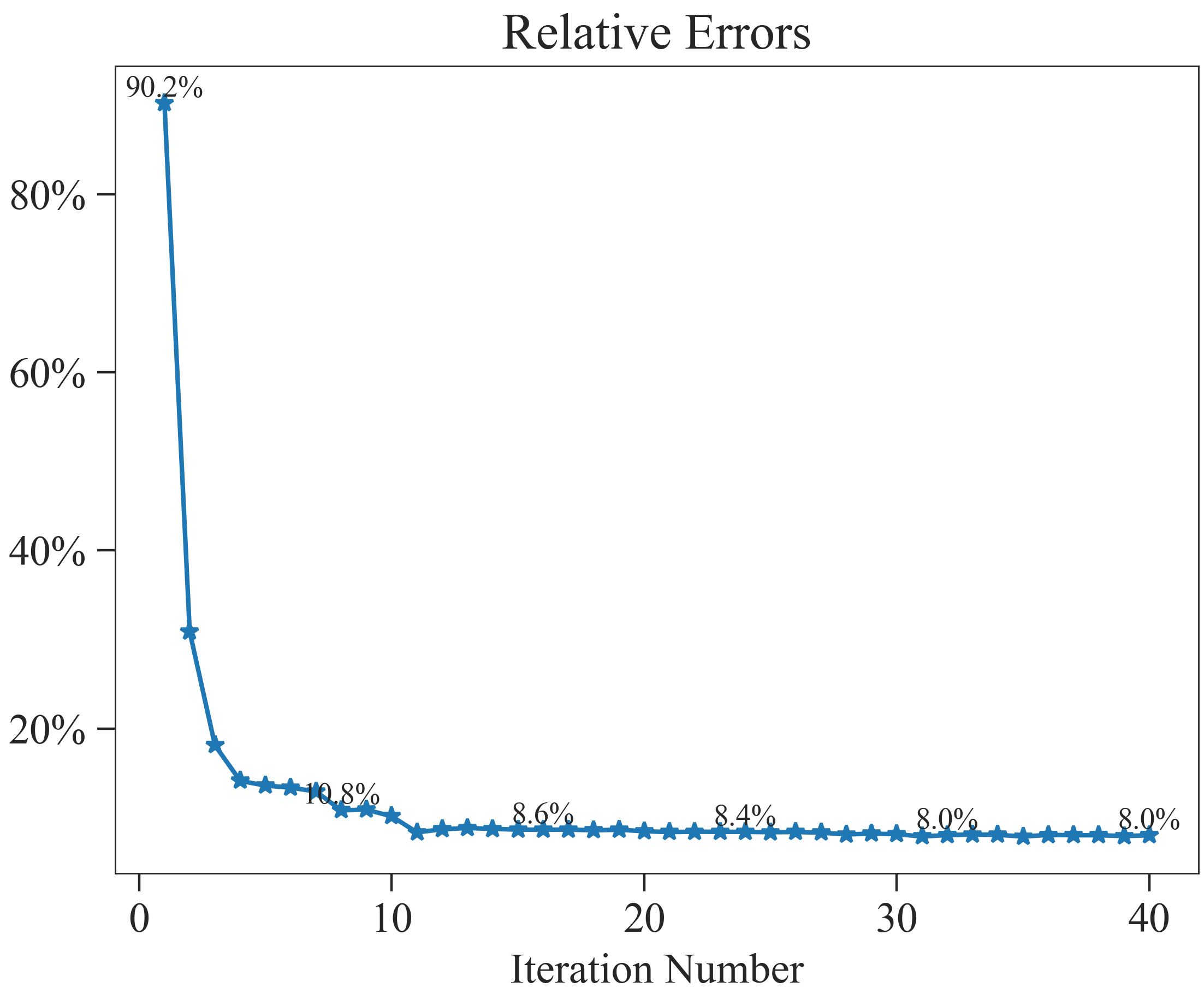}} 
	\subfloat[Step values of $\lambda$]{
		\includegraphics[ keepaspectratio=true, width=0.314\textwidth, clip=true]{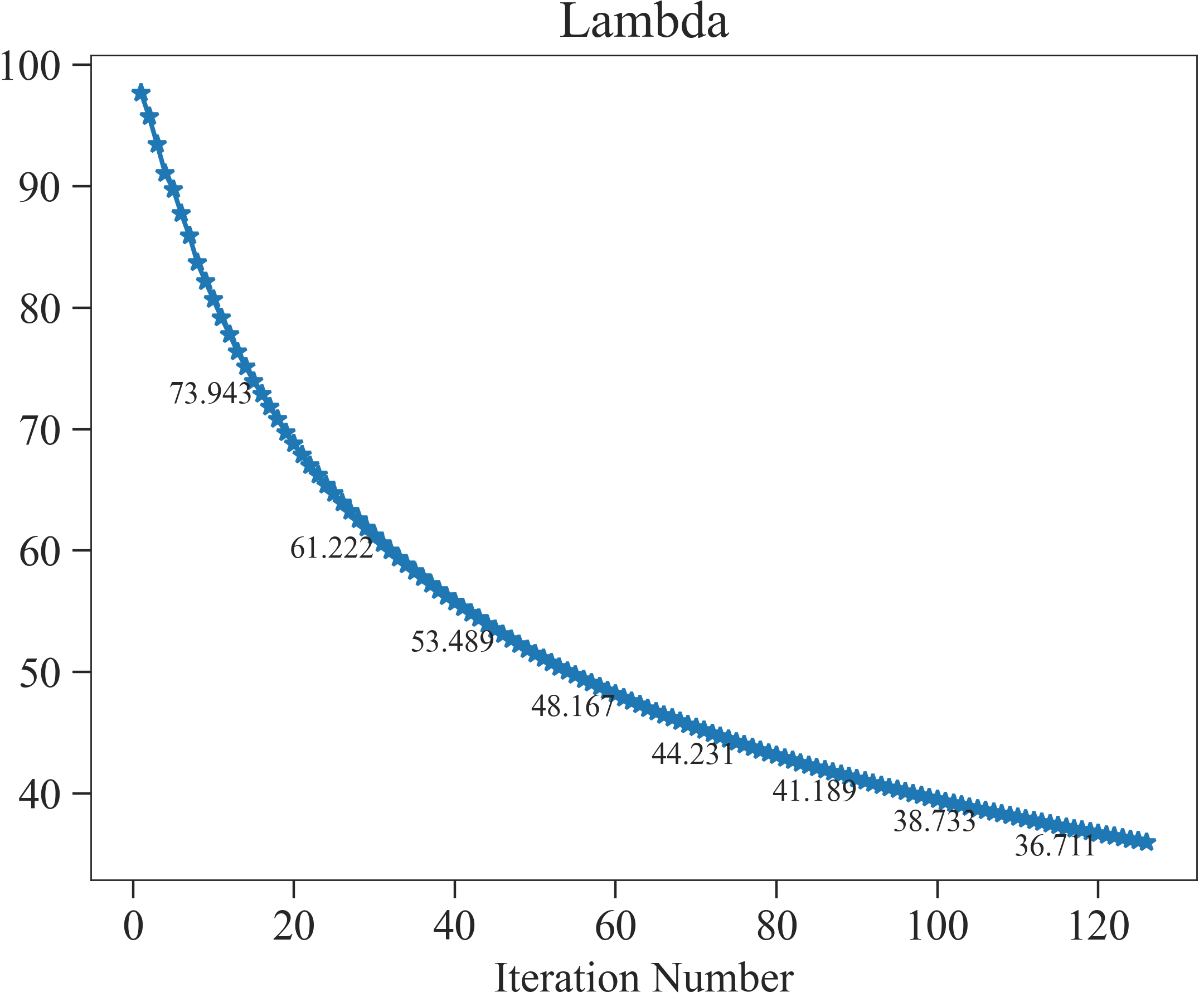}} 
	\subfloat[Step norms]{
		\includegraphics[ keepaspectratio=true, width=0.3142\textwidth, clip=true, trim=3pt 4pt 2pt 0pt]{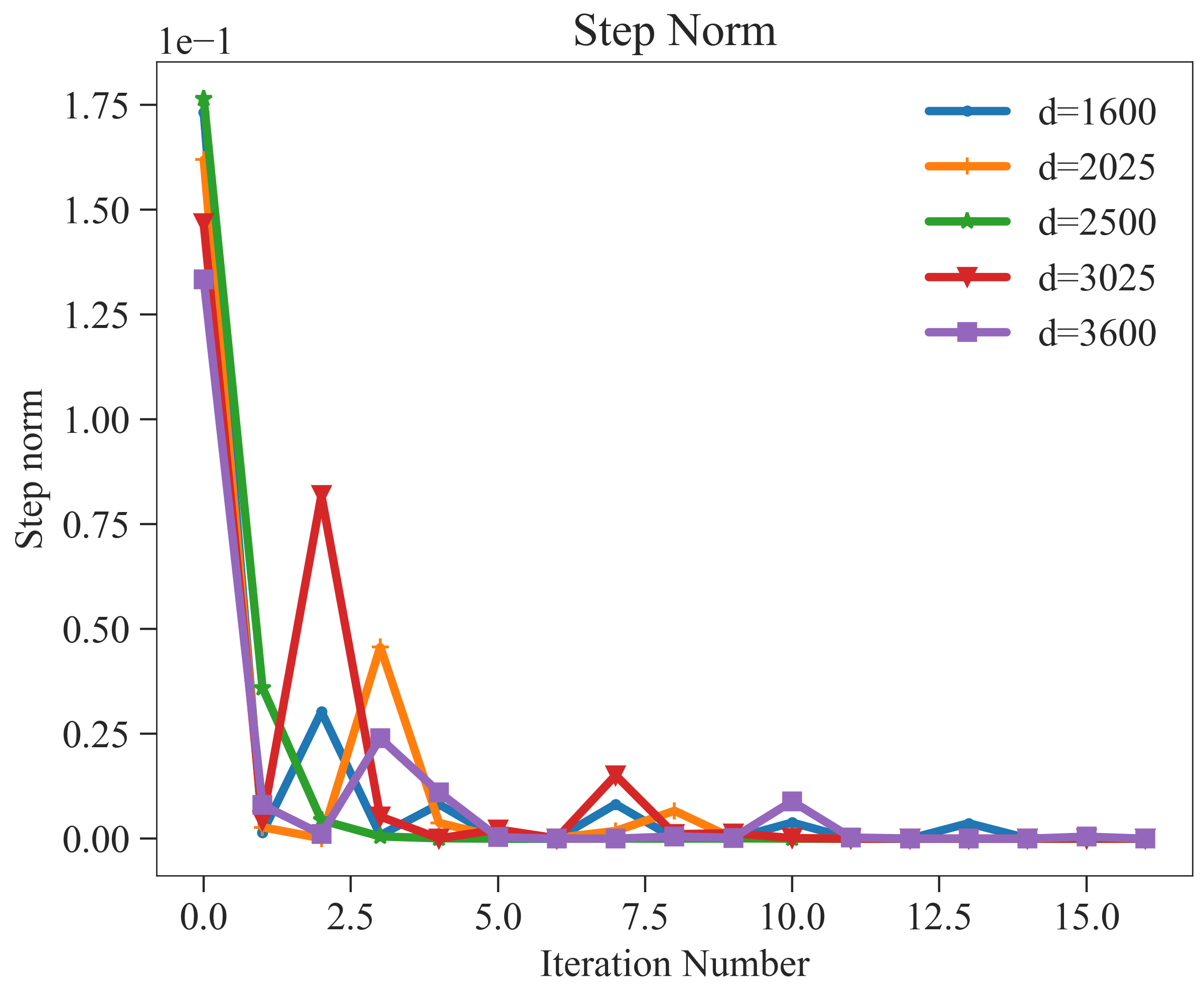}} 
	
	\caption{\emph{\small (a): Relative error of the estimated posterior means in $L^2$-norm under the mesh size {$50 \times 50$}; (b): The step values of $\lambda$ obtained by NCP-iMFVI; (c): Logarithm of the step norms computed by NCP-iMFVI method with different discretized dimensions $n = \lbrace1600, 2025, 2500, 3025, 3600 \rbrace$. }}
	
	\label{fig:Darcyrelative}
	
\end{figure}

\section{Conclusion}\label{sec4}
In this paper, we construct the NCP-iMFVI method in the infinite-dimensional space, which provides an efficient computational method for applying the iMFVI approach based on the hierarchical Bayesian model. 
The NCP-iMFVI method avoids the obstacle of priors being mutually singular, which is caused by the hyper-parameter $\lambda$ changing.
The established NCP-iMFVI approach is applied to linear inverse problems with Gaussian noises, deriving an explicit form of the posterior measure.
We employed this approach to three inverse problems of the simple smooth equation, the multi-frequency Helmholtz equation, and the steady-state Darcy flow problem, respectively.
For all of the inverse problems, NCP-iMFVI provides an estimated posterior mean function of $u$ that is similar to the background truth, quantifies the uncertain information of $u$ based on the visual and quantitative evidence.
The descending trend of step values of hyper-parameter $\lambda$ is rapid at the beginning of the iteration process, which is similar to the trend of the relative errors.
And $\lambda$ is finally converged.
Moreover, the convergence speed is not affected by discretized dimensions, demonstrating that the NCP-iMFVI method has mesh independence property.

The current NCP-iMFVI method is based on the analysis under the hierarchical linear problem.
For solving the non-linear inverse problems, the NCP-iMFVI method can only manage the linearized part and provide an approximated posterior measure, as shown in Subsection $\ref{subsec3.3}$.
As a result, the approximated probability measure could be inaccurate for some highly non-linear problems.
In finite-dimensional spaces, the NCP formulation has been employed to solve a non-linear geostatistical inverse problem \cite{papaspiliopoulos2003non}. 
Thus developing the NCP-iMFVI method for solving non-linear inverse problems is worth investigating in future work.
On the other hand, under our mean-field assumption, the parameters are assumed to be independent of each other, and the degree of dependence of each parameter cannot be accurately portrayed.
To ameliorate this problem, the hierarchical VI, which is able to capture dependencies between parameters, has been studied in \cite{tran2015variational} under the finite-dimensional setting.
It is worthwhile developing the infinite-dimensional hierarchical VI methods that can reserve the dependencies of the parameters.

\section{Appendix}
\subsection{The infinite-dimensional variational inference theory}

In this section, we intend to offer an introduction to the infinite-dimensional variational inference theory, which is a brief version of \cite{jia2021variational}. We should highlight that we improve the statement of Theorem 11 in \cite{jia2021variational} in order to make the theory more suitable. 
Let the Bayesian formula on the Hilbert space be defined by
\begin{align}
	\frac{d\mu}{d\mu_0}(x) = \frac{1}{Z_{\mu}} \exp (-\Phi(x)),
\end{align}
where $\Phi(x):\mathcal{H} \rightarrow \mathbb{R}$ is a continuous function, and $\exp(-\Phi(x))$ is integrable with respect to $\mu_0$.
Here we denote that $\mu_0$ represents the prior measure, and $\mu$ is the posterior we intend to estimate on some Hilbert space $\mathcal{H}$. 
$Z_{\mu}$ is a Constant making sure that measure $\mu$ is indeed a probability measure. 
As our aim is to choose a closest measure $\nu$ to estimate $\mu$, the variational inference problem can be modeled as 
\begin{align}\label{eq1}
	\arg\min \limits_{\nu \in \mathcal{A}} D_{KL} (\nu \Arrowvert \mu),
\end{align}
where $\mathcal{A}\subset \mathcal{M}(\mathcal{H})$ is a set of``simpler'' measures that can be calculated efficiently, and $\mathcal{M}(\mathcal{H})$ is the Borel measure set on $\mathcal{H}$.
For a fixed Constant $M$, we assume the variable $x = (x_1, x_2, \cdots, x_M)$. Particularly, in the main text, we obtain that $M = 2$, and $x_1 = \lambda, x_2 = v$. We specify the Hilbert space $\mathcal{H}$ and subset $\mathcal{A}$ as 
\begin{align}
	\mathcal{H} = \prod^M_{j=1}\mathcal{H}_j, \quad \mathcal{A} = \prod^M_{j=1}\mathcal{A}_j
\end{align}
where $\mathcal{H}_j, j = 1, \cdots, M$ are a series of separable Hilbert space and $\mathcal{A}_j \subset \mathcal{M}(\mathcal{H}_j)$. Let $\nu := \prod^M_{i=1}\nu^i$ be a probability measure such that $\nu(dx) = \prod^M_{i=1}\nu^i(dx)$. With these assumptions, the minimization problem of ($\ref{eq1}$) can be rewritten as
\begin{align}\label{eq3}
	\arg\min \limits_{\nu^i \in \mathcal{A}_i}D_{KL} \bigg (\prod^M_{i=1}\nu^i \bigg \Arrowvert \mu \bigg )
\end{align}
for suit sets $\mathcal{A}_i$ with $i = 1, 2, \cdots, M$. Here we need to introduce the approximate probability measure $\nu$ given in (\ref{eq1}) which is equivalent to $\mu_0$, defined by
\begin{align}
	\frac{d\nu}{d\mu_0}(x) = \frac{1}{Z_{\nu}} \exp (-\Phi_{\nu}(x)).
\end{align}
A nature way for introducing an independence assumption is to assume that the potential $\Phi_{\nu}(x)$ can be decomposed as 
\begin{align}
	\exp (-\Phi_{\nu}(x)) = \prod^M_{i=1} \exp \bigg (\Phi^i_{\nu}(x_i) \bigg ),
\end{align}
where $x = (x_1, \cdots, x_M)$. Given these considerations, the following assumption is introduced.
\begin{assumption}\label{ass1}
	Let us introduce a reference probability measure
	\begin{align}
		\mu_r(dx) = \prod^M_{i=1}\mu^i_r(dx_j),
	\end{align}
	which is equivalent to the prior probability measure with the following relation:
	\begin{align}
		\frac{d\mu_0}{d\mu_r}(x) = \frac{1}{Z_0} \exp (-\Phi^0(x)).
	\end{align}
	For each $i=1, 2, \cdots, M$, there is a predefined continuous function $a_i(\epsilon, x_i)$, where $\epsilon$ is a positive number and $x_i \in \mathcal{H}_i$. Concerning these functions, we assume that $\mathbb{E}^{\mu^i_r}[a_i(\epsilon, \cdot)] < \infty$ where $\epsilon \in [0, \epsilon^i_0), i = 1, \cdots, M$ with $\epsilon^i_0$ is a small positive number we firstly defined. We also assume that the approximate probability measure $\nu$ is equivalent to the reference measure $\mu_r$ and that the Radon-Nikodym derivative of $\nu$ with respect to $\mu_r$ takes the following form
	\begin{align}
		\frac{d\nu}{d\mu_r}(x) = \frac{1}{Z_r} \exp \bigg (-\sum^M_{i=1}\Phi^r_i(x_i) \bigg ).
	\end{align}
	
\end{assumption}
Following Assumption $\ref{ass1}$, we know that the approximation measure can be decomposed as $\nu(dx) = \prod^M_{i=1}\nu^i(dx_i)$ with
\begin{align}\label{eq2}
	\frac{d\nu^i}{d\mu^i_r}(x) = \frac{1}{Z^i_r} \exp (-\Phi^r_i(x_i)).
\end{align}
Here denote that $Z^i_r = \mathbb{E}^{\mu^i_r}[\exp(-\Phi^r_i(x_i))]$ which ensures that $\nu_i$ is indeed a probability measure. For $i = 1, 2, \cdots, M$, let $\mathcal{Z}_i$ be a Hilbert space that embedded in $\mathcal{H}_i$, then we introduce
\begin{align*}
	R^1_i &= \bigg \lbrace \Phi^r_i \bigg | \sup \limits_{1/N \leqslant \lVert x_i \rVert \leqslant N} \Phi^r_i(x_i) < \infty, \quad \text{forall} \ N > 0 \bigg \rbrace \\
	R^2_i &= \bigg \lbrace \Phi^r_i \bigg | \int_{\mathcal{H}_i} \exp (-\Phi^r_i(x_i))\max(1, a_i(\epsilon, x_i))\mu^i_r(dx_i) < \infty, \quad \text{forall} \ \epsilon \in [0, \epsilon^i_0) \bigg \rbrace
\end{align*}
where $\epsilon^i_0$ and $a_i(\cdot, \cdot)$ are defined as in Assumption $\ref{ass1}$. With these preparations, we can define $\mathcal{A}_i, i = 1, 2, \cdots, M$ as follows:
\begin{align}
	\mathcal{A}_i = \left\{
	\begin{tabular}{l|l}
		\multirowcell{2}[0pt][l]{$\nu^i \in \mathcal{M}(\mathcal{H}_i)$} &
		\multirowcell{2}[0pt][l]{$\nu^i$ is equivalent to $\mu^i_r$ with (\ref{eq2}) holding true,\\
			and $\Phi^r_i \in R^1_i \bigcap R^2_i$} \\
		&
	\end{tabular}
	\right\}
\end{align}

Now, we are able to state the main theorem that yields practical iterative algorithms:
\begin{theorem}\label{the1}
	Assume that the approximate probability measure in problem $(\ref{eq3})$ satisfies Assumption $\ref{ass1}$, For $i = 1, 2, \cdots, M$, we denote $T^i_N = \lbrace x_i | 1/N \leqslant \lVert x_i \rVert_{\mathcal{Z}_i} \leqslant N \rbrace$, with $N$ being an arbitrary positive Constant. For each reference measure $\mu^i_r$, we assume that $\sup_N \mu^i_r(T^i_N)=1$. In addition, we assume 
	\begin{align}
		\sup \limits_{x_i \in T^i_N}\int_{\prod_{j \neq i}\mathcal{H}_j} \bigg (\Phi^0(x)+\Phi(x) \bigg )1_A(x)\prod \limits_{j \neq i}\nu^j(dx_j) < \infty
	\end{align}
	and 
	\begin{align}
		\int_{\mathcal{H}_i} \exp \bigg (-\int_{\prod_{j \neq i}\mathcal{H}_j}(\Phi^0(x)+\Phi(x))1_{A^c}(x)\prod \limits_{j \neq i}\nu^j(dx_j) \bigg )M_i(x)\mu^i_r(dx_i) < \infty,
	\end{align}
	where $A := \lbrace x | \Phi^0(x) + \Phi(x) \geqslant 0 \rbrace$, and $M_i := \max(1, a_i(\epsilon, x_i))$ with $i, j = 1, 2, \cdots, M$. Then the problem $(\ref{eq3})$ possess a solution $\nu = \prod^M_{i=1}\nu_i \in \mathcal{M}(\mathcal{H})$ with the following form
	\begin{align}
		\frac{d\nu}{d\mu_r}(x) \varpropto \exp \bigg (-\sum^M_{i=1}\Phi^r_i(x_i) \bigg),
	\end{align}
	where 
	\begin{align}
		\Phi^r_i(x_i) = \int_{\prod_{j \neq i}\mathcal{H}_j} \bigg (\Phi^0(x) + \Phi(x) \bigg )\prod \limits_{j \neq i}\nu^j(dx_j) + \text{Const}
	\end{align}
	and 
	\begin{align}
		\nu^i(dx_i) \varpropto \exp (-\Phi^r_i(x_i))\mu^i_r(dx_i).
	\end{align}
\end{theorem}

We should point out that Theorem $\ref{the1}$ and the definition of $R^1_i$ and $R^2_i, i = 1, 2, \cdots, M$ are slightly different from the statements given in \cite{jia2021variational}. 
The proof of Theorem 11 in \cite{jia2021variational} can be taken step by step to prove Theorem $\ref{the1}$. 
Actually, the version of Theorem $\ref{the1}$ can be regarded as a more appropriate amelioration of Theorem 11 in \cite{jia2021variational}, which can be verified more easily for practical problems.\\

\subsection{Assumption 1 and Theorem 15-16 in \cite{dashti2013bayesian}}

Here we provide the Assumption 1 and Theorem 15-16 in \cite{dashti2013bayesian} needed in proving Theorem 2.1, which is shown in Assumption $\ref{assump1app}$, Theorem $\ref{theorem1app}$ and $\ref{theorem2app}$, respectively.
Here we need to state that the symbols in \cite{dashti2013bayesian} have been changed to the notations in our paper for the assumption and theorems below.

\begin{assumption}\label{assump1app}
	(Assumption 1 in \cite{dashti2013bayesian})
	Let us denote $X = \mathcal{H}_u\times \mathbb{R}$, and assume $\Phi \in C(X\times \mathbb{R}^{N_d}; \mathbb{R})$.
	Assume further that there are functions $M_i:\mathbb{R}^{+}\times \mathbb{R}^{+} \rightarrow \mathbb{R}^{+}, i=1, 2$, monotonic non-decreasing separately in each argument, and with $M_2$ strictly positive, such that for all $v \in \mathcal{H}_u$, $\bm{d}, \bm{d}_1, \bm{d}_2 \in B_{\mathbb{R}^{N_d}}(0, r)$,
	\begin{align*}
		\Phi(v, \lambda; \bm{d}) &\geq -M_1(r, |\lambda|\lVert v \rVert_{\mathcal{H}_u}), \\
		\lvert \Phi(v, \lambda; \bm{d}_1) - \Phi(v, \lambda; \bm{d}_2)\rvert &\leq M_2(r, |\lambda|\lVert v \rVert_{\mathcal{H}_u})\lVert \bm{d}_1 - \bm{d}_2 \rVert_{\mathbb{R}^{N_d}}.
	\end{align*}
\end{assumption}

\begin{theorem}\label{theorem1app}
	(Theorem 15 in \cite{dashti2013bayesian})
	Let Assumption $\ref{assump1app}$ hold.
	Assume that $\mu_0(X) = 1$, and that $\mu_0(X \cap B) > 0$ for some bounded set $B$ in $X$.
	Assume additionally that, for every fixed $r>0$,
	\begin{align*}
		\exp(M_1(r, |\lambda|\lVert v \rVert_{\mathcal{H}_u})) \in L^1_{\mu_0}(X; \mathbb{R}),
	\end{align*}
	where $L_{\mu_0}^1(X;\mathbb{R})$ represents $X$-valued integrable functions under the measure $\mu_0$.
	Then for every $\bm{d} \in \mathbb{R}^{N_d}$, $Z_{\mu} = \int_X \exp(-\Phi(v, \lambda;\bm{d}))\mu^v_0(dv)\mu^{\lambda}_0(d\lambda)$ is positive and finite, and the probability measure $\mu$ is well defined, which is given by
	\begin{align*}
		\frac{d\mu}{d\mu_0}(v, \lambda) = \frac{1}{Z_{\mu}}\exp(-\Phi(v, \lambda; \bm{d})).
	\end{align*}
\end{theorem}

\begin{theorem}\label{theorem2app}
	(Theorem 16 in \cite{dashti2013bayesian})
	Let Assumption $\ref{assump1app}$ hold.
	Assume that $\mu_0(X) = 1$ and that $\mu_0(X \cap B) > 0$ for some bounded set $B$ in $X$.
	Assume additionally that, for every fixed $r > 0$,
	\begin{align*}
		\exp(M_1(r, |\lambda|\lVert v \rVert_{\mathcal{H}_u}))(1 + M_2(r, |\lambda|\lVert v \rVert_{\mathcal{H}_u})^2) \in L^1_{\mu_0}(X; \mathbb{R}).
	\end{align*}
	Then there is $C = C(r) > 0$ such that, for all $\bm{d}, \bm{d}^{\prime} \in B_Y(0, r)$
	\begin{align*}
		\sqrt{\frac{1}{2}\int \bigg (1 -  \sqrt{\frac{d\mu^{\prime}}{d\mu}} \bigg )^2d\mu} \leq C\lVert \bm{d} - \bm{d}^{\prime} \rVert_{\mathbb{R}^{N_d}},\\
	\end{align*}
	where measures $\mu^{\prime}$, $\mu$ represent the posterior measure according to $\bm{d}^{\prime}$, $\bm{d}$, respectively.
\end{theorem}

\subsection{Non-centered Gibbs sampling in Subsection 3.1 of the main text}

Here we provide the Non-centered Gibbs sampling method as pseudocode, and readers can seek more details in \cite{chen2018dimension}.

\begin{algorithm}
	\caption{Non-centered pCN within Gibbs}
	\label{alg B}
	\begin{algorithmic}[1]
		\STATE{Fix $\beta \in (0, 1]$, initialize $\lambda_0 = \bar{\lambda}$, $v_0 \sim \mu^v_0$ and set $k=0$ and maximum iteration number $N_{\max}$;}
		\REPEAT
		\STATE{Propose $\hat{v}_k=(1-\beta^2)v_k+\beta \zeta_k, \quad \zeta_k \sim \mathcal{N}(0, \mathcal{C}_0)$;}
		\STATE{Set $v_{k+1} = \hat{v}_k$ with probability
			\begin{align*}
				\min \bigg\lbrace 1, \exp \bigg(\Phi(v_k, \lambda_k) - \Phi(\hat{v}_k, \lambda_k) \bigg)\bigg\rbrace
			\end{align*}
			or else set $v_k = \hat{v}_k$;}
		\STATE{Propose $\hat{\lambda}_k \sim \mathcal{N}(\bar{\lambda}_k, \sigma_k)$, where $\mathcal{N}(\bar{\lambda}_k, \sigma_k)$ is the proposal measure of $\hat{\lambda}_k$;}
		\STATE{Set $\lambda_{k+1} = \hat{\lambda}_k$ with probability
			\begin{align*}
				\min \bigg\lbrace 1, \exp \bigg(\Phi(v_k, \lambda_k) - \Phi(v_k, \hat{\lambda}_k) \bigg)
				\frac{q(\hat{\lambda}_k, \lambda_k)\mu^{\lambda}_0(\hat{\lambda}_k)}{q(\lambda_k, \hat{\lambda}_k)\mu^{\lambda}_0(\lambda_k)} \bigg\rbrace
			\end{align*}
			or else set $\lambda_{k+1}=\lambda_k$;}
		\STATE{Set $k = k+1$;}
		\UNTIL{$k = N_{\max}$.}
		%\STATE{Return $\lbrace (\lambda_k, v_k)\rbrace^{N_{\max}}_0$}
	\end{algorithmic}
\end{algorithm}

Inspired by \cite{agapiou2014analysis}, the proposal measure of $\lambda$ in step 5 of Algorithm $\ref{alg B}$ can be specified based on our settings.
Because of the Gaussian noise and prior measures, the posterior measure of $\lambda$ can be calculated explicitly since it is also a Gaussian measure.
According to the discussion in Subsection 2.5, we provide the posterior measure $\mathcal{N}(\bar{\lambda}_k, \sigma_k)$of $\lambda_k$ as the proposal measure at each step, where
\begin{align*}
	\frac{1}{\sigma_k} = \tau\lVert Hv_k\rVert^{2} + \frac{1}{\sigma}, \quad 
	\bar{\lambda}_k = \sigma_k \bigg (\tau\langle d, Hv_k\rangle + \frac{\bar{\lambda}}{\sigma} \bigg ),
\end{align*}
and $\bar{\lambda}$, $\sigma$ is the mean and variance of the prior measure $\mu^{\lambda}_0$, respectively.

\section*{Acknowledgments}
This work was supported by the NSFC grant 12271428.
and the Major projects of the NSFC grants 12090020, 12090021 and 
the National Key R\&D program of the Ministry of Science and Technology of China grant 2020YFA0713403. 

\bibliographystyle{amsplain}
\bibliography{references}

\providecommand{\bysame}{\leavevmode\hbox to3em{\hrulefill}\thinspace}
\providecommand{\MR}{\relax\ifhmode\unskip\space\fi MR }
% \MRhref is called by the amsart/book/proc definition of \MR.
\providecommand{\MRhref}[2]{%
  \href{http://www.ams.org/mathscinet-getitem?mr=#1}{#2}
}
\providecommand{\href}[2]{#2}
\begin{thebibliography}{10}

\bibitem{agapiou2014analysis}
Sergios Agapiou, Johnathan~M. Bardsley, Omiros Papaspiliopoulos, and Andrew~M.
  Stuart, \emph{Analysis of the {G}ibbs sampler for hierarchical inverse
  problems}, SIAM/ASA Journal on Uncertainty Quantification \textbf{2} (2014),
  no.~1, 511--544.

\bibitem{Agapiou2013SPTA}
Sergios Agapiou, Stig Larsson, and Andrew~M. Stuart, \emph{Posterior
  contraction rates for the {B}ayesian approach to linear ill-posed inverse
  problems}, Stochastic Processes and their Applications \textbf{123} (2013),
  no.~10, 3828--3860.

\bibitem{Agapiou2017SS}
Sergios Agapiou, Omiros Papaspiliopoulos, Daniel Sanz-Alonso, and Andrew~M.
  Stuart, \emph{Importance sampling: {I}ntrinsic dimension and computational
  cost}, Statistical Science \textbf{32} (2017), no.~3, 405–431.

\bibitem{arridge2019solving}
Simon Arridge, Peter Maass, Ozan {\"O}ktem, and Carola-Bibiane Sch{\"o}nlieb,
  \emph{Solving inverse problems using data-driven models}, Acta Numerica
  \textbf{28} (2019), 1--174.

\bibitem{Bao_2010}
Gang Bao, Shui-Nee Chow, Peijun Li, and Haomin Zhou, \emph{Numerical solution
  of an inverse medium scattering problem with a stochastic source}, Inverse
  Problems \textbf{26} (2010), no.~7, 074014.

\bibitem{Bao_2015}
Gang Bao, Peijun Li, Junshan Lin, and Faouzi Triki, \emph{Inverse scattering
  problems with multi-frequencies}, Inverse Problems \textbf{31} (2015), no.~9,
  093001.

\bibitem{Beskos2017JCP}
Alexandros Beskos, Mark Girolami, Shiwei Lan, Patrick~E. Farrell, and Andrew~M.
  Stuart, \emph{Geometric {MCMC} for infinite-dimensional inverse problems},
  Journal of Computational Physics \textbf{335} (2017), 327--351.

\bibitem{Beskos2015SC}
Alexandros Beskos, Ajay Jasra, Ege~A. Muzaffer, and Andrew~M. Stuart,
  \emph{Sequential {M}onte {C}arlo methds for {B}ayesian elliptic inverse
  problems}, Statistics and Computing \textbf{25} (2015), no.~4, 727--737.

\bibitem{bishop2006pattern}
Christopher~M. Bishop and Nasser~M. Nasrabadi, \emph{Pattern {R}ecognition and
  {M}achine {L}earning}, Springer, New York, 2006.

\bibitem{blei2017variational}
David~M. Blei, Alp Kucukelbir, and Jon~D. McAuliffe, \emph{Variational
  inference: {A} review for statisticians}, Journal of the American Statistical
  Association \textbf{112} (2017), no.~518, 859--877.

\bibitem{bui2012analysis}
Tan Bui-Thanh and Omar Ghattas, \emph{Analysis of the {H}essian for inverse
  scattering problems: {I. I}nverse shape scattering of acoustic waves},
  Inverse Problems \textbf{28} (2012), no.~5, 055001.

\bibitem{bui2013computational}
Tan Bui-Thanh, Omar Ghattas, James Martin, and Georg Stadler, \emph{A
  computational framework for infinite-dimensional {B}ayesian inverse problems
  {P}art {I}: {T}he linearized case, with application to global seismic
  inversion}, SIAM Journal on Scientific Computing \textbf{35} (2013), no.~6,
  A2494--A2523.

\bibitem{Thanh2016IPI}
Tan Bui-Thanh and Quoc~P. Nguyen, \emph{{FEM}-based discretization-invariant
  {MCMC} methods for {PDE}-constrained {B}ayesian inverse problems}, Inverse
  Problems \& Imaging \textbf{10} (2016), no.~4, 943–975.

\bibitem{calvetti2018iterative}
Daniela Calvetti, Matthew Dunlop, Erkki Somersalo, and Andrew~M. Stuart,
  \emph{Iterative updating of model error for {B}ayesian inversion}, Inverse
  Problems \textbf{34} (2018), no.~2, 025008.

\bibitem{calvetti2009conditionally}
Daniela Calvetti, Harri Hakula, Sampsa Pursiainen, and Erkki Somersalo,
  \emph{Conditionally {G}aussian hypermodels for cerebral source localization},
  SIAM Journal on Imaging Sciences \textbf{2} (2009), no.~3, 879--909.

\bibitem{Calvetti2008hypermodels}
Daniela Calvetti and Erkki Somersalo, \emph{Hypermodels in the {B}ayesian
  imaging framework}, Inverse Problems \textbf{24} (2008), no.~3, 034013.

\bibitem{Chen2021SISC}
Peng Chen and Omar Ghattas, \emph{Stein variational reduced basis {B}ayesian
  inversion}, SIAM Journal on Scientific Computing \textbf{43} (2021), no.~2,
  A1163--A1193.

\bibitem{Chen2019NIPS}
Peng Chen, Keyi Wu, Joshua Chen, Tom O'Leary-Roseberry, and Omar Ghattas,
  \emph{Projected {S}tein variational {N}ewton: {A} fast and scalable
  {B}ayesian inference method in high dimensions}, Advances in Neural
  Information Processing Systems \textbf{32} (2019), 15104--15113.

\bibitem{chen2018dimension}
Victor Chen, Matthew~M. Dunlop, Omiros Papaspiliopoulos, and Andrew~M. Stuart,
  \emph{Dimension-robust {MCMC} in {B}ayesian inverse problems}, arXiv preprint
  arXiv:1803.03344 (2018).

\bibitem{Cotter2009IP}
Simon~L. Cotter, Massoumeh Dashti, James~C. Robinson, and Andrew~M. Stuart,
  \emph{Bayesian inverse problems for functions and applications to fluid
  mechanics}, Inverse Problems \textbf{25} (2009), no.~11, 115008.

\bibitem{cotter2013}
Simon~L. Cotter, Gareth~O. Roberts, Andrew~M. Stuart, and David White,
  \emph{{MCMC} methods for functions: {M}odifying old algorithms to make them
  faster}, Statistical Science \textbf{28} (2013), no.~3, 424--446.

\bibitem{da2006introduction}
Giuseppe Da~Prato, \emph{An {I}ntroduction to {I}nfinite-{D}imensional
  {A}nalysis}, Springer-Verlag, Berlin, 2006.

\bibitem{Prato2014book}
Giuseppe Da~Prato and Jerzy Zabczyk, \emph{Stochastic {E}quations in {I}nfinite
  {D}imensions}, Cambridge University Press, Cambridge, 2014.

\bibitem{da2014stochastic}
\bysame, \emph{Stochastic {E}quations in {I}nfinite {D}imensions}, Cambridge
  university press, Cambridge, 2014.

\bibitem{Dashti2013IP}
Masoumeh Dashti, Kody~J.H. Law, Andrew~M. Stuart, and Jochen Voss, \emph{{MAP}
  estimators and their consistency in {B}ayesian nonparametric inverse
  problems}, Inverse Problems \textbf{29} (2013), no.~9, 095017.

\bibitem{dashti2013bayesian}
Masoumeh Dashti and Andrew~M. Stuart, \emph{The {B}ayesian {A}pproach to
  {I}nverse {P}roblems}, Handbook of uncertainty quantification, Springer,
  2017, pp.~311--428.

\bibitem{Dunlop2020SMAIJCM}
Matthew~M. Dunlop, Tapio Helin, and Andrew~M. Stuart, \emph{Hyperparameter
  estimation in {B}ayesian {MAP} estimation: {P}arameterizations and
  consistency}, SMAI Journal of Computational Mathematics \textbf{6} (2020),
  69--100.

\bibitem{dunlop2017hierarchical}
Matthew~M. Dunlop, Marco~A. Iglesias, and Andrew~M. Stuart, \emph{Hierarchical
  {B}ayesian level set inversion}, Statistics and Computing \textbf{27} (2017),
  no.~6, 1555--1584.

\bibitem{engl1996regularization}
Heinz~W. Engl, Martin Hanke, and Andreas Neubauer, \emph{Regularization of
  inverse problems}, Kluwer Academic Publishers Group, Dordrecht, 1996.

\bibitem{feng2018adaptive}
Zhe Feng and Jinglai Li, \emph{An adaptive independence sampler {MCMC}
  algorithm for {B}ayesian inferences of functions}, SIAM Journal on Scientific
  Computing \textbf{40} (2018), no.~3, A1301--A1321.

\bibitem{Fichtner2011Book}
Andreas Fichtner, \emph{Full {S}eismic {W}aveform {M}odelling and {I}nversion},
  Springer, Heidelberg, 2010.

\bibitem{Ghattas2021ActaNumerica}
Omar Ghattas and Karen Willcox, \emph{Learning physics-based models from data:
  {P}erspectives from inverse problems and model reduction}, Acta Numerica
  \textbf{30} (2021), 445--554.

\bibitem{Guha2015JCP}
Nilabja Guha, Xiaoqing Wu, Yalchin Efendiev, Bangti Jin, and Bani~K. Mallick,
  \emph{A variational {B}ayesian approach for inverse problems with skew-t
  error distribution}, Journal of Computational Physics \textbf{301} (2015),
  377--393.

\bibitem{hinze2008optimization}
Michael Hinze, Ren{\'e} Pinnau, Michael Ulbrich, and Stefan Ulbrich,
  \emph{Optimization with {PDE} {C}onstraints}, Springer, New York, 2008.

\bibitem{jia2021stein}
Junxiong Jia, Peijun Li, and Deyu Meng, \emph{Stein variational gradient
  descent on infinite-dimensional space and applications to statistical inverse
  problems}, SIAM Journal on Numerical Analysis \textbf{60} (2022), no.~4,
  2225--2252.

\bibitem{jia2019recursive}
Junxiong Jia, Bangyu Wu, Jigen Peng, and Jinghuai Gao, \emph{Recursive
  linearization method for inverse medium scattering problems with complex
  mixture {G}aussian error learning}, Inverse Problems \textbf{35} (2019),
  no.~7, 075003.

\bibitem{Jia2022VINet}
Junxiong Jia, Yanni Wu, Peijun Li, and Deyu Meng, \emph{Variational inverting
  network for statistical inverse problems of partial differential equations},
  arXiv:2201.00498 (2022), 1--46.

\bibitem{jia2021variational}
Junxiong Jia, Qian Zhao, Zongben Xu, Deyu Meng, and Yee Leung,
  \emph{Variational {B}ayes' method for functions with applications to some
  inverse problems}, SIAM Journal on Scientific Computing \textbf{43} (2021),
  no.~1, A355--A383.

\bibitem{Jin2012JCP}
Bangti Jin, \emph{A variational {B}ayesian method to inverse problems with
  implusive noise}, Journal of Computational Physics \textbf{231} (2012),
  no.~2, 423--435.

\bibitem{jin2010hierarchical}
Bangti Jin and Jun Zou, \emph{Hierarchical {B}ayesian inference for ill-posed
  problems via variational method}, Journal of Computational Physics
  \textbf{229} (2010), no.~19, 7317--7343.

\bibitem{kaipio2006statistical}
Jari Kaipio and Erkki Somersalo, \emph{Statistical and {C}omputational
  {I}nverse {P}roblems}, Springer-Verlag, New York, 2005.

\bibitem{Kirsch2011Book}
Andreas Kirsch, \emph{An {I}ntroduction to the {M}athematical {T}heory of
  {I}nverse {P}roblems}, Springer, New York, 2011.

\bibitem{Lassas2004IP}
Matti Lassas and Samuli Siltanen, \emph{Can one use total variation prior for
  edge-preserving {B}ayesian inversion?}, Inverse Problems \textbf{20} (2004),
  no.~5, 1537--1563.

\bibitem{papaspiliopoulos2008stability}
Omiros Papaspiliopoulos and Gareth Roberts, \emph{Stability of the {G}ibbs
  sampler for {B}ayesian hierarchical models}, The Annals of Statistics
  \textbf{36} (2008), no.~1, 95--117.

\bibitem{papaspiliopoulos2003non}
Omiros Papaspiliopoulos, Gareth~O. Roberts, and Martin Sk{\"o}ld,
  \emph{Non-centered parameterisations for hierarchical models and data
  augmentation}, Bayesian Statistics 7: Proceedings of the Seventh Valencia
  International Meeting (New York), Oxford University Press, 2003,
  pp.~307--326.

\bibitem{papaspiliopoulos2007general}
\bysame, \emph{A general framework for the parametrization of hierarchical
  models}, Statistical Science \textbf{22} (2007), no.~1, 59--73.

\bibitem{petra2014computational}
Noemi Petra, James Martin, Georg Stadler, and Omar Ghattas, \emph{A
  computational framework for infinite-dimensional {B}ayesian inverse problems,
  {P}art {II}: {S}tochastic {N}ewton {MCMC} with application to ice sheet flow
  inverse problems}, SIAM Journal on Scientific Computing \textbf{36} (2014),
  no.~4, A1525--A1555.

\bibitem{Pillai2014SPDE}
Natesh~S. Pillai, Andrew~M. Stuart, and Alexandre~H. Thi{\'e}ry, \emph{Noisy
  gradient flow from a random walk in {H}ilbert space}, Stochastic Partial
  Differential Equations: Analysis and Computations \textbf{2} (2014), no.~2,
  196--232.

\bibitem{pinski2015algorithms}
Frank~J. Pinski, Gideon Simpson, Andrew~M. Stuart, and Hendrik Weber,
  \emph{Algorithms for {K}ullback--{L}eibler approximation of probability
  measures in infinite dimensions}, SIAM Journal on Scientific Computing
  \textbf{37} (2015), no.~6, A2733--A2757.

\bibitem{Pinski2015SIAMMA}
\bysame, \emph{Kullback-{L}eibler approximation for probability measures on
  infinite dimensional space}, SIAM Journal on Mathematical Analysis
  \textbf{47} (2015), no.~6, 4091--4122.

\bibitem{reed2012methods}
Michael Reed and Barry Simon, \emph{Methods of {M}odern {M}athematical
  {P}hysics. {I}. {F}unctional {A}nalysis}, Academic Press, New York-London,
  1978.

\bibitem{saibaba2016randomized}
Arvind~K. Saibaba, Jonghyun Lee, and Peter~K. Kitanidis, \emph{Randomized
  algorithms for generalized {H}ermitian eigenvalue problems with application
  to computing {K}arhunen--{L}o{\`e}ve expansion}, Numerical Linear Algebra
  with Applications \textbf{23} (2016), no.~2, 314--339.

\bibitem{Stuart2010ActaNumerica}
Andrew~M. Stuart, \emph{Inverse problems: {A} {B}ayesian perspective}, Acta
  Numerica \textbf{19} (2010), 451--559.

\bibitem{tran2015variational}
Dustin Tran, Rajesh Ranganath, and David~M. Blei, \emph{The variational
  {G}aussian process}, arXiv preprint arXiv:1511.06499 (2015).

\bibitem{villa2021hippylib}
Umberto Villa, Noemi Petra, and Omar Ghattas, \emph{h{IPPY}lib: {A}n extensible
  software framework for large-scale inverse problems governed by {PDE}s:
  {P}art {I}: {D}eterministic inversion and linearized {B}ayesian inference},
  ACM Transactions on Mathematical Software (TOMS) \textbf{47} (2021), no.~2,
  1--34.

\bibitem{wang2019frequentist}
Yixin Wang and David~M Blei, \emph{Frequentist consistency of variational
  {B}ayes}, Journal of the American Statistical Association \textbf{114}
  (2019), no.~527, 1147--1161.

\bibitem{weglein2003inverse}
Arthur~B. Weglein, Fernanda~V. Ara{\'u}jo, Paulo~M. Carvalho, Robert~H. Stolt,
  Kenneth~H. Matson, Richard~T. Coates, Dennis Corrigan, Douglas~J. Foster,
  Simon~A. Shaw, and Haiyan Zhang, \emph{Inverse scattering series and seismic
  exploration}, Inverse Problems \textbf{19} (2003), no.~6, R27--R83.

\bibitem{wong2011exact}
Yau~Shu Wong and Guangrui Li, \emph{Exact finite difference schemes for solving
  {H}elmholtz equation at any wavenumber}, International Journal of Numerical
  Analysis and Modeling, Series B \textbf{2} (2011), no.~1, 91--108.

\bibitem{zhang2018advances}
Cheng Zhang, Judith B{\"u}tepage, Hedvig Kjellstr{\"o}m, and Stephan Mandt,
  \emph{Advances in variational inference}, IEEE Transactions on Pattern
  Analysis and Machine Intelligence \textbf{41} (2018), no.~8, 2008--2026.

\bibitem{zhang2020convergence}
Fengshuo Zhang and Chao Gao, \emph{Convergence rates of variational posterior
  distributions}, The Annals of Statistics \textbf{48} (2020), no.~4,
  2180--2207.

\bibitem{zhou2020bayesian}
Qingping Zhou, Tengchao Yu, Xiaoqun Zhang, and Jinglai Li, \emph{Bayesian
  inference and uncertainty quantification for medical image reconstruction
  with poisson data}, SIAM Journal on Imaging Sciences \textbf{13} (2020),
  no.~1, 29--52.

\end{thebibliography}
\end{document}